\documentclass[11pt]{article}

\usepackage{microtype}
\usepackage{amsmath, amssymb, amsthm}
\usepackage{thmtools,thm-restate}
\usepackage[margin=1.0in]{geometry}
\usepackage{enumerate}
\usepackage{hyperref}
\usepackage{graphicx} 
    \usepackage[round]{natbib}
    \usepackage[font=small]{caption}
    \hypersetup{
        pdftitle={}, 
        pdfauthor={}, 
        linktocpage=true,
        colorlinks=true, 
        linkcolor=black, 
        citecolor=black, 
        urlcolor=blue 
    }
\usepackage{xcolor}
\usepackage{wrapfig}
\usepackage{algorithm}
\usepackage[noend]{algpseudocode}
\usepackage{etoolbox}
\usepackage{dsfont}

\theoremstyle{plain}
\newtheorem{theorem}{Theorem}
\newtheorem{lemma}{Lemma}[section]
\newtheorem{proposition}[lemma]{Proposition}

\newtheorem{corollary}[lemma]{Corollary}

\theoremstyle{definition}
\newtheorem{definition}[theorem]{Definition}

\newtheorem{example}[theorem]{Example}

\let\originalleft\left
\let\originalright\right
\renewcommand{\left}{\mathopen{}\mathclose\bgroup\originalleft}
\renewcommand{\right}{\aftergroup\egroup\originalright}

\newcommand{\LLR}{\operatorname{LLR}}

\newcommand{\CLLR}{\operatorname{cLLR}}

\newcommand{\NCLLR}{\operatorname{ncLLR}}
\newcommand{\SC}{{\rm SC}}

\newcommand{\Lap}{\mathrm{Lap}}
\newcommand{\var}{\text{var}}
\newcommand{\loc}{\rm loc}
\newcommand{\KL}{\text{KL}}
\newcommand{\TV}{\text{\rm TV}}
\newcommand{\Hell}{\text{H}}
\newcommand{\thetarange}{\Phi}
\newcommand{\secondstephigh}{\mathtt{Count}}
\newcommand{\secondsteplow}{{\rm nCLLRE}}
\newcommand{\thirdmoment}{\zeta}
\newcommand{\fourthmoment}{\zeta}
\newcommand{\nB}{D_2}
\newcommand{\Qest}{\mathcal{F}}
\newcommand{\Qtest}{\mathcal{F}^{\text{\rm test}}}
\newcommand{\Qesteps}{\Qest_\epsilon}
\newcommand{\Qtesteps}{\mathcal{F}_{\epsilon}^{\text{\rm test}}}

\newcommand{\LinfoP}{J_{\TV, P}^{-1}}
\newcommand{\initial}{\mathcal{M}_{\fourthmoment, C}}
\newcommand{\berryesseen}{\nu}
\newcommand{\Ahigh}{\mathcal{A}_{{\rm high}, C, \fourthmoment}}
\newcommand{\anyAhigh}{\mathcal{A}_{{\rm high}}}
\newcommand{\anyinitial}{\mathcal{M}}
\newcommand{\Alow}{\mathcal{A}_{{\rm low}, C, \fourthmoment}}
\newcommand{\anyAlow}{\mathcal{A}_{{\rm low}}}
\newcommand{\est}{\hat{\theta}}
\newcommand{\epsn}{\epsilon_n}
\newcommand{\deltan}{\delta_n}
\newcommand{\error}{\mathcal{E}}

\renewcommand{\epsilon}{\varepsilon}
\newcommand{\eps}{\epsilon}

\newcommand{\ierror}[3]{\mathfrak{R}_{#2}(#3, #1)}
\newcommand{\errorloc}[3]{\mathcal{E}^{\rm loc}_{#1}(#2, #3)}
\newcommand{\errorlocall}[5]{\mathcal{E}^{\rm loc}_{#1}(#2, #3, #4, #5)}
\newcommand{\MOC}[3]{\omega_{#1}(#2, #3)}
\newcommand{\MOCall}[5]{\omega_{#1}(#2, #3, #4, #5)}

\newcommand{\MOCallt}[5]{\omega_{#1}(#2, #3, #4, #5)}
\newcommand{\QSC}{{\rm SC}_{\Qtest}}
\newcommand{\ntarget}{m}

\newcommand{\cP}{\mathcal{P}}

\newcommand{\set}[1]{\left\{ {#1} \right\}}
\newcommand{\paren}[1]{\left( {#1} \right)}

\newcommand{\AS}[1]{}
\newcommand{\AM}[1]{}
\newcommand{\JU}[1]{}
\newcommand{\as}[1]{}
\newcommand{\ju}[1]{}
\newcommand{\am}[1]{}

\newcommand{\mypar}[1]{\smallskip\noindent\emph{#1}}

\begin{document}

\title{Instance-Optimal Differentially Private Estimation}

    \author{Audra McMillan\thanks{\texttt{audra\_mcmillan@apple.com}. Apple, Inc.} \and Adam Smith\thanks{\texttt{ads22@cs.bu.edu}. Department of Computer Science, Boston University.} \and Jonathan Ullman\thanks{\texttt{jullman@ccs.neu.edu} Khoury College of Computer Sciences, Northeastern University.}}

\date{}

\maketitle
\pagenumbering{roman}
\thispagestyle{empty}

\begin{abstract}
In this work, we study local minimax convergence estimation rates subject to $\epsilon$-differential privacy. Unlike worst-case rates, which may be conservative, algorithms that are locally minimax optimal must adapt to \emph{easy} instances of the problem.  We construct locally minimax differentially private estimators for one-parameter exponential families and estimating the tail rate of a distribution.  In these cases, we show that optimal algorithms for simple hypothesis testing, namely the recent optimal private testers of \citet{Canonne:2019}, directly inform the design of locally minimax estimation algorithms.
\end{abstract}

{\footnotesize
\setcounter{tocdepth}{2}
}

\pagenumbering{arabic}
\section{Introduction}
While the primary goal of statistical inference is to reveal properties of a population, many statistical estimators also reveal a significant amount of information about their sample, and this becomes a serious problem when the sample contains sensitive private information about individuals.  As a response, \emph{differential privacy}~\citep{DworkMNS06} has emerged as a strong formal criterion for a statistical procedure to protect individual privacy. Differentially private algorithms are deployed in a variety of settings, from the public data products for the 2020 US decennial census to Google's keyboard prediction models~\citep{Google-FL-blog} and Apple device analytics~\citep{Apple17}.

Differential privacy is a constraint on an estimator that requires the distribution of the estimator's outputs to be insensitive to changing a single individual's data, and it offers a strong semantic guarantee that no attacker can infer much more about any individual than they could have inferred had that individual's data never been collected~\citep{KasiviswanathanS08}.  This semantic guarantee does not rely on any assumptions about the adversary's background knowledge and capabilities. In contrast, alternative approaches to protecting privacy have often been undermined by underestimating the abilities of the attacker.  Although differential privacy is a constraint that significantly limits inference with small sample sizes,
most statistical tasks are compatible with differential privacy given a large enough sample.  

There is now a large body of work on differentially private estimation, which includes \emph{minimax optimal} differentially private estimators for many estimation tasks (e.g.~\citet{DuchiJW13,BunUV14,DworkSSUV15}).  A minimax optimal estimator is one that minimizes the maximum loss over all distributions in some family.  However, even a minimax optimal estimator can be undesirable in practice because it might achieve the same error on all distributions, even if some distributions are \emph{easier} than the worst-case distributions.

A more refined guarantee is called \emph{local minimax optimality}.  While the actual definition is necessarily subtle, intuitively a local minimax optimal estimator simultaneously has the best possible error on every distribution, which means the error must automatically adapt to distributions that are easier.  To illustrate this with a simple example from the non-private setting, suppose we are given a sample of size $n$ from a Bernoulli distribution $P = \mathrm{Ber}(\theta)$ and want to estimate the parameter $\theta \in (0,1)$. The empirical mean has mean-squared error $\theta(1-\theta)/n \leq 1/4n$.  No estimator can have better error than $1/4n$ on all Bernoulli distributions (roughly because samples of size $n$ from $\mathrm{Ber}(\frac12 - \frac 1{2\sqrt{n}})$ and $\mathrm{Ber}(\frac12 + \frac 1{2\sqrt{n}})$ are hard to reliably distinguish), so the empirical mean is (globally) minimax optimal. But it is also \textit{locally} minimax optimal because it adapts automatically to the ``easy" values of $\theta$ close to 0 or 1. In contrast, a hypothetical estimator that had mean-squared error exactly $1 /{4n}$ for all values of $\theta$ would be minimax optimal but \textit{not} locally minimax optimal.

We study the design of locally minimax differentially private estimators. We provide:
\begin{itemize}
    \item A connection between locally minimax differentially private estimators and differentially private simple hypothesis testing: namely, the local estimation rate for the class of differentially private estimators is given by inverting the sample complexity of the optimal differentially private hypothesis test.
    Such a connection was previously shown in the non-private setting \citep{Donoho:1991} and in the more restrictive locally differentially private setting \citep{Duchi:2018,Rohde:2019}.
    
    \item Locally minimax differentially private estimators for one-parameter exponential families. In the small $\epsilon$ (that is, $\epsilon=O(1/\sqrt{n})$) regime, our estimator is directly informed by the locally\footnote{\textit{Local differential privacy} refers to the model of differential privacy where data subjects randomize their own data points before sending it to the server. It is a more restricted model than central differential privacy, the main privacy model of interest in this paper. See Section~\ref{sec:related} } DP estimator introduced in \citep{Duchi:2018}, who show the locally differentially private version of this estimator is locally minimax optimal. For larger $\epsilon$, our estimators are directly informed by the structure of the approximately optimal differentially private simple hypothesis tests of \citet{Canonne:2019}.  In particular, our estimator critically relies on a refined version of their optimal test, introduced in this work, with additional properties. 
    
    \item A general approach to nonparametric estimation of one-dimensional functionals. We illustrate its application to estimating tail decay rates. 
\end{itemize}

\mypar{Simple Hypothesis Testing and Local Estimation Rates (Section~\ref{sec:testing}).} 
As shown by \citet{Donoho:1991}, local minimax estimation is closely related to simple hypothesis testing.  The connection was originally developed in the non-private setting, but applies more generally to any restricted estimation setting.  Suppose we have a population $P$ from some family $\mathcal{P}$ and want to estimate a statistic $\theta(P)$.  We have a sample $X \sim P^n$ and an estimator $\hat\theta(X)$. Given two distribution $P,Q \in \mathcal{P}$ we can use $\hat\theta$ as the basis for a simple hypothesis test that distinguishes $P$ and $Q$ by looking at $\hat\theta(X)$ and checking if it's closer to $\theta(P)$ or $\theta(Q)$, and this approach will be a successful hypothesis test if and only if $\hat\theta$ has sufficiently small error for both populations $P$ and $Q$.  See Figure~\ref{graphofconnection} for a pictorial representation of how $\hat\theta$ can distinguish $P$ and $Q$.  Some pairs $P,Q \in \mathcal{P}$ cannot be reliably distinguished with a sample of size $n$ and some can.  
Informally, we say that $\hat\theta$ is \textit{locally minimax optimal} if it can be used in this fashion to obtain a hypothesis test for \emph{any} pair of distributions in $\mathcal{P}$ that can be distinguished using $n$ samples. This formulation makes it clear that lower bounds for simple hypothesis testing automatically give lower bounds on the local estimation rate.  Although hypothesis tests for specific pairs of distributions do not inherently yield optimal estimators, the structure of optimal tests can guide the construction of locally minimax estimators.  We show that this process of converting hypothesis testing results into estimation rates can be carried out in the private setting, and instantiate it for several univariate estimation problems.  

\begin{figure}
	\centering
        \includegraphics[scale=0.4]{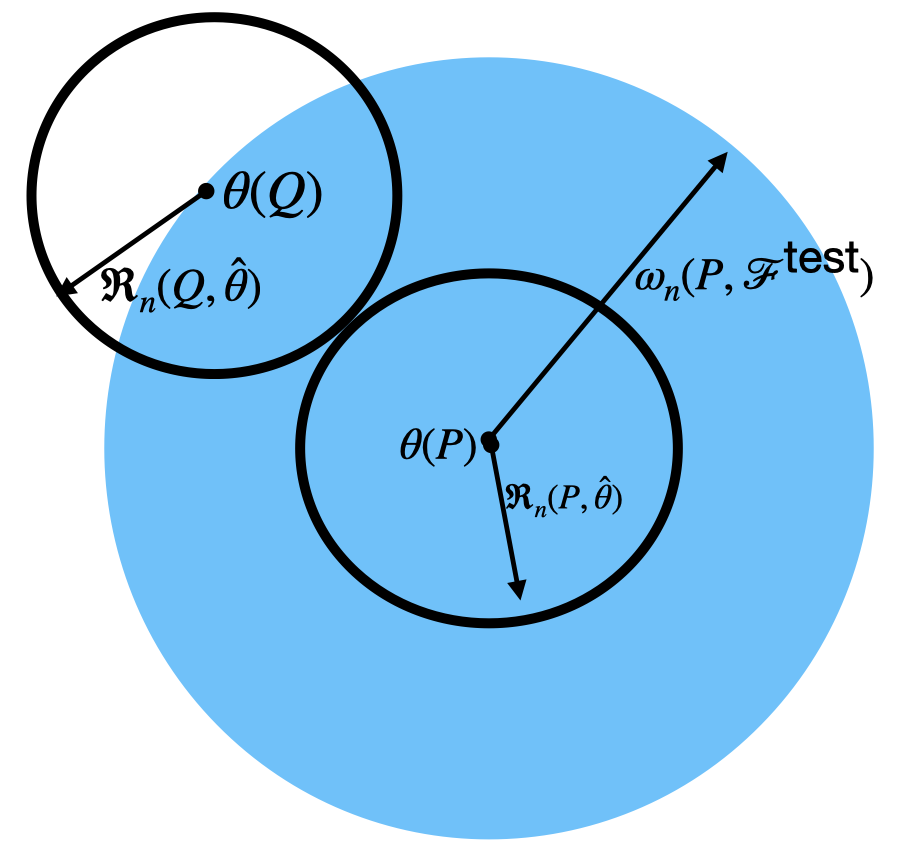}
        \caption{Graphical representation of the connection between simple hypothesis testing and local estimation rates.}
        \label{graphofconnection}
\end{figure}

In the non-private setting, the sample complexity of distinguishing between two distributions $P$ and $Q$ is $\Theta(1/H^2(P,Q))$, where $H(P,Q)$ is the Hellinger distance, and hence the Hellinger distance is the relevant distance when characterising local estimation rates in the non-private setting.
\cite{Duchi:2018} showed that in the \textit{local DP} setting, the sample complexity is $\Theta(1/\epsilon^2 \TV(P,Q)^2)$, where $\TV(P,Q)$ is the total variation distance. \cite{Canonne:2019} showed that the sample complexity in the central DP setting is more nuanced. However, in this work, we show that it has a simple form in the high privacy regime. When $\eps=O(1/\sqrt{n})$, the sample complexity is $\Theta(1/\epsilon \TV(P,Q))$, the square root of the sample complexity in the local DP setting.
This move from the Hellinger distance to the total variation distance has implications for how well one can expect estimation algorithms to adapt to problem-specific difficulty. For example, the fact that non-private algorithms for Bernoulli parameter estimation can adapt to problem-specific difficulty, while local DP algorithms and central DP algorithms in the high privacy regime cannot, is a direct consequence of the fact $H({\rm Bernoulli}(\theta),{\rm Bernoulli}(\theta+\alpha))$ is a function of $\theta$, while $\TV({\rm Bernoulli}(\theta),{\rm Bernoulli}(\theta+\alpha))$ is independent of $\theta$.  We discuss this further in Section~\ref{generalhighprivacy}.

While we show that this framework is suitable for univariate estimation problems, it is not generally suitable for estimating multivariate statistics, as this simple-hypothesis-testing formulation does not fully capture private estimation for multivariate statistics.  In particular,  one provably cannot achieve the local estimation rate even for simple tasks like estimating the mean of a multivariate Gaussian with identity covariance~\citep{BunUV14,DworkSSUV15} since the lower bounds on hypothesis testing and estimation depend on the dimension in different ways\footnote{For example, the sample complexity for privately distinguishing between two Gaussian distributions with identity covariance at total variation distance $\alpha$ is 
$O(\sqrt{d})$ (for constant $\alpha$ and $\eps$)
%$O(d^{1/2}/\alpha^2+d^{1/3}/\alpha^{4/3}\epsilon^{2/3}+1/\alpha\epsilon)$
(see, e.g., \citet{Narayanan2022PrivateHH}), while the sample complexity required for privately estimating a Gaussian with identity covariance to within total variation distance $\alpha$ is 
$\Omega(d / \log(d))$
%$\Omega(d/\alpha\epsilon\log(d))$ 
\citep{KamathLSU19}.
}. 
We leave it to future work to develop a suitable notion of local minimax estimation for higher-dimensional problems.

\mypar{Exponential Families (Section~\ref{expfams})} We give a DP estimator for one-parameter exponential families that uniformly achieves the private, locally minimax-optimal error under suitable regularity conditions. The estimator works (and is optimal) for any setting of $\eps=O(1)$. We identify two qualitatively different regimes: the ``low privacy'' regime, $\eps = \Omega(1/\sqrt{n})$, and the ``high privacy'' regime, $\eps = O(1/\sqrt{n})$.  In the low-privacy regime, privacy can be achieved without increasing the asymptotic error of the estimator, while in the high-privacy regime, the error due to privacy dominates the sampling error.  A weaker version of the low-privacy result appears in \citet{Smith11}; however, that result matches the best nonprivate error only for $\eps = \omega(1/\sqrt[4]{n})$, instead of $\eps = \Omega(1/\sqrt{n})$.

In both regimes,  our algorithm first uses a subroutine of \citet{Karwa:2018} to identify a rough, initial approximation $\hat\theta_0$ to the true parameter. 
The next step is to compute and release a (noisy) test statistic $\hat f = f_{n,\eps,\hat\theta_0}(X)$. In the low privacy regime, this statistic is, very roughly, 
the same one that arises in
the private simple hypothesis test of \citet{Canonne:2019} for distinguishing $\hat\theta_0$ from $\hat\theta_0 + \alpha$, where $\alpha$ is roughly the local minimax error at $\hat\theta_0$. 
The exact form of the statistic is more subtle, and relies on a linearization of the model in a neighborhood of $\hat\theta_0$.
The statistic takes a simpler form in the high privacy regime.
Finally, we take the  estimate $\hat \theta$ to be the unique solution to $\hat f = \mathbb{E}_{X\sim P_{\hat \theta}^n}\paren{f_{n,\eps,\hat\theta_0}(X)}$, which finds the value $ \hat \theta$ for which the expected value of the test statistic matches the observation $\hat f$. The key in both regimes is to prove that $f_{n,\eps,\hat\theta_0}$ is a good test statistic not only for distinguishing $\hat \theta_0$ from $\hat \theta_0+\alpha$, but for distinguishing all pairs of the form $(\theta, \theta+\alpha)$ for $\theta$ in a neighborhood of $\hat \theta_0$.

Our approach parallels that of \citet{Duchi:2018}, who developed a similar result for the more restricted setting of \textit{locally differentially private} algorithms. Indeed, in the high privacy regime, the structure of the optimal estimator is very similar to theirs, and the asymptotic sample complexity of the optimal (central-model) private estimator is exactly the square root of that of the optimal locally-private estimator. In the low-privacy regime, however, the estimators' structure differs. In all cases, the lower bound techniques are quite different.

\mypar{Estimation of More General Functionals (Section~\ref{sec:noparametric}).} In addition to parametric estimation problems, our framework applies to the estimation of one-dimensional functionals $T(P)$ of distributions, even when the functional of interest does not completely describes the underlying data distribution $P$. We discuss general approaches to such problems and explore the estimation of tail decay rates in real-valued distributions, an example also studied in depth by \citet{Donoho:1991}.

There are several natural meanings to local optimality in such a setting. 
Following \citet{Donoho:1991}, we seek estimation algorithms that, for each $\theta$, achieves error rate $\mathfrak{R}(\theta)$ for all distributions $P$ in the subfamily $\set{P \in \cP: T(P)=\theta}$, where $\cP$ is the family of distributions to which the true population is assumed to belong and $\mathfrak{R}(\theta)$ is the optimal local estimation rate for at least one distribution in this set.
Fairly generically, one can devise near-optimal differentially private algorithms whenever testing the compound hypothesis $(\set{P \in \cP: T(P) \leq \theta_0},\set{P \in \cP: T(P) \geq \theta_1})$ is equivalent to a simple hypothesis testing problem of distinguishing two specific distributions (with parameters $\theta_0$ and $\theta_1$, respectively). We illustrate this with the design of near-optimal estimators for tail decay rates.

\subsection{Related Work}
\label{sec:related}

While the literature on differentially private statistical inference is too vast to survey, we give an overview of the most closely related work.  For additional discussion of the literature, we direct the reader to the survey of \citet{KamathU20}.

\mypar{Minimax Optimality Under Privacy Constraints.} There is now an extensive body of literature on differentially private estimation, which is too large to fully survey here.  The most technically relevant prior work to our work are the results of \citet{Canonne:2019} characterizing optimal differentially private simple hypothesis testing.  The first global minimax lower bounds for multivariate differentially private estimation were given by \citet{BunUV14,DworkSSUV15, SteinkeU17}, based on a technique called \emph{fingerprinting} or \emph{tracing}.  Work by \citet{DuchiF14, KamathSU20} also gave minimax lower bounds for private mean estimation of univariate heavy-tailed statistics, and \citet{AlonLMM19} give minimax lower bounds for privately estimating a univariate distribution in CDF distance.

There are also numerous constructions of minimax optimal differentially private estimators for specific tasks.  Perhaps most closely related to our work are the estimators of~\citet{Karwa:2018} who construct locally minimax optimal estimators for the parameters of a univariate Gaussian, which is a special case of our constructions.

\mypar{Beyond Global Sensitivity.} Several works in the differential privacy literature give general purpose techniques for privately estimating \emph{empirical quantities} in a way that adapts to easy datasets (datasets on which the empirical quantity is stable). These techniques include smooth sensitivity \citep{NissimRS07}, propose-test-release \citep{DworkL09} and the use of Lipschitz extensions to extend regions of low variablility in the quantity of interest 
\citep{Chen:2013, BlockiBDS13, KasiviswanathanNRS13}. 
The most closely related work to ours is that of \cite{Asi:2020}, who give a general class of differentially private estimators for computing empirical quantities that are locally optimal (under some regularity assumptions). 
However, in this work we study estimators for \emph{population quantities}.  While estimating empirical and population quantities are very related, they are fundamentally distinct.  To see why,  consider the example of computing the mean of a Gaussian random variable $N(\mu,\sigma^2)$.  In the non-private setting, the empirical mean gives a locally minimax optimal estimator for $\mu$.  However, applying the locally minimax optimal estimator of Asi and Duchi for the empirical mean will have mean-squared error $\infty$ for any sample size.  In contrast, there is a differentially private estimator for the quantity $\mu$ that has mean-squared error roughly $\sigma^2 / n + 
\sigma^2 / \varepsilon^2 n^2$ for $\eps\leq 1$ (e.g., \cite{Karwa:2018}).
Thus, we have to reason directly about population statistics when we try to construct locally minimax private estimators, and cannot simply apply the transformation of Asi and Duchi to an arbitrary locally minimax non-private estimator.

\mypar{Local Differential Privacy.} Our work studies the standard \emph{centralized model} of differential privacy, where we assume that the estimator $M$ receives the samples $X_1,\dots,X_n$ as input.  There is also a large body of research on so-called \emph{local differential privacy}~\citep{KasiviswanathanLNRS08}, where we assume that differential privacy is applied to each sample before it is collected.  In its most basic non-interactive form, this means that the mechanism can be written in the form $A(M(X_1),\dots,M(X_n))$ where $M$ is differentially private and $A$ is arbitrary.

Locally differentially private estimators are known to have significantly worse rates than general differentially private estimators~\citep{KasiviswanathanLNRS08,BeimelNO08,ChanSS11,DuchiJW13, EdmondsNU20}.  Recent work gives locally minimax optimal estimators subject to local differential privacy~\citep{Duchi:2018,Rohde:2019}.  In addition to different minimax rates, there are key conceptual differences between the local and centralized settings that make the centralized setting more complex to reason about. In particular: (1) The complexity of simple hypothesis testing under local differential privacy is characterized by the total variation distance between the two distributions, whereas a much more subtle notion is required for centralized differential privacy, and (2) The local minimax rate subject to local differential privacy is always larger than that of non-private estimation, whereas our results show that the local minimax rate subject to centralized differential privacy can be either the same or larger than non-private estimation in different ranges of the privacy parameter.

\vspace{0.1in}

\section{Local Estimate Rates and Simple Hypothesis Testing}

\subsection{Local Estimation Rates and Uniform Achievability}

Let $\Delta(\chi)$ be the set of all distributions on a space $\chi$ and $\mathcal{P}\subset\Delta(\chi)$ be a set of distributions on $\chi$. Let $\theta: \mathcal{P}\to \mathbb{R}$ be a functional on $\mathcal{P}$, so for any distribution $P\in\mathcal{P}$, $\theta(P)$ is the parameter that we want to estimate. 
Let $\Qest$ be a class of (potentially randomised) functions $\est:\chi^n \to \mathbb{R}$. For any estimator $\est$ in $\Qest$, $\est$ has local error rate $\ierror{\est}{n}{P}$ if for all $P\in\mathcal{P}$ and $n\in\mathbb{N}$, if $X_1, \cdots, X_n\sim P$ then with probability 0.75: \[|\est(X_1, \cdots, X_n)-\theta(P)|\le \ierror{\est}{n}{P}.\]
Notice that this error rate is \emph{instance specific} in the sense that the error rate is a function of the distribution being sampled from. Since worst-case analysis can be too pessimistic in practice, and the local rate allows the error rate to adapt to \emph{easy} instances of the problem. Defining a notion of \emph{instance optimality} is nuanced since no algorithm can be optimal for all $P$; that is, one can not define an algorithm $\est$ such that $\ierror{\est}{n}{P}\le \min_{\theta'\in\Qest}\ierror{\hat{\theta'}}{n}{P}$ for all $P$. This is easy to see since for any algorithm $\est$ and distribution $P$, the algorithm $\hat{\theta}'(X_1, \cdots, X_n)=P$ satisfies $\ierror{\hat{\theta'}}{n}{P}=0\le \ierror{\est}{n}{P}$. Of course, this algorithm is not a good point of comparison because it does poorly on distributions that are not $P$. Thus, we want to compare to algorithms that perform well on at least two distributions. This leads us to the following definition of the
\emph{optimal local estimation rate} at $P$ by:
\begin{equation}\label{localerrorrate}\errorlocall{n}{P}{\Qest}{\mathcal{P}}{\theta} = \sup_{Q\in\mathcal{P}}\inf_{\est\in\Qest}\max\{\ierror{\est}{n}{P}, \ierror{\est}{n}{Q}\}. \end{equation}
We call this definition the local estimation rate based on the intuition that the hardest distributions to distinguish from $P$ are those that are ``close'' or ``local'' to $P$ (Fig.~\ref{graphofconnection}). The local estimation rate is also sometimes to referred to as the rate of the \emph{hardest one dimensional sub-problem}.
We say an algorithm $\est$ is \emph{instance optimal} if $\ierror{\est}{n}{P} = \errorlocall{n}{P}{\Qest}{\mathcal{P}}{\theta}$ for all $P\in\mathcal{P}$. Intuitively, if $\hat{\theta}$ is instance optimal then for every distribution $P$, if $\hat{\theta}$ performs poorly on $P$, then there exists another distribution $Q$ such that no algorithm $\est'$ performs well on both $P$ and $Q$. In contrast, the trivial algorithm $\hat{\theta}'(X_1, \cdots, X_n)=P$ performs well on $P$, but unnecessarily sacrifices performance on distributions $Q$ far from $P$. Hence the optimal local estimation rate gives a specific kind of lower bound on the performance of any algorithm.

The estimator $\est$ in eqn~\eqref{localerrorrate} has the advantage of being told the two distributions $P$ and $Q$. Hence, unlike worst-case optimality, which is always achieved by some algorithm, an instance optimal algorithm does not necessarily  exist for every estimation problem. 
In fact, a main question in this area is \emph{when} do instance optimal algorithms exist? When an instance optimal algorithm exists we will say the estimation problem satisifies \emph{uniform achievability}. This question of uniform achievability, under the constraint of differential privacy, is the main question of interest in this work. This question has been studied previously in the non-private setting \citep{Donoho:1991} and under the constraint of local differential privacy \citep{Rohde:2019, Duchi:2018}. 
We will refer to the subset of $\Qest$ that contains all $\epsilon$-DP estimators (defined in Section~\ref{DP}) as $\Qesteps$. 

\subsection{Simple Hypothesis Testing}

The crucial insight for understanding the optimal local estimation rate is the connection to simple hypothesis testing. In simple hypothesis testing, we are given two distributions $P$ and $Q$ and the goal is to design an algorithm that given $n$ samples drawn from either $P$ or $Q$, will, with high probability, correctly guess which distribution the samples were drawn from. 
We say a test $T:\chi^n\to\{0,1\}$ distinguishes between $P$ and $Q$ with $n$ samples if $\mathbb{P}(T(P^n)=0)\ge0.75$ and $\mathbb{P}(T(Q^n)=1)\ge0.75$, where the probability is taken over both the randomness in the sample, and the randomness in $T$. Let $\chi^*=\cup_{n\in\mathbb{N}}\chi^n$. We will use $\SC_T(P,Q)$ to denote the sample complexity of a test $T$, i.e., \[\SC_T(P,Q) = \inf\{ n\in\mathbb{N}\;|\; \text{for all } N\ge n, \mathbb{P}(T(P^N)=0)\ge 0.75 \text{ and } \mathbb{P}(T(Q^N)=1)\ge 0.75\}.\]

For every estimator class $\Qest$, we can define an associated class of binary testing algorithms, $\Qtest$, to be the class of binary (potentially randomised) functions $T:\chi^n\to\{0,1\}$ obtained from $\Qest$ by thresholding: \begin{equation}
    \Qtest = \left\{ T_{f,\tau}(X) = \begin{cases} 0 & {\rm if } f(X)<\tau \\ 1 & {\rm otherwise} \end{cases} \;\Bigg|\; f\in\Qest, \tau\in\mathbb{R}\right\}.
    \label{eq:ftest}    
\end{equation}
We will use this translation throughout this work.
Given a class of tests $\Qtest$, define $\SC_{\Qtest}(P,Q) = \inf_{T\in\Qtest} \SC_T(P,Q)$. That is,  $\QSC(P,Q)$ is the smallest $n$ such that there exists a test $T\in\Qtest$ that distinguishes $P$ and $Q$.

\subsection{Connecting Local Estimation Rates and Simple Hypothesis Testing}

Consider the definition of the optimal local estimation rate given in Eqn~\ref{localerrorrate}. Given two distributions $P$ and $Q$, if $\theta(P)$ and $\theta(Q)$ are close then it is easy to find an estimator that performs well on both $P$ and $Q$ (e.g. the estimator that outputs $\frac{1}{2}|\theta(P)-\theta(Q)|$). Similarly, if there exists a test that distinguishes $P$ and $Q$, then it is easy to define an estimator that performs well on both $P$ and $Q$ (e.g. by outputting the test result). Thus, the supremum in the definition is achieved at a distribution $Q$ that is as far as possible from $P$, while still being indistinguishable from $P$. This intuition gives rise to the definition of the \emph{modulus of continuity} at $P\in\mathcal{P}$:
\[\MOCall{n}{P}{\Qest}{\mathcal{P}}{\theta} = \sup\{\;|\theta(P)-\theta(Q)|\;\mid\; \QSC(P,Q)> n \text{ and } Q\in\mathcal{P}\}.\]

The following theorem formalises the intuition above and allows us to translate the question of characterizing $\errorlocall{n}{P}{\Qest}{\mathcal{P}}{\theta}$ into characterizing $\QSC$. This is useful since characterizations of $\QSC$ in a variety of settings already exist, in particular a characterization of $\QSC$ when $\Qest$ is the class of all differentially private estimators was given in \cite{Canonne:2019}.
We say $\Qest$ is closed under post-processing if for any $\est\in\Qest$ and $f:\mathbb{R}\to\mathbb{R}$, $f\circ\est\in\Qest$. \cite{Donoho:1987} studied the characterisation of $\errorlocall{n}{P}{\Qest}{\mathcal{P}}{\theta}$ where $\Qest$ is the class of all possible estimators; their work can be extended to work for any class of estimators closed under post-processing.

\begin{restatable}{proposition}{prop:rlowerbound}
\label{lowerbound} 
For any $\mathcal{P}\subset\Delta(\chi)$, statistic $\theta$, and class of estimators $\Qest$, if $\Qest$ is closed under post-processing and contains all constant functions then 
for all $P\in\mathcal{P}$ and $n\in\mathbb{N}$, 
\[\errorlocall{n}{P}{\Qest}{\mathcal{P}}{\theta}=\ \tfrac 1 2  \MOCall{n}{P}{\Qtest}{\mathcal{P}}{\theta},\] 
where $\Qtest$ is as defined in eqn~\eqref{eq:ftest}.
\end{restatable}

When $\mathcal{P}$ and $\theta$ are clear from context, we write $\MOC{n}{P}{\Qtest}$ for $\MOCall{n}{P}{\Qtest}{\mathcal{P}}{\theta}$, and similarly $\errorloc{n}{P}{\Qest}$ for $\errorlocall{n}{P}{\Qest}{\mathcal{P}}{\theta}$.
We will primarily be concerned with the class of differentially private estimators in this paper, which is closed under post-processing and contains all constant functions. We include the proof below to build intuition for this connection. 

\begin{proof} 

Let us first prove that $\error_n^{\loc}(P, \mathcal{P}, \Qest, \theta)\ge \frac12 \MOCall{n}{P}{\Qtest}{\mathcal{P}}{\theta}$. 
Suppose for sake of contradiction that $\error_n^{\loc}(P, \mathcal{P}, \Qest, \theta)< \frac12 \MOCall{n}{P}{\Qtest}{\mathcal{P}}{\theta}$. Then there exists $\hat\theta\in\Qest$ and $Q\in\mathcal{P}$ such that $\QSC(P,Q)> n$, $\frac12 \MOCall{n}{P}{\Qtest}{\mathcal{P}}{\theta}~\le~\frac12 |\theta(P)~-~\theta(Q)|$ and, 
\[\ierror{\est}{n}{P}< \tfrac12 \MOCall{n}{P}{\Qtest}{\mathcal{P}}{\theta} \;\;\; \text{and }\;\;\;\ierror{\est}{n}{Q}< \tfrac12 \MOCall{n}{P}{\Qtest}{\mathcal{P}}{\theta}\]
Therefore, $T_{\est, \frac{1}{2}(\theta(P)+\theta(Q))}$ (as defined in eqn~\eqref{eq:ftest}) distinguishes $P$ and $Q$ with $n$ samples, which is a contradiction since $\QSC(P,Q)> n$. Figure~\ref{graphofconnection} gives a graphical representation of this, if the balls do not overlap then we have a test that distinguishes $P$ and $Q$.

For the opposite inequality, we need to show that for all $Q\in\mathcal{P}$, there exists an estimator $\est~\in~\Qest$ such that $\max\{\ierror{\est}{n}{P}, \ierror{\est}{n}{Q}\}\le \tfrac12 \omega_{n, \QSC}(P) $. First suppose that $\QSC(P,Q)>n$, that $Q$ lies inside the blue ball around $P$ in Figure~\ref{graphofconnection}, and $\frac12 |\theta(P)-\theta(Q)|\le \frac12 \omega_{n, \QSC}(P)$. Let $\est$ be the constant function that always outputs $\frac12 |\theta(P)~-~\theta(Q)|$ so $\ierror{\est}{n}{P}=\frac12 |\theta(P)-\theta(Q)|$ and $\ierror{\est}{n}{Q}=\frac12 |\theta(P)-\theta(Q)|$, so we are done. Finally, suppose that $\QSC(P,Q)\le n$ so there exists $\est\in\Qest$ and $\tau\in\mathbb{R}$ such that 
\[\mathbb{P}[\est(P^n)\le\tau]\ge 0.75 \;\;\;\text{ and }\;\;\; \mathbb{P}[\est(Q^n)\le\tau]\ge 0.75.\] 

Let $f:\mathbb{R}\to\mathbb{R}$ be defined by $f(x)=\theta(P)$ if $\hat{\theta}(x)\le\tau$ and $f(x)=\theta(Q)$ is $\hat{\theta}(x)\ge \tau$ so $f\circ\est\in\Qest$ and $\ierror{f\circ\est}{n}{P}=0$ and $\ierror{f\circ\est}{n}{Q}=0$, so we are done. 
\end{proof}

For distributions $P$ and $Q$, let $$\Hell(P,Q) = \sqrt{\int_{\chi} (P(x)-Q(x))^2~\mathrm{d}x}$$ be the Hellinger distance between $P$ and $Q$. It is well known that for the class of all estimators, \[\QSC(P,Q)~=~\Theta\left(\frac{1}{\Hell^2(P,Q)}\right),\] so the following corollary is an immediate consequence of Theorem~\ref{lowerbound}. Define the $\Hell$-information by:
\[J^{-1}_{\Hell, P}(\beta) = \sup\left\{|\theta(P)-\theta(Q)|\;\Big|\; \Hell(P,Q)\le \beta, Q\in\mathcal{P}\right\}.\]

\begin{corollary}[Non-private optimal local estimation rate \citep{Donoho:1987}]\label{nonprivLER}
Let $\Qest$ be the set of all functions, then there exists constants $C_1$ and $C_2$ such that for any family $\mathcal{P}$, any statistic $\theta$,
any distribution $P$, and $n\in\mathbb{N}$, \[\error_n^{\loc}(P, \mathcal{P}, \Qest, \theta)\in\\\left[J^{-1}_{\Hell, P}\left(\frac{C_1}{\sqrt{n}}\right), J^{-1}_{\Hell, P}\left(\frac{C_2}{\sqrt{n}}\right)\right].\]
\end{corollary}

\subsection{Super efficiency}
\label{sec:supereff}

The optimal local estimation rate $\errorlocall{n}{\cdot}{\Qest}{\mathcal{P}}{\theta}$ has the property that for any estimator $\hat{\theta}$, if $\hat{\theta}$ achieves better accuracy than $\errorlocall{n}{P}{\Qest}{\mathcal{P}}{\theta}$ at some distribution $P$, then there exists a distribution $Q$ such that the accuracy of $\hat{\theta}$ at $Q$ is \emph{at least} as bad as $\errorlocall{n}{Q}{\Qest}{\mathcal{P}}{\theta}$.
One can also ask if an estimation rate satisfies the stronger condition of having a \emph{super-efficiency} result. Roughly, an estimation rate $R$ has a super-efficiency result if for any estimator $\hat{\theta}$ that achieves better accuracy than $R(P)$ at a particular value $P$, there exists another value $Q$ where the accuracy of $\hat{\theta}$ is \emph{strictly} worse than $R(Q)$. 
A super-efficiency result for a given rate $R$ shows, in a sense, that that $R$ is a meaningful target rate.
The optimal local estimation rate does not necessarily satisfy a super-efficiency result for general families.
Super-efficiency may hold for \emph{specific} families but a general result seems to require further assumptions. We leave the question of super-efficiency of the optimal local estimation rate to future work, since our focus is the general regime.

\section{Differentially Private Simple Hypothesis Testing and the Optimal Local Estimation Rate in the High Privacy Setting}
\label{sec:testing}

In this section, we will discuss the optimal test statistic for differentially private simple hypothesis testing and characterise the optimal local estimation rate in the high privacy setting. The test statistic we give is a slight variant on that presented in \cite{Canonne:2019}, who first characterised the sample complexity of differentially private simple hypothesis testing. The test statistic given here is more efficient and more amenable to the estimation problem. The characterisation of the local estimation rate in the high privacy regime is simpler than in other regimes, and offers a direct comparison to the local estimation rates in the non-private and local differential privacy regimes.

\subsection{Differential Privacy}\label{DP} 
In this work we are concerned with estimators that satisfy \emph{differential privacy}, which we will formally define in this section. Let $\mathcal{X}$ be a data universe 
and $\mathcal{X}^n$ be the space of datasets of size $n$. Two datasets $d, d' \in \mathcal{X}^n$ are neighboring, denoted $d \sim d'$, if they differ on a single record. %, $(x_i, y_i)$.
Let $\mathcal{Y}$ be an output space. 

\begin{definition}[$\epsilon$-Differential Privacy \citep{DworkMNS06}]\label{def:DP} Given privacy parameters $\epsilon\ge0$ and $\delta\in[0,1]$,
a randomized mechanism $M: \mathcal{X}^n \rightarrow \mathcal{Y}$ is $\epsilon$-\emph{differentially private} if for all datasets $d \sim d' \in \mathcal{X}^n$, and events $E\subseteq\mathcal{Y}$,

\begin{align*} \label{def:dp-with-inputs}
&\Pr[M(d, \text{hyperparams}) \in E]
\leq e^\epsilon \cdot\Pr[M(d', \text{hyperparams}) \in E]+\delta,
\end{align*}

where the probabilities are taken over the randomness induced by $M$. 
\end{definition}

The key intuition for this definition is that the distribution of outputs on input dataset $d$ is almost indistinguishable from the distribution of outputs on input dataset $d'$. Therefore, given the output of a differentially private mechanism, it is impossible to confidently determine whether the input dataset was $d$ or $d'$. 
For strong privacy guarantees, the privacy-loss parameter is typically taken to be a small constant less than $1$ (note that $e^\epsilon \approx 1+\epsilon$ as $\epsilon \rightarrow 0$) and $\delta$ is taken to be very small (say $10^{-6}$). In fact, for simple hypothesis testing, we can show that if $\epsilon>1$, then for any $\delta\in[0,1]$, the private sample complexity within a constant factor of the non-private sample complexity, i.e., $\SC_{\epsilon, \delta}(P,Q)=\Theta(\SC(P,Q))$. Hence, for the remainder of this work, we will assume that $\epsilon\le 1$.
Note if $\Qesteps$ is the set of all $\epsilon$-DP estimators, then $\Qtesteps$ is the set of all $\epsilon$-DP tests.

\subsection{An Optimal Differentially Private Simple Hypothesis Test}\label{DPtestingsec}

A characterisation of the sample complexity of differentially private simple hypothesis testing was given in \cite{Canonne:2019}. They showed that a simple \emph{noisy and clamped} version of the log likelihood ratio test gave an optimal sample complexity differentially private simple hypothesis test. Given distributions $P$ and $Q$, let $\CLLR_a^b$ be the clamped log-likelihood statistic with thresholds $a$ and $b$, and $\NCLLR_a^b$ be a noisy version: 
\begin{equation}\label{cllr}
\CLLR_a^b(X) = \sum_{i=1}^n \left[\ln\frac{P(x_i)}{Q(x_i)}\right]_a^b\;\; \text{ and }\;\; \NCLLR_a^b(X)=\CLLR_a^b(X)+\Lap\left(\frac{|b-a|}{\epsilon }\right).
\end{equation}
In the original version of this test, the authors' proved that this test statistic gave rise to an optimal test when one set $b=\epsilon$ and $a=-\epsilon'$, where $\epsilon'$ is some function of $\epsilon, P$ and $Q$. In Appendix~\ref{DPtestingappendix}, we improve on their results to show that setting $b=\Theta(\epsilon)$ and $-a=\Theta(\epsilon)$ is sufficient. This
extension is crucial to us in our estimation algorithm where $\epsilon'$ can not be computed. This is also of independent interest as an improvement of the testing result: unlike the original test presented in \cite{Canonne:2019}, setting $a=-\epsilon$ and $b=\epsilon$, results in an efficient test which only requires oracle access to $P$ and $Q$. The original result in \cite{Canonne:2019} required full knowledge of the distributions $P$ and $Q$ in order to compute $\epsilon'$. 
In order to simplify notation we use $\SC_{\epsilon}(P,Q):= \SC_{\Qesteps}(P,Q)$ to denote the optimal sample complexity for distinguishing $P$ and $Q$ using an $\epsilon$-DP algorithm.
The proof of the following proposition is found in Section~\ref{arobustthresholds}

\begin{restatable}{proposition}{restaterobustthresholds}
\label{robustthresholds}
If $\epsilon=O(1)$ then for all $a=\Theta(\epsilon)$ and $b=\Theta(\epsilon)$, there exists constants $C_1$ and $C_2$ such that for all distributions $P$ and $Q$,
\[\SC_{\NCLLR_a^b}(P,Q)\in[C_1 \cdot \SC_{\epsilon}(P,Q), C_2 \cdot \SC_{\epsilon}(P,Q)].\]

\end{restatable}

The sample complexity of $\NCLLR_{\epsilon}^{\epsilon}$, characterised in \cite{Canonne:2019}, has a nuanced dependence on $\epsilon$, $P$ and $Q$. If $\epsilon$ is large enough, privacy comes for free, and $\SC_{\epsilon}(P,Q)=\Theta\left(\SC(P,Q)\right)$. As $\epsilon$ decreases the dependence becomes more complicated. However, in Lemma~\ref{smalleps} we will show that once $\epsilon$ is small enough, $\epsilon_n\le\frac{1}{\sqrt{n}}$, the dependence is once again simple. 

For hypothesis tests with constant error probabilities, the sample complexity bounds are equivalent, up to constant factors, for pure $\epsilon$-differential privacy, and the less strict notions of approximate $(\epsilon,\delta)$-differential privacy and concentrated differential privacy \citep{Dwork2016ConcentratedDP, Bun:CDP} (see \citet[Lemma 5]{Acharya:2018}). Consequently, the test $\NCLLR_{\epsilon}^{\epsilon}$ is optimal (up to constants) for each of these weaker notions. The class of estimators defined by each of these notions is closed under post-processing and thus, by Theorem~\ref{lowerbound}, the optimal local estimation rate is, up to constants, the same for each of these notions. This may seem like a contradiction since there are many well-known cases of asymptotic gaps in the estimation rate of pure differential privacy and approximate differential privacy. However, the optimal local estimation rate need not be \emph{uniformly achievable} under all (or any) of these notions of privacy, leaving room for a gap in the achievable estimation rate under pure, concentrated and approximate DP.

\subsection{A Lower Bound for Instance Optimal Estimation in the High Privacy Regime}\label{generalhighprivacy}

The characterization of the local estimation rate is significantly more complex in the central DP regime than in the local DP or non-private regimes. This is a direct consequence of the characterisation of the optimal sample complexity of simple hypothesis testing being more nuanced in the central DP regime than the local DP or non-private regimes. However, the existence of a simple characterisation of the sample complexity in the high privacy regime allows us to give a simple characterisation of the local estimation rate in that regime.

For distributions $P$ and $Q$, let $\TV(P,Q) = (1/2) \int |P(x)-Q(x)|dx$ be the total variation distance. For $\beta\in[0,1]$, we define the $L_1$-information at a distribution $P\in\mathcal{P}$ by
\begin{equation}\label{L1info}
J_{\TV, P}^{-1}(\beta) = (1/2)\cdot\sup\left\{|\theta(P)-\theta(Q)|\;\Big|\; \TV(P, Q)\le\beta, Q\in\mathcal{P}\right\}.
\end{equation}

Note that the $L_1$-information is the analogue of the $\Hell$-information, which characterizes the sample complexity in the non-private setting, using the total variation distance (also known as the $L_1$-norm) instead of the Hellinger distance. Our estimation rate in the high privacy regime is characterized by the $L_1$-information. This follows immediately from Theorem~\ref{lowerbound} and Lemma~\ref{smalleps} which we'll state below. 

\begin{theorem}\label{smallepslower} Let $\Qesteps$ be the set of all $\epsilon$-differentially private estimators. For any constant $k$ there exists a constants $C_1$ and $C_2$ such that if $\epsn\le \frac{k}{\sqrt{n}}$, then for all families $\mathcal{P}$, $P\in\mathcal{P}$, and $n\in\mathbb{N}$, \[\LinfoP\left(\frac{C_1}{n\epsn}\right)\le \error^{\loc}_n(P, \mathcal{P}, \Qesteps, \theta) \le \LinfoP\left(\frac{C_2}{n\epsn}\right).\]
\end{theorem}

Theorem~\ref{smallepslower} is interesting to contrast with Corollary~\ref{nonprivLER}, which gives the estimation rate in the non-private regime. Note first that if $\epsn=O(\frac{1}{\sqrt{n}})$ then $n\epsn<\sqrt{n}$, so the estimation rate is indeed slower under the constraint of privacy. Further, the metric characterizing the problem changes from the Hellinger distance to the total variance distance. A similar phenomenon is observed under local differential privacy constraints in \cite{Duchi:2018}.

\begin{theorem}[Local DP \citep{Duchi:2018}]
Let $\Qest_{\text{local},\epsilon}$ be the set of all $\epsilon$-locally differentially private functions, there exists constants $C_1$ and $C_2$ such that for all families $\mathcal{P}$, estimators $\theta$, 
any $P\in\mathcal{P}$, and $n\in\mathbb{N}$, \[\LinfoP\left(\frac{C_1}{\epsilon\sqrt{n}}\right)\le \error_n^{\loc}(P, \Qest_{\text{local},\epsilon}) \leq \LinfoP\left(\frac{C_2}{\epsilon\sqrt{n}}\right).\]
\end{theorem}

The corresponding class of testing functions $\Qtest$ contains the set of all $\epsilon$-local DP binary functions. \cite{Duchi:2014} showed that the sample complexity for distinguishing between two distributions $P$ and $Q$ under local differential privacy is 
$\Theta\left(\frac{1}{\epsilon^2TV^2(P,Q)}\right)$.

As discussed in \cite{Duchi:2018}, the change from the Hellinger modulus of continuity to the total variation modulus of continuity has implications for how well one can expect estimation algorithms in the high privacy setting to adapt to problem-specific difficulty. For example in the case of Bernoulli estimation, the non-private local estimation rate for a Bernoulli with parameter $p\in[0,1]$ is  $\Theta(\sqrt{p(1-p)/n})$, which shows that estimation algorithms in the non-private (and low central privacy setting) are able to adapt to ``easy`` instances of the problem. In contrast, in the high privacy setting, the local estimation rate is $\Theta\left(\frac{1}{\epsilon n}\right)$, which is the same for all $p$, showing that private algorithms in this regime are not able to adapt to ``easy" instances. As mentioned earlier, this is a direct consequence of the fact that the Hellinger distance between ${\rm Bernoulli}(p)$ and ${\rm Bernoulli}(p+\alpha)$ is a function of $p$, while the total variation distance between these two distributions is independent of $p$. 

Theorem~\ref{smallepslower} is a direct consequence of the following characterisation of the sample complexity of private hypothesis testing in the high privacy regime. The proof follows from the fact that in the high privacy regime, $\epsilon\le k/\sqrt{n}$, a noisy Scheff\'e test performs as well as the optimal test $\NCLLR_{\epsilon}^{\epsilon}$.

\begin{lemma}[High Privacy Sample Complexity Characterisation]
\label{smalleps}
For any constant $k$, there exists constants $C_1$ and $C_2$ such that for any distributions $P$ and $Q$, if $\epsilon_n\le\frac{k}{\sqrt{n}}$ then \[SC_{\epsilon_n}(P,Q)\in\left[\frac{C_1}{\epsilon_n \cdot \TV(P,Q)}, \frac{C_2}{\epsilon_n \cdot \TV(P,Q)}\right].\]
\end{lemma}

Before we prove Lemma~\ref{smalleps}, a quick note on the privacy parameters. We will allow our privacy parameter, $\epsilon$, to vary with the size of the database, $n$, so let $\epsilon_n$ be a sequence and $n:[0,\infty)\to\mathbb{N}$ be such that $\epsilon_{n(\epsilon)}=\epsilon$ and $n(\epsilon_n)$. We will often abuse notation and drop the argument of the function, e.g., referring to $\epsilon_n$ as simply $\epsilon$. We will assume that $\epsilon$ is decreasing, so the larger the dataset, the more private we require our algorithm to be. We will say a simple hypothesis testing problem has sample complexity $n=SC_{\epsilon_n}(P,Q)$ if $n=n(\epsilon)$ is the smallest value such that $SC_{\epsilon}(P,Q)$ and $n(\epsilon)$ intersect.

\begin{proof}
The lower bound portion of this lemma is not specific to the high privacy setting; there exists $C_1$ such that for all $\epsilon$, $SC_{\epsilon}(P,Q)\ge\frac{C_1}{\epsilon\TV(P,Q)}$. One way to prove this is as a direct consequence of \cite[Theorem 11]{Acharya:2018}. This theorem argues that one can lower bound the sample complexity of an $\epsilon$-DP test by upper bounding the Hamming distance between two datasets of size $n$ drawn from either $P$ and $Q$, i.e., the Hamming distance between $X$ and $Y$ where $X\sim P^n$ and $Y\sim Q^n$. 

For the upper bound, we will show that a noisy version of the simple Scheff\'e test has sample complexity $O(1/\epsilon TV(P,Q))$ in the high privacy regime. Let $E=\{x\;\|\; P(x)>Q(x)\}$ be the Scheff\'e set and define the test statistic $f_E$ by, for any database $X=\{x_1, \cdots, x_n\}$, \[f_E(X) = \frac{1}{n}\sum_{i=1}^n \mathds{1}_{x_i\in E}+\text{Lap}\left(\frac{1}{\epsilon n}\right).\]
Then by definition of the total variation distance, \[\mathbb{E}_{X\sim P^n}[f_E(X)]-\mathbb{E}_{X\sim Q^n}[f_E(X)] = \Pr_{x\sim P}(x\in E)-\Pr_{x\sim Q}(x\in E) = \text{TV}(P,Q).\]
Further, 
\[\max\{\var_{X\sim P^n}(f_E(X)), \var_{X\sim Q^n}(f_E(X))\}\le \frac{1}{n}+\frac{1}{\epsilon^2n^2}\le \frac{1+k^2}{\epsilon^2 n^2},\]
where the last inequality follows since $\epsilon\le\frac{k}{\sqrt{n}}$. Therefore, if $n\ge \sqrt{\frac{1+k^2}{12}}\frac{1}{\epsilon_n \text{TV}(P,Q)}$, we have that \[\mathbb{E}_{X\sim P^n}[f_E]-\mathbb{E}_{X\sim Q^n}[f_E]\ge \frac{1}{12}\max\{\sqrt{\var_{X\sim P^n}(f_E(X))}, \sqrt{\var_{X\sim Q^n}(f_E(X))}\}.\] A simple application of Chebyshev's inequality (for details see \cite[Lemma 2.6]{Canonne:2019}) implies that there exists a threshold $\tau$ such that the test that outputs $P$ if $f_E(X)\ge \tau$ and $Q$ otherwise, distinguishes between $P$ and $Q$ with sample complexity $\sqrt{\frac{1+k^2}{12}}\frac{1}{\epsilon_n \text{TV}(P,Q)}$.
\end{proof}

\section{One-Parameter Exponential Families: Characterising the Optimal Local Estimation Rate and Uniform Achievability}\label{expfams}

We now turn our attention to an example where uniform achievability is possible under differential privacy: one-parameter exponential families. In this section we will characterize the optimal local estimation rate of estimating the parameter in a one-parameter exponential family, then show that this optimal local estimation rate is uniformly achievable under differential privacy. In particular, we will see how the results of Section~\ref{DPtestingsec} on the form of the optimal DP simple hypothesis test, and it's sample complexity, inform the design of the locally minimax estimator.

One parameter exponential families are a broad class of families of distributions that encompasses many natural distributions. Examples of exponential families include Poisson distributions, Binomial distributions, normal distributions with known variance and normal distributions with known mean.
Formally, a \emph{one parameter exponential family}, $\mathcal{P_{\mu}}=\{P_{\theta}\}$, is determined by a base measure $\mu$ such that for each $\theta$, the distribution $P_{\theta}$ has density \[p_{\theta}(x) := e^{\theta x-A(\theta)} \quad  \text{(relative to $\mu$)},\] where $A(\theta) = \ln\int e^{\theta x} d\mu(x)$ is the normalisation.\footnote{It is common to see a sufficient statistic, $T(x)$, included in the definition of an exponential family so that $p_{\theta}(x) := e^{\theta T(x)-A(\theta)}$. Defined in this way, an exponential family can be defined over any space, not simply $\mathbb{R}$. However, for the purpose of estimating $\theta$, the two definitions are equivalent up to a change in the base measure, $\mu$.} Note that the mean and the variance have the following simple formulations: $\mathbb{E}_{\theta}[x]=A'(\theta)$ and $\var_{\theta}(x)=A''(\theta)$. The formula for $p_{\theta}$ does not give a well defined distribution for values of $\theta$ for which $A(\theta)=\infty$, so each measure $\mu$ has an associated range which we will denote $\thetarange_{\mu}=\{\theta\;|\; A(\theta)<\infty\}$. When $\mu$ is clear from context, we will drop the dependence on $\mu$ and refer to $\thetarange_{\mu}$ simply as $\thetarange$. 

 Let us begin with the characterization of the optimal local estimation rate. The formal version of this theorem is a combination of Corollary~\ref{smallepslowerexp} and Corollary~\ref{largeepslowerexp}, which characterize the optimal local estimation rate separately for the high and low privacy regimes.

\begin{theorem}[Characterization of Optimal Local Estimation Rate---Simplified from Corollaries \ref{smallepslowerexp} and \ref{largeepslowerexp}]\label{thm:informalrate}
For all exponential families (i.e., measures $\mu$), $\delta>0$, all sequences of privacy parameters $\epsn\in[0,1]$, $n\in\mathbb{N}$, and $\theta_0\in\thetarange_{\mu}$,
  \[\omega_n(P_{\theta_0}, \Qtesteps) = \Theta_{\mu}\left( \frac{1}{\sqrt{A''(\theta_0)}\min\{n\epsn, \sqrt{n}\}} \right) ,\]
  where the $\Theta_{\mu}$ notation hides constants depending only on $\mu$ (but not $\theta_0$).
\end{theorem}

This convergence result is uniform in a fairly strong sense. Given a family defined by a measure $\mu$, there exists constants $C_1,C_2$ such that for all sequences $\epsn\in[0,1]$, sufficiently large $n$, and $\theta_0\in\thetarange_{\mu}$,  $$\omega_n(P_{\theta_0}, \Qtesteps) \in \left[\frac{C_1}{\sqrt{A''(\theta_0)}\min\{n\epsn, \sqrt{n}\}}, \frac{C_2}{\sqrt{A''(\theta_0)}\min\{n\epsn, \sqrt{n}\}}\right].$$
The formal statements of this theorem are slightly stronger than Theorem~\ref{thm:informalrate} in that we show that the constants $C_1$ and $C_2$ depend only on a few properties of $\mu$. We will discuss these properties later in this section. The non-private local estimation rate for exponential families is $\omega_n(P_{\theta_0}, \Qtest)~=~\Theta_{\mu}\left( \frac{1}{\sqrt{A''(\theta_0)n}} \right)$, so we can see that in the low privacy regime, privacy comes for free. In the high privacy regime, this characterisation matches the L1-information at $P$ as expected from Theorem~\ref{smallepslower}.

Under some mild conditions, this optimal local estimation rate is actually uniformly achievable. That is, there exists an algorithm that achieves the optimal local estimation rate. The following is an informal statement of Proposition~\ref{prop:highprivupper} and Proposition~\ref{mainexp}, which contain the formal uniform achievability statements in the high and low privacy regimes separately.

\begin{theorem}[Uniform Achievability------Simplified from Propositions~\ref{prop:highprivupper} and~\ref{mainexp}]\label{informalmainexp} For all exponential families (i.e., all measures $\mu$), there is an algorithm $\mathcal{A}_\mu$ such that for all $\theta_0\in\thetarange$, and for all sequences $\epsn\in[0,1]$, $\delta\in[0,1]$ and $n\in\mathbb{N}$, $\mathcal{A}_{\mu}(\epsilon, \delta, \cdot)$ is $(\epsilon, \delta)$-DP and
  \[\ierror{\mathcal{A}_{\mu}}{n}{\theta_0}= O_\mu( \omega_n(P_{\theta_0}, \Qtesteps) )\, .\]

The $O_\mu$ notation hides constants depending on $\mu$ (but not $\theta_0$). 
\end{theorem}

In the low privacy regime, the optimal test is based on the clamped log-likelihood ratio test from \cite{Canonne:2019}. 
Much of the work of both Theorem~\ref{thm:informalrate} and Theorem~\ref{informalmainexp} goes into finding the right conditions for uniform convergence. There are several key quantities that determine the optimal local estimation rate, and when it is uniformly achievable.

\begin{itemize}
    \item Define the radius of smoothness of $A''$ around $\theta$ as \[\kappa(\theta)= \max\left\{r\;\Big|\; \forall \theta'\in[\theta-r, \theta+r], \frac{A''(\theta)}{A''(\theta')}\in\left[\frac{1}{2}, 2\right]\right\}.\] By the continuity of $A''$, $\kappa(\theta)>0$ for all $\theta$. Recall that $A''(\theta)$ is the variance of $P_{\theta}$ so $\kappa(\theta)$ is related to the smoothness of the variance. Our theorems will be strongest for families where $\kappa(\theta)$ is large for most $\theta$ of interest. 
    Given the characterisation of $\omega_n(P_{\theta_0}, \Qtesteps)$ in Theorem~\ref{thm:informalrate}, if $\kappa(\theta)$ is large, then this means that the local rate varies slowly. 
     The parameter $\kappa(\theta)$ affects the achievability in two main ways.
    \begin{itemize}
    
        \item The form we give for the local estimation rate holds for sample sizes $n$ above some threshold that depends on $\kappa(\theta)$. Specifically, one requirement is that  $\LinfoP(\kappa(\theta)) \geq \frac{1}{\epsilon n}$. 
        That is, $n$ must be large enough that if $\theta$ satisfies $TV(P_\theta, P_{\theta'}) \leq 1/\epsilon n$, then $\theta'\in\thetarange_{\mu}(\theta)$, where $\thetarange_{\mu}(\theta)= [\theta - \kappa(\theta), \theta + \kappa(\theta)]
        $. 
        This condition ensures that with high probability our private estimate lies within $\kappa(\theta)$ of $\theta$.
    
        \item In order for our procedure to succeed (that is, produce an accurate estimate) with probability at least $1-\beta$, we require that there exists a constant $C>0$ such that  $\kappa(\theta)~\geq~\frac 1 C \cdot \frac {\sqrt{\log(2/\beta)}}{ \sqrt{A''(\theta)}}$ for all $\theta$.
        Under this condition, the distributions $P_{\theta}$ are sub-Gaussian, that is 
        \[\mathbb{P}_{x\sim P_{\theta}}\left(|x-A'(\theta)|\ge (2+C)\sqrt{A''(\theta)}\sqrt{\ln(2/\beta)}\right)\le\beta.\] 
        This light-tailed property ensures that with high probability a dataset sampled from $P_{\theta}$ lies mostly in an interval of width $O(\sqrt{A''(\theta)}\sqrt{\ln(2/\beta)})$. 
        This allows us to limit the amount of noise added for privacy to also scale with the standard deviation $\sqrt{A''(\theta_0)}$. Without a light tailed assumption, additional noise needs to be added to maintain privacy, resulting in a worse estimation rate. We see this effect in estimating the parameter of a Bernoulli distribution, where the scale of the noise needed to maintain privacy scales with $\frac{1}{\epsilon n}$, rather than $\frac{\sqrt{p(1-p)}}{\epsilon n}$ which would be predicted by Theorem~\ref{thm:informalrate}. The family of Bernoulli distributions fails to satisfy this assumption unless we constrain $\min(p,1-p)$ to be at least a constant.
    \end{itemize}
    \item We will also require that the central standardised fourth moment is bounded. That is, there exists a constant $\thirdmoment$ such that \[\frac{\mathbb{E}_{\theta}\Big(x-A'(\theta)\Big)^4}{A''(\theta)^2} = \frac{\mathbb{E}_{\theta}\Big(x - \mathbb{E}_{\theta} x\Big)^4}{\var(P_{\theta})^2} \le\thirdmoment.\] 
    The central standardised fourth moment is also known as the \emph{kurtosis}, this assumption allows us to give a \emph{lower} bound on the tails of $P_{\theta}$. That is, there exists a constant $c$ such that $\Pr_{P_{\theta}}[X\ge A'(\theta)+\frac{1}{2}\sqrt{A''(\theta)}]\ge c$. This assumption is required for our algorithm to properly estimate the standard deviation $\sqrt{A''(\theta)}$, which plays a crucial role in our estimator. It is possible that this assumption can be weakened with an improved private variance estimator. 
\end{itemize}

\subsection{Examples of Exponential Families}

Before we move onto the proofs of Theorems~\ref{thm:informalrate} and~\ref{informalmainexp}, let us consider a few examples of simple exponential families and the implications of these theorems.

\begin{example}[Gaussian mean with known variance] Note we can write \[\frac{1}{\sigma\sqrt{2\pi}}e^{-\frac{1}{2}\left(\frac{x-\mu}{\sigma}\right)^2} = e^{x\frac{\mu}{\sigma}-\frac{1}{2}\left(\frac{\mu}{\sigma}\right)^2} e^{-\frac{1}{2}\frac{x^2}{\sigma^2}}.\] Thus, if $\sigma$ is known we can define an exponential family by $\theta=\mu/\sigma$, $A(\theta)=\frac{1}{2}\theta^2$ and $d\mu(x) = e^{-\frac{1}{2}\frac{x^2}{\sigma^2} }dx$. Notice that $A''(\theta)=1$ so $\kappa(\theta)=\infty$ for all $\theta$. Further, the central standardised fourth moment is 3. This is the ideal behavior for the conditions needed for Theorem~\ref{thm:informalrate} and Theorem~\ref{informalmainexp} to hold.
Therefore, the local minimax optimal rate for privately estimating $\theta$ is $\left[\frac{C_1}{\min\{n\epsn, \sqrt{n}\}}, \frac{C_2}{\min\{n\epsn, \sqrt{n}\}}\right]$, which implies that local minimax optimal rate for privately estimating the mean $\mu$ is \[\left[\frac{C_1\sigma}{\min\{n\epsn, \sqrt{n}\}}, \frac{C_2\sigma}{\min\{n\epsn, \sqrt{n}\}}\right].\] This recovers a result of \cite{Karwa:2018}.
\end{example}

\begin{example}[Poisson family] Recall that the Poission distribution, characterized by paramater $\lambda>0$, assigns mass to nonnegative integers according to \[\frac{\lambda^xe^{-\lambda}}{x!} = e^{x \ln\lambda-\lambda}\frac{1}{x!}.\] We can define an exponential family by taking  $\theta=\ln\lambda$, $A(\theta)=\lambda=e^{\theta}$ and the base measure $\mu$ that assigns mass $\frac 1 {x!}$ to all nonnegative integers $x$. 
Thus, $A''(\theta) = e^{\theta}=\lambda$, which implies that $\kappa(\theta)=\ln 2$ for all $\theta$. Further, the central fourth moment is $3\lambda^2+\lambda$; normalized by the square of the variance, we get $3+\frac 1 \lambda$.
Suppose there exists a constant $c>0$ such that it is guaranteed that $\lambda>c$. Once $n$ is sufficiently large, the local minimax optimal rate for privately estimating $\theta$ is $\left[\frac{C_1}{\sqrt{\lambda}\min\{n\epsn, \sqrt{n}\}}, \frac{C_2}{\sqrt{\lambda}\min\{n\epsn, \sqrt{n}\}}\right]$. Using the first-order Taylor approximation for  $\lambda=e^{\theta}$, we see that the local minimax optimal rate for privately estimating $\lambda$ is in \[\left[\frac{C_1\sqrt{\lambda}}{\min\{n\epsn, \sqrt{n}\}}, \frac{C_2\sqrt{\lambda}}{\min\{n\epsn, \sqrt{n}\}}\right]\]
for constants $C_1$ and $C_2$ depending on $c$.
\end{example}

\subsection{Basic Facts about Exponential Families}

Let us begin by reviewing some basic properties of exponential families. A family $\{p_{\theta}(x)\}$ has monotone likelihood ratio if for all $ \theta<\theta'$, $\frac{p_{\theta'}(x)}{p_{\theta}(x)}$ is a non-decreasing function of $x$. Exponential families have monotone likelihood ratio.

\begin{lemma}[Lehmann \& Romano, Lemma 3.4.2]\label{monotonelikelihood}
Let $\{p_{\theta}(x)\}$ be a family with monotone likelihood ratio, then
\begin{itemize}
\item If $g$ is a nondecreasing function of $x$, then $\mathbb{E}_{\theta}g(x)$ is a nondecreasing function of $\theta$.
\item For any $\theta<\theta'$, and any $t$, $\mathbb{P}_{\theta}(x>t)<\mathbb{P}_{\theta'}(x>t).$
\end{itemize}
\end{lemma}

\begin{restatable}{corollary}{rmonotonecontinuous}
\label{monotonecontinuous} Assume that $A''$ is continuous. Then 
for all $\theta$, the function $h\mapsto \TV(P_{\theta}, P_{\theta+h})$ is continuous and monotonically increasing on $h\ge 0$.
\end{restatable}

\begin{lemma}\label{belongsinkappa} For any $\theta, \theta'$,
if $|A'(\theta')-A'(\theta)|\le \frac{1}{2}A''(\theta)\kappa(\theta)$ then $\theta'\in\thetarange(\theta).$
\end{lemma}

\begin{proof}
If $|\theta-\theta'|\ge\kappa(\theta)$ then $|A'(\theta)-A'(\theta')|\ge \min_{\theta''\in[\theta,\theta']}A''(\theta'')|\theta-\theta'|\ge \frac{1}{2}A''(\theta)\kappa(\theta).$ 
\end{proof}

The following concentration inequality is proved in Section~\ref{aconcentrationexp}.

\begin{restatable}{lemma}{rconcentrationexp}{\em [Concentration Inequality for Exponential Families]}
\label{concentrationexp} For all measures $\mu$, $\theta\in\thetarange$, and $\beta\in[0,1]$, 

{\color{black} 
\[\mathbb{P}_{x\sim P_{\theta}}\left(|x-A'(\theta)|\ge 2\sqrt{A''(\theta)}\sqrt{\ln(2/\beta)} + \frac{\ln(2/\beta)}{\kappa(\theta)}\right)\le\beta.\]

In particular, if $\kappa(\theta)\geq \frac 1 C \cdot \frac {\sqrt{\log(2/\beta)}}{ \sqrt{A''(\theta)}}$,
\[\mathbb{P}_{x\sim P_{\theta}}\left(|x-A'(\theta)|\ge (2+C)\sqrt{A''(\theta)}\sqrt{\ln(2/\beta)}\right)\le\beta.\]

}

\end{restatable}

Lemma~\ref{concentrationexp} shows that the tail of a distribution in an exponential family transitions from exponential to Gaussian as we move further out into the tail. How far into the tail one has to move for the tails to be Gaussian is a function of the standard deviation $\sqrt{A''(\theta)}$ and the stability of the standard deviation $\kappa(\theta)$. 

\subsection{Non-Private Estimation}

Before we start designing our differentially private locally optimal estimator, let us first discuss the locally optimal estimator in the non-private setting. 
Given a sample $X\sim P_{\theta}^n$, let \[\mathcal{A}_{opt}(X) = A'^{-1}\left(\frac{1}{n}\sum_{i=1}^n x_i\right) \, .\]

\begin{proposition}[Characterization of Optimal Local Estimation Rate in Non-Private Regime \citep{BN:1978}] 
\label{nonprivupper}
$\mathcal{A}_{opt}$ is the optimal non-private estimation algorithm and, for all measures $\mu$ and $\theta_0\in\thetarange$, has rate \[\error_n^{\rm loc}(P_{\theta}, \mathcal{P_{\mu}}, \mathcal{Q}_{\rm est}, \theta_0) = \Theta\left(\frac{1}{\sqrt{A''(\theta_0) n}}\right).\] 
\end{proposition}

\begin{restatable}{proposition}{rnonprivuniformachievability}{\em [Uniform Achievability in Non-Private Regime]}\label{nonprivuniformachievability}
For all measures $\mu$ and $\theta_0\in\thetarange$ and $n\in\mathbb{N}$, if $n\ge \frac{36}{\kappa(\theta_0)^2A''(\theta_0)}$, then \[\ierror{\mathcal{A}_{opt}}{n}{\theta_0} \le \frac{6}{\sqrt{nA''(\theta_0)}}.\] 
\end{restatable}

\begin{proof}[Proof of Proposition~\ref{nonprivuniformachievability}]
Note that $\mathbb{E}_{\theta_0}[\frac{1}{n}\sum_{i=1}^n x_i]=nA'(\theta_0)$ and $\var_{\theta_0}[\frac{1}{n}\sum_{i=1}^n x_i]=\frac{A''(\theta_0)}{n}$. Thus, with probability 8/9, 
\begin{equation}\label{meanacc}
\left|\frac{1}{n}\sum_{i=1}^n x_i-A'(\theta_0)\right|\le 3\sqrt{\frac{A''(\theta_0)}{n}}.
\end{equation}
Now, 
$\frac{6}{\sqrt{nA''(\theta_0)}}\le\kappa(\theta_0)$ implies that $3\sqrt{\frac{A''(\theta_0)}{n}}\le\frac{1}{2}A''(\theta_0)\kappa(\theta_0)$. So, by Equation~\ref{belongsinkappa}, if Equation~\eqref{meanacc} holds then $\mathcal{A}_{opt}(X)\in\thetarange(\theta_0)$.
Therefore, with probability 8/9-0.1,
\begin{align*}
\left|A'^{-1}\left(\frac{1}{n}\sum_{i=1}^n x_i\right)-\theta_0\right|&\le \max_{t\in[\frac{1}{n}\sum_{i=1}^n x_i, A'(\theta_0)]} (A'^{-1})'(t)\left|\frac{1}{n}\sum_{i=1}^n x_i-A'(\theta_0)\right|\\
&\le \max_{\theta'\in[A'^{-1}(\frac{1}{n}\sum_{i=1}^n x_i), \theta_0]} \frac{1}{A''(\theta')} 3\sqrt{\frac{A''(\theta_0)}{n}}\\
&\le 6\frac{1}{\sqrt{nA''(\theta_0)}}. \qedhere
\end{align*}
\end{proof}

\subsection{Initial Estimator}

In both the high and low privacy settings, our first step will be to a get a crude estimate of $\mathbb{E}_{\theta_0}[x]$.
This initial estimate will then be used to obtain a more refined estimate of $\theta_0$. In both cases a sufficient initial estimate is given by a slight variation of the mean estimator given in \cite{Karwa:2018}. Note that we could use this estimate of $A'(\theta_0)=\mathbb{E}_{\theta_0}[x]$ to get an estimate of $\theta_0$ in the same way we did in $\mathcal{A}_{opt}$. However, the resulting estimator of $\theta$ is suboptimal by a factor of $\sqrt{\ln n}$. A full description of the initial estimator is given in Appendix~\ref{appendix:meanest}, we will denote it by $\initial$.

\begin{restatable}{theorem}{rinitialestthm}
\label{initialestthm}
There exists constants $c>0$ and $b>0$ such that for all $\epsilon>0$, $\delta\in[0,1]$, $\thirdmoment>0$, and $C>0$, there exists an $(\epsilon, \delta)$-DP algorithm, $\initial$, such that for all measures $\mu$ and $\theta_0\in\thetarange$ if 
\begin{itemize}
\item $\frac{\mathbb{E}_{\theta_0}[|x-A'(\theta_0)|^3]}{\sqrt{A''(\theta_0)}^3}\le \thirdmoment$
\item $\kappa(\theta_0)\ge \frac{1}{C}\frac{1}{\sqrt{A''(\theta_0)}}$
\end{itemize}
then for all $n\in\mathbb{N}$ such that $n\ge \frac{c\thirdmoment^2\ln(1/\delta)}{\epsilon}$, if $X\sim P_{\theta_0}^n$, then with probability $0.8$,  \[\left|\initial(X)-\frac{1}{n}\sum_{x\in X} x\right| \le b(6+C)\left(\frac{\sqrt{A''(\theta_0)}}{n\epsilon}\sqrt{\ln(n)}\right). \]
\end{restatable}

\begin{restatable}{corollary}{rallforalog}
\label{allforalog}
There exists a constant $c>0$ such that for all $\epsilon_n=\Omega(\frac{\ln n}{n})$, $\delta\in[0,1]$, $\thirdmoment>0$, $C>0$ there exists an $(\epsilon_n, \delta)$-DP algorithm, $\initial$ and constants $D>0$ and $N\in\mathbb{N}$ such that for all measures $\mu$ and $\theta_0\in\thetarange$ if 
\begin{itemize}
\item $\frac{\mathbb{E}_{\theta_0}[|x-A'(\theta_0)|^3]}{\sqrt{A''(\theta_0)}^3}\le \thirdmoment$
\item $\kappa(\theta_0)\ge \frac{1}{C}\frac{1}{\sqrt{A''(\theta_0)}}$
\end{itemize}
then for all $n\in\mathbb{N}$ such that $n\ge \max\{N, \frac{c\thirdmoment^2\ln(1/\delta)}{\epsilon_n}\}$, with probability at least 0.8, \[|A'^{-1}(\initial(X))-\theta_0|\le D\left(\frac{1}{\sqrt{nA''(\theta_0)}}+\frac{1}{n\epsilon_n\sqrt{A''(\theta_0)}}\sqrt{\ln(n)}\right).\]
\end{restatable}

\begin{proof}
By Theorem~\ref{nonprivupper} and Theorem~\ref{initialestthm}, there exists a constant $C>0$ such that for all $\theta_0\in\thetarange$ satisfying the two conditions and $n\ge \frac{c\thirdmoment^2\ln(1/\delta)}{\epsn}$ with probability 0.8 we have, 
\[|\initial(X)-A'(\theta_0)| = \frac{C}{2}\left( \frac{\sqrt{A''(\theta_0)}}{\sqrt{n}}+ \frac{\sqrt{A''(\theta_0)}}{n\epsn}\sqrt{\ln\left(n\right)}\right).\] 
Now, since $\kappa(\theta_0)\sqrt{A''(\theta_0)}\ge B$ and $\epsn=\Omega(\frac{\ln n}{n})$, there exists $N\in\mathbb{N}$ such that for all $n>N$, \[\frac{C}{2}\left(\frac{\sqrt{A''(\theta_0)}}{\sqrt{n}}+\frac{\sqrt{A''(\theta_0)}\sqrt{\ln n}}{\epsn n}\right)\le \frac{1}{2}A''(\theta_0)\kappa(\theta_0)\] combined with Lemma~\ref{belongsinkappa} implies that $A'^{-1}(\initial(X))\in\thetarange(\theta_0)$.
Therefore, 
\begin{align*}
|(A')^{-1}(\initial(X))-\theta_0| &\le \max_{t\in[\initial(X), A'(\theta_0)]}(A')^{-1})'(t)|\initial(X)-\theta_0|\\
&= \max_{t\in[\initial(X), A'(\theta_0)]}\frac{1}{A''(A'^{-1}(t))}|\initial(X)-\theta_0|\\
&\le 2\frac{1}{A''(\theta_0)}\frac{C}{2}\left( \frac{\sqrt{A''(\theta_0)}}{\sqrt{n}}+ \frac{\sqrt{A''(\theta_0)}}{n\epsn}\sqrt{\ln\left(n\right)}\right)\\
&= C\left( \frac{1}{\sqrt{nA''(\theta_0)}}+ \frac{1}{n\epsn\sqrt{A''(\theta_0)}}\sqrt{\ln\left(n\right)}\right)
\end{align*}
\end{proof}

\subsection{High Privacy Regime}

We begin with the high privacy regime. While the noisy clamped log-likelihood ratio test is optimal in general for private simple hypothesis testing, a simpler test works in the high privacy regime. This test, which informs our design of the private estimator in this section, is a simple noisy counting test, and looks very similar to the optimal test in the local DP setting, presented in \cite{Duchi:2018}. The form of the estimation rate is also simpler in this section since, as we saw in Lemma~\ref{smalleps}, the sample complexity of the differentially private simple hypothesis testing takes on a simpler form in this regime. 

\subsubsection{Characterising the Optimal Local Estimation Rate in High Privacy Regime}

Recall from Corollary~\ref{smallepslower} that the optimal local estimation rate in the high privacy regime is characterized by the $L_1$-information, defined in Equation~\eqref{L1info}: for $\beta\in[0,1]$
\[J_{\TV, \theta}^{-1}(\beta) = \sup\{|h|\;|\; \TV(P_{\theta}, P_{\theta+h})\le\beta\}.\]
The following lemma characterizes the $L_1$-information, and hence the optimal local estimation rate, in terms of properties of the one-parameter exponential family. The proof can be found in Appendix~\ref{aboundonmodulus}.

\begin{restatable}{lemma}{rboundonmodulus}
\label{boundonmodulus} For all $\fourthmoment>0$, there exists a constant $C$ such that for all measures $\mu$, $\theta_0\in\thetarange$, and $\beta\in[0,1]$, if $ \frac{\mathbb{E}_{P_{\theta}}(X-A'(\theta_0))^4}{A''(\theta_0)^2}\le \fourthmoment\le \frac{9}{128\beta}$ and $\kappa(\theta)\ge J_{\TV,\theta}^{-1}(\beta)$ then, \[J_{\TV, \theta}^{-1}(\beta)\in \left[\frac{1}{\sqrt{2}} \frac{\beta}{\sqrt{A''(\theta)}}, C \frac{\beta}{\sqrt{A''(\theta)}}\right].\]
\end{restatable}

The following corollary follows immediately from Theorem~\ref{smallepslower} and Lemma~\ref{boundonmodulus}.

\begin{corollary}[Optimal Local Estimation Rate in the High Privacy Regime]\label{smallepslowerexp} For all constants $k$ and $\fourthmoment>0$, there exists a constants $C_1$, $C_2$ and $C_3$ such that for all measures $\mu$ and $\theta_0\in\thetarange$, if $ \frac{\mathbb{E}_{P_{\theta}}(X-A'(\theta_0))^4}{A''(\theta_0)^2}\le \fourthmoment$, $\epsilon_n\le \frac{k}{\sqrt{n}}$, $\kappa(\theta)\ge J^{-1}_{\TV,\theta}(\frac{C_3}{n\epsilon_n})$ and $\frac{C_3}{n\epsilon_n}\le  \frac{9}{128\fourthmoment}$ then for all $n\in\mathbb{N}$, \[\error^{\loc}_n(P_{\theta_0}, \mathcal{P_{\mu}}, \mathcal{Q}_{\rm est, \epsilon}, \theta) \in \left[\frac{C_1}{n\epsilon_n\sqrt{A''(\theta_0)}}, \frac{C_2}{n\epsilon_n\sqrt{A''(\theta_0)}}\right].\]
\end{corollary}

\subsubsection{Uniform Achievability in High Privacy Regime}

In this section we show that in the high privacy regime, uniform achievability is achieved using a simple estimator based on estimating $\mathbb{P}_{\theta}(x>\mathbb{E}_{\theta}[x])$.
Let $f_t(X) = \frac{1}{n}\sum_{i=1}^n \mathds{1}_{x_i>t}$. Note that by Lemma~\ref{monotonelikelihood},
\[g_t(\theta):=\mathbb{E}_{\theta}[f_t(x)]=\mathbb{P}_{\theta}(x>t)\]

is monotone and invertible in $\theta$, so Algorithm~\ref{algo:estexphighprivstep1} is well-defined. Our estimator $\mathcal{A}_{\rm high}$ requires as input a DP mean estimator, so $\hat{t}$ is an estimate of $\mathbb{E}_{\theta_0}[x]$. 
We refine the estimate $\hat{t}$ by using an estimator with lower sensitivity. Any sufficiently accurate mean estimator can be used for $\mathcal{M}$, but we note that the estimator described in Theorem~\ref{initialestthm} (derived from \cite{Karwa:2018}) is sufficient.

\begin{algorithm} \caption{$\secondstephigh$}\label{algo:estexphighprivstep1}
\begin{algorithmic}[1]
\Require{Sample $X \sim P_{\theta_0}^{n}$ , $\hat{t}$}, $\epsilon$
\State $\hat{\theta} = g_{\hat t}^{-1}(f_{\hat{t}}(X)+\Lap\left(\frac{1}{\epsilon n}\right))$.
\end{algorithmic}
\end{algorithm}

\begin{algorithm} \caption{$\anyAhigh$}\label{algo:estexphighprivstep2}
\begin{algorithmic}[1]
\Require{Sample $X_1 \sim P_{\theta_0}^{n}$ and $X_2 \sim P_{\theta_0}^{n}$, an $(\epsilon/2, \delta)$-DP mean estimator $\anyinitial$}
\State $\hat{t}=\anyinitial(X_1)$.
\State $\hat{\theta} = \secondstephigh(X_2,\hat{t}, \epsilon/2)$.
\end{algorithmic}
\end{algorithm}

\begin{restatable}[Uniform Achievability in High Privacy Regime]{proposition}{rprophighprivupper} \label{prop:highprivupper} For any $\epsn>0$ and $\deltan\in[0,1]$, $\mathcal{A}_{high}$ is $(\epsn,\deltan)$-DP.
Further, there exists a constant $c>0$ such that for all constants $\fourthmoment>0$, $C>0$, and $\delta\in[0,1]$, there exists an estimator $\initial$ such that there exists constants $N\in\mathbb{N}$ and $C'>0$ such that for all exponential families (i.e. any measure $\mu$) and $\theta_0\in\thetarange$, if 
\begin{itemize}
\item $\Omega\left(\frac{\ln n}{n}\right)\le \epsn\le O\left(\frac{1}{\sqrt{n}}\right)$
\item $\frac{\mathbb{E}_{P_{\theta_0}}(X-A'(\theta_0))^4}{A''(\theta_0)^2}\le \fourthmoment$,
\item $\kappa(\theta_0)\ge \frac{1}{C}\frac{1}{\sqrt{A''(\theta_0)}}$
\end{itemize}
then for all $n\in\mathbb{N}$ such that $\kappa(\theta_0)\ge J^{-1}_{\TV,\theta_0}(\frac{1}{\epsilon_n n})$ and $n\ge \max\{N,  \frac{c\fourthmoment^2\ln(1/\deltan)}{\epsn}\}$,
\[\ierror{\Ahigh}{n}{\theta_0} \le C'\cdot J_{\TV, \theta_0}^{-1}\left(\frac{1}{\epsn n}\right)=O\left(\frac{1}{n\epsn\sqrt{A''(\theta_0)}}\right),\]
where $\Ahigh$ is $\anyAhigh$ with initial mean estimator $\initial$.
\end{restatable}

Note that the upper bound in Proposition~\ref{prop:highprivupper} matches the characterization of the optimal local estimation rate given in Lemma~\ref{smallepslowerexp}. 
Thus, Proposition~\ref{prop:highprivupper} implies uniform achievability; there exists an algorithm that achieves the optimal local estimation rate for every $\theta_0\in\thetarange$. Note that many of the conditions required in Proposition~\ref{prop:highprivupper} were already present in Corollary~\ref{allforalog}. Indeed, we will primarily use these conditions to ensure that our initial estimate $\mathcal{M}(X_1)$ is sufficiently accurate. 

The proof proceeds by arguing that the test defined by the test statistic $f_{\hat{t}}(X)+\Lap\left(\frac{1}{\epsilon n}\right)$ is good enough to distinguish $P_{\theta_0}$ from $P_{\theta_1}$ provided $|\theta_0-\theta_1|\ge CJ_{\TV,\theta_0}^{-1}\left(\frac{1}{n\epsilon_n}\right)$, and thus the estimator inherited by this tester will, with high probability, not output such a $\theta_1$.
The main technical challenge in this section will be showing that $f_{\hat{t}}(X)+\Lap\left(\frac{1}{\epsilon n}\right)$ is a good test statistic for distinguishing between $P_{\theta_0}$ and $P_{\theta_1}$ when $|\theta_0-\theta_1|\ge CJ_{\TV,\theta_0}^{-1}\left(\frac{1}{n\epsilon_n}\right)$. We first show that if $\hat{t}$ is a good enough estimate for $A'(\theta_0)$, then $|\mathbb{P}_{\theta_0}(X>\hat{t})-\mathbb{P}_{\theta_1}(X>\hat{t})|\approx \TV(P_{\theta_0}, P_{\theta_1})$. Note that this would be obviously true if $\hat{t}=\frac{A(\theta_0)-A(\theta_1)}{\theta_0-\theta_1}$, so the majority of the work goes into proving that $\hat{t}$ is close enough to this ideal boundary point. Then, we show that the standard deviation of the statistic $f_{\hat{t}}(X)+\Lap\left(\frac{1}{\epsilon n}\right)\approx \frac{1}{\epsilon n}$, so the test will distinguish $P_{\theta_0}$ and $P_{\theta_1}$ provided $\TV(P_{\theta_0}, P_{\theta_1})=\Omega(\frac{1}{\epsilon n})$, as required. The following is the main technical lemma in this section, it is proved in Section~\ref{ahighprivexptest}.

\begin{restatable}{lemma}{rhighprivexptest}
\label{highprivexptest} For all positive constants $\fourthmoment$ and $b$, and $\Omega\left(\frac{\ln n}{n}\right)\le \epsn\le O\left(\frac{1}{\sqrt{n}}\right)$, there exists constants $N\in\mathbb{N}$ and $C>0$ such for all measures $\mu$, $\theta_0\in\thetarange$ and $n\in\mathbb{N}$ such that $n\ge N$, if 
\begin{enumerate}
\item $\max_{\theta\in\thetarange(\theta_0)}\frac{\mathbb{E}_{P_{\theta}}[(X-A'(\theta_0))^4]}{A''(\theta)^2}\le \fourthmoment$.
\item\label{initialacc} $|\hat{t}-A'(\theta_0)|\le \textcolor{black}{b\left(\sqrt{\frac{A''(\theta_0)}{n}}+\frac{\sqrt{A''(\theta_0)\ln\left(n\right)}}{\epsilon n}\right)}$
\item $\kappa(\theta_0)\ge J^{-1}_{\TV,\theta_0}\left(\frac{8\sqrt{2}}{\epsilon_n n}\right)$
\end{enumerate} 
then for all $\theta_1\in\thetarange(\theta_0)$ such that $|\theta_0-\theta_1|\ge CJ_{\TV, \theta_0}^{-1}\left(\frac{1}{\epsilon_nn}\right)$, then there exists a threshold $\tau$ such that the test \[\begin{cases}
P_{\theta_0} & \text{if } f_{\hat{t}}(X)+\Lap\left(\frac{1}{\epsilon n}\right)\le \tau\\
P_{\theta_1} & \text{if } f_{\hat{t}}(X)+\Lap\left(\frac{1}{\epsilon n}\right)\ge \tau
\end{cases}
\]
distinguishes between $\theta_0$ and $\theta_1$. Furthermore, $\mathbb{E}_{P_{\theta_0}}[f_{\hat{t}}(X)]\le \tau\le \mathbb{E}_{P_{\theta_1}}[f_{\hat{t}}(X)].$
\end{restatable}

The following corollary translates from the above testing result to an estimation result. The intuition for this conversion is that if the test statistic $f_{\hat{t}}(X)+\Lap\left(\frac{1}{\epsilon n}\right)$ can distinguish $\theta_0$ from $\theta_1$, then the estimation algorithm $\secondstephigh(X,\hat{t})$ is unlikely to output $\theta_1$ when the data is drawn from $P_{\theta_0}$. 

\begin{restatable}{corollary}{rhighprivoracleest}
\label{highprivoracleest} For all positive constants $\fourthmoment$ and $b$, and $\Omega\left(\frac{\ln n}{n}\right)\le \epsilon\le O\left(\frac{1}{\sqrt{n}}\right)$, there exists constants $N\in\mathbb{N}$ and $C>0$ such for all measures $\mu$, $\theta_0\in\thetarange$ and $n\in\mathbb{N}$ such that $n\ge N$, if 
\begin{enumerate}
\item $\max_{\theta\in\thetarange(\theta_0)}\frac{\mathbb{E}_{P_{\theta}}[(X-A'(\theta_0))^4]}{A''(\theta)^2}\le \fourthmoment$,
\item\label{initialacc} $|\hat{t}-A'(\theta_0)|\le \textcolor{black}{b\left(\sqrt{\frac{A''(\theta_0)}{n}}+\frac{\sqrt{A''(\theta_0)\ln\left(n\right)}}{\epsilon n}\right)}$,
\item $\kappa(\theta_0)\ge J^{-1}_{\TV, \theta_0}\left(\frac{1}{\epsilon_n n}\right)$,
\end{enumerate} 
then we have with probability 0.75,
\[|\secondstephigh(X,\hat{t})-\theta_0| \le C\cdot J_{\TV, \theta_0}^{-1}\left(\frac{1}{\epsilon_n n}\right)=O\left(\frac{1}{n\epsilon_n\sqrt{A''(\theta_0)}}\right).\]
\end{restatable}

\begin{proof} Let $C$ and $N$ be as in Lemma~\ref{highprivexptest}.
Let $X\sim P_{\theta_0}^n$ and suppose for sake of contradiction that $\mathcal{A}_{\text{high}}(X)~=~\theta_1$ where $|\theta_1-\theta_0|\ge CJ_{\TV, \theta_0}^{-1}\left(\frac{1}{\epsilon_nn}\right)$. Then by definition, $f_{\hat{t}}(X)+\Lap\left(\frac{1}{\epsilon n}\right) = \mathbb{E}_{Y\sim P_{\theta_1}}[f_{\hat{t}}(Y)]$. Therefore, the test in Lemma~\ref{highprivexptest} would have rejected $P_{\theta_0}$, which only happens with probability $0.25$.
\end{proof}

The proof of Proposition~\ref{prop:highprivupper} follows directly from Corollary~\ref{highprivoracleest} and Theorem~\ref{initialestthm}. Note that the conditions in Proposition~\ref{prop:highprivupper} are sufficient to ensure that the mean estimator from Theorem~\ref{initialestthm} is accurate enough that $\hat{t}$ satisfies the required condition in Corollary~\ref{highprivoracleest} with high probability.

\subsection{Low Privacy Regime}

In the low privacy regime, $\epsilon=\Omega\left(\frac{1}{\sqrt{n}}\right)$, we claim that the optimal local estimation rate for privacy is asymptotically the same as the non-private rate. The estimator that achieves the optimal local estimation rate is derived from the noisy clamped log-likelihood test outlined in Section~\ref{cllr}, the optimal algorithm for privately distinguishing two distributions.

\subsubsection{Characterising the Optimal Local Estimation Rate in Low Privacy Regime}

A main component of this claim is that for exponential families, the modulus of continuity of the non-private sample complexity $\SC$ is equal to the modulus of continuity of the private sample complexity, $\SC_{\epsilon_n}$ in this parameter regime. 
The proof of the following proposition is found in Section~\ref{alowprivmodcontinuity}.

\begin{restatable}{proposition}{rlowprivmodcontinuity}
\label{lowprivmodcontinuity} For all $\fourthmoment>0$, there exists positive constants $k$, $C_1$, $C_2$ and $N$ such that for all measures $\mu$, $\theta\in\thetarange$ and $n\in\mathbb{N}$ such that $n\ge N$, if 
\begin{itemize}
\item $\epsn\ge \frac{k}{\sqrt{n}}$, 
\item $\max_{\theta'\in\thetarange(\theta)}\frac{\mathbb{E}_{\theta'}((x-A'(\theta'))^4)}{A''(\theta')^2}\le \fourthmoment$  
\item $\kappa(\theta)\ge \frac{C_2}{\sqrt{n A''(\theta)}}$
\end{itemize} 
then, \[\frac{C_1}{\sqrt{nA''(\theta)}}\le\omega_n(P_{\theta}, \Qtesteps)\le \frac{C_2}{\sqrt{nA''(\theta)}}.\]
This implies $\omega_n(P_{\theta}, \Qtesteps) =\Theta\left(\omega_n(P_{\theta}, \Qtest)\right)$.
\end{restatable}

The following corollary follows immediately from Theorem~\ref{lowerbound} and Proposition~\ref{lowprivmodcontinuity}.

\begin{corollary}[Lower Bound in the Low Privacy Regime]\label{largeepslowerexp}
For all $\fourthmoment>0$, there exists positive constants $k$ and constants $C_1$, $C_2$ and $N$ such that under the same conditions as Proposition~\ref{lowprivmodcontinuity}, \[\error^{\loc}_n(P_{\theta_0}, \mathcal{P_{\mu}}, \Qesteps, \theta) \in \left[\frac{C_1}{2\sqrt{nA''(\theta_0)}}, \frac{C_2}{2\sqrt{nA''(\theta_0)}}\right]. \]
\end{corollary}

\subsubsection{Uniform Achievability in the Low Privacy Regime}

Our estimator in the low privacy regime is based on the $\NCLLR$ test described in Section~\ref{DPtestingsec}. Recall that the probability density function of $P_{\theta}$ has the form $p_{\theta}(x)=e^{\theta x-A(\theta)}$. Intuitively, the idea is that if we appropriately clamp the test statistic $x-\hat{t}$ then it contains roughly the same amount of information as the clamped log-likelihood ratio. Suppose $|\theta_1-\theta_0|=\omega_n(P_{\theta_0}, \Qtesteps)$ then
\begin{align*}
(\theta_1-\theta_0)[x-\hat{t}]_{-\epsilon/\omega_n(P_{\theta_0}, \Qtesteps)}^{\epsilon/\omega_n(P_{\theta_0}, \Qtesteps)} &= [(\theta_1-\theta_0)x-(\theta_1-\theta_0)\hat{t}]_{-\epsilon}^{\epsilon} \\
&= \left[\ln\frac{P_{\theta_1}(x)}{P_{\theta_0}(x)}+A(\theta_1)-A(\theta_0)-(\theta_1-\theta_0)\hat{t}\right]_{-\epsilon}^{\epsilon} \\
&\approx \left[\ln\frac{P_{\theta_1}(x)}{P_{\theta_0}(x)}\right]_{-\epsilon}^{\epsilon}
\end{align*}
where the final approximation holds since $\hat{t}\approx \mathbb{E}_{\theta_0}[x]=A'(\theta)$. The final term is the optimal test for distinguishing the two distributions. The main technical difficulty is in showing that the approximations do not affect the sample complexity of the test too much. As in the high privacy setting, we need an initial mean estimator to estimate $\hat{t}\approx\mathbb{E}_{\theta_0}[x]$ and $\tilde{\alpha}\approx \omega_n(P_{\theta_0}, \Qtesteps)$. Again the estimator from Theorem~\ref{initialestthm} will suffice. For ease of notation, let $\alpha_n(\theta) = \frac{1}{\sqrt{nA''(\theta)}}$.

\begin{algorithm} \caption{$\secondsteplow$}\label{algo:estexp}
\begin{algorithmic}[1]
\Require{Sample $X_1 = (x_1, \cdots, x_n)\sim P_{\theta_0}^n$, $\hat{t}$, C, $\epsilon$}
\State $\tilde{\alpha} = \alpha_n(A'^{-1}(\hat{t}))$. 
\State $f_{\tilde{\alpha}}(X_1) = \frac{1}{n}\sum_{i=1}^n [x_i-\hat{t}]_{\frac{-\epsilon}{C\tilde{\alpha}}}^{\frac{\epsilon}{C\tilde{\alpha}}}$
\State $\widehat{f} = f_{\tilde{\alpha}}(X_1) + \Lap\left(\frac{2}{\epsilon n C\tilde{\alpha}})\right)$
\State \Return $\hat{\theta} = \arg\min_{\theta} \left|\widehat{f}-\mathbb{E}_{X\sim P_{\theta}^n}\left(f_{\tilde{\alpha}}(X)\right)\right|$
\end{algorithmic}
\end{algorithm}

\begin{algorithm} \caption{$\anyAlow$, estimating exponential family parameters}\label{algo:estexp}
\begin{algorithmic}[1]
\Require{Sample $X_1\sim P_{\theta_0}^n$, $X_2\sim P_{\theta_0}^n$, and an $(\epsilon/2, \delta)$-DP mean estimator $\anyinitial$, C}
\State $\hat{t} = \anyinitial(X_2)$. 
\State \Return $\hat{\theta} = \secondsteplow(X_1, \hat{t}, C, \epsilon/2)$
\end{algorithmic}
\end{algorithm}

\begin{restatable}{proposition}{rmainexp}{\em [Uniform Achievability in Low Privacy Regime]}\label{mainexp}
For any $\epsn>0, \deltan\in[0,1], C>0$, $\mathcal{A}_{low}$ is $(\epsn, \deltan)$-DP. Further, for all $\fourthmoment>0$, $C>0$, $\epsn>0$, and $\deltan\in[0,1]$, there exists an initial estimator $\initial$ such that there exists constants $N\in\mathbb{N}$, $k\ge 0$, $C'>0$, $D>0$ such that for all exponential families (i.e. any measure $\mu$) and $\theta_0\in\thetarange$ if 
\begin{itemize}
\item $\epsn\ge\frac{k}{\sqrt{n}}$ 
\item $\frac{\mathbb{E}_{\theta}((x-\mathbb{E}_{\theta}(x))^4)}{A''(\theta)^2}\le \fourthmoment$, 
\item $\kappa(\theta)\ge \frac{1}{C}\frac{1}{\sqrt{A''(\theta)}}$
\end{itemize}
then for all $n\in\mathbb{N}$ such that $n\ge \max\left\{N, \frac{c\fourthmoment^2\ln(1/\deltan)}{\epsn}, \frac{2\ln\left(\frac{1}{(DA''(\theta_0))^2}\right)}{(DA''(\theta_0))^2}\right\}$ then we have 
\[\ierror{\Alow}{n}{\theta_0}\le C'\omega_n(P_{\theta_0}, \Qtesteps),\]
where $\Alow$ is $\anyAlow$ with initial mean estimator $\initial$.
\end{restatable}

The main technical part of our proof is the following lemma, whose proof we defer to Section~\ref{amainlemmaexp}.

\begin{restatable}{lemma}{rmainlemmaexp}
\label{mainlemmaexp} 
For all $\fourthmoment>0$, there exists constants $N\in\mathbb{N}$, $k\ge 0$, $C>0$, $b>0$ and $D>0$ such that for all measures $\mu$ and $\theta_0\in\thetarange$ if 
\begin{itemize}
\item $\epsn\ge\frac{k}{\sqrt{n}}$ 
\item $\frac{\mathbb{E}_{\theta}((x-\mathbb{E}_{\theta}(x))^4)}{A''(\theta)^2}\le \fourthmoment$, 
\item $|\theta_0-A'^{-1}(\hat{t})|\le \min\{\kappa(\theta_0), b\epsn\sqrt{nA''(\theta_0)}\}$, 
\end{itemize}
then for all $n\in\mathbb{N}$ such that $n\ge \max\{N, \frac{D}{A''(\theta_0)(\kappa(\theta_0))^2}\}$ and all $\theta_1$ such that $|\theta_1-\theta_0|\ge C\omega_n(P_{\theta_0}, \Qtesteps)$, there exists a threshold $ \tau$ such that the test
\begin{equation}\label{testworks}
\begin{cases}
P_{\theta_0} & \text{if } f_{\tilde{\alpha}}(X)+\Lap\left(\frac{2}{\epsn nC\tilde{\alpha}}\right) \le\tau\\
P_{\theta_1} & \text{if } f_{\tilde{\alpha}}(X)+\Lap\left(\frac{2}{\epsn n}C\tilde{\alpha}\right) \ge\tau
\end{cases}
\end{equation} distinguishes between $P_{\theta_0}$ and $P_{\theta_1}$ with $n$ samples. 
Furthermore, $\tau$ can be chosen so $|\mathbb{E}_{X\sim P_{\theta_0}^n}[\hat{f}_{\tilde{\alpha}}(X)]-\tau|\le |\mathbb{E}_{X\sim P_{\theta_1}^n}[\hat{f}_{\tilde{\alpha}}(X)]-\tau|$. 
\end{restatable}

As in the previous section, in the translation from testing result in Lemma~\ref{mainlemmaexp} to the bound on the estimation rate we argue that algorithm $\secondsteplow$ is unlikely to return $\theta$ such that $|\theta_0-\theta|\ge D\omega_{n, \SC_{\epsilon}}(\theta_0)$ since this would result in the induced test failing, which is unlikely to occur.

\begin{restatable}{corollary}{rmaincorlowpriv}\label{maincorlowpriv}
For all $\fourthmoment>0$, there exists constants $N\in\mathbb{N}$, $k\ge 0$, $C>0$, $b>0$ and $D>0$ such that for all measures $\mu$ and $\theta_0\in\thetarange$ if 
\begin{itemize}
\item $\epsn\ge\frac{k}{\sqrt{n}}$ 
\item $\frac{\mathbb{E}_{\theta}((x-\mathbb{E}_{\theta_0}(x))^4)}{A''(\theta_0)^2}\le \fourthmoment$, 
\item $|\theta_0-A'^{-1}(\hat{t})|\le \min\{\kappa(\theta_0), b\epsn\sqrt{nA''(\theta_0)}\}$, 
\end{itemize}
then for all $n\in\mathbb{N}$ such that $n\ge \max\{N, \frac{D}{A''(\theta_0)(\kappa(\theta_0))^2}\}$ we have with probability 0.75
\[|\secondsteplow(X,\hat{t}, C)-\theta_0|\le C\omega_n(P_{\theta_0}, \Qtesteps)\]
\end{restatable}

\begin{proof}[Proof of Theorem~\ref{mainexp}] 
First, we claim that the estimator from Theorem~\ref{initialestthm} is sufficient to satisfy the conditions of Lemma~\ref{mainlemmaexp}.
Suppose that $\mathcal{A}_{\text{low}}(X) = \theta_1$ and $|\theta_1-\theta_0|\ge C\omega_{n, \SC_{\epsilon}}(\theta_0)$. Then, $|\hat{f}-\mathbb{E}_{\theta_1}(f_{\tilde{\alpha}})| \le |\hat{f}-\mathbb{E}_{\theta_0}(f_{\tilde{\alpha}})|,$ which implies that the test in Lemma~\ref{mainlemmaexp} would have rejected $\theta_0$, which only happens with probability 0.25.
\end{proof}
Finally, as in the previous section, Proposition~\ref{mainexp} follows by showing that the estimator from Corollary~\ref{allforalog} satisfies the conditions of Corollary~\ref{maincorlowpriv}. 

\begin{proof}[Proof of Proposition~\ref{mainexp}] 

The proposition follows from a combination of Corollary~\ref{maincorlowpriv} and Corollary~\ref{allforalog}. Let $N, k , C_1, D_1$ and $D_2$ be as in Corollary~\ref{maincorlowpriv}. Note that since $\frac{D_2}{A''(\theta_0)(\kappa(\theta_0))^2}\le \frac{D_2}{B^2}$ we can assume that $N\ge \frac{D_2}{A''(\theta_0)(\kappa(\theta_0))^2}$. By Corollary~\ref{allforalog}, there exists a constant $C_2$ such that
\[|A'^{-1}(\initial(X))-\theta_0|\le C_2\left(\frac{1}{\sqrt{nA''(\theta_0)}}+\frac{1}{n\epsilon\sqrt{A''(\theta_0)}}\sqrt{\ln(n)}\right).\]
Again since $\kappa(\theta_0)\sqrt{A''(\theta_0)}\ge B$, there exists constants $N_1$ and $D$ such that for all $n>N_1$ if $\frac{\sqrt{\ln n}}{\sqrt{n}}\le D A''(\theta_0)$ then, \[C_2\left(\frac{1}{\sqrt{nA''(\theta_0)}}+\frac{1}{n\epsilon\sqrt{A''(\theta_0)}}\sqrt{\ln(n)}\right)\le \min\{\kappa(\theta_0), D_1\epsilon\sqrt{nA''(\theta_0)}.\] Thus, the estimator from Corollary~\ref{allforalog} satisfies the requirements of Corollary~\ref{maincorlowpriv} and so we are done.
\end{proof}

\section{Nonparametric Estimation of Functionals}
\label{sec:noparametric}

In the previous section, the statistic of interest fully characterised the distribution. In this section, we will study the problem of estimating a statistic that does not characterise the distribution. While one can still define the local estimation rate as in Equation~\eqref{localerrorrate} in this setting, we will follow the standard set by \cite{Donoho:1991} by focusing on a slightly different notion of local estimation rate and modulus of continuity. Given a family of distributions $\mathcal{P}$ and a statistic $\theta:\mathcal{P}\to\mathbb{R}$, we define the \emph{modulus of continuity with respect to $\theta$} at any value $t\in\mathbb{R}$ as 
\begin{align}
\nonumber\MOCallt{n}{t}{\Qtest}{\mathcal{P}}{\theta} &= \sup\left\{|\theta(P)-\theta(Q)|\;\Big| \; \QSC(P,Q)\ge n, \theta(P)=t, P, Q\in\mathcal{P}\right \}\\
&= \max_{P\in\mathcal{P}, \theta(P)=t} \MOCallt{n}{P}{\Qtest}{\mathcal{P}}{\theta}. \label{modulust}
\end{align}
The quantity is the worst case modulus of any distribution $P$ in the family $\mathcal{P}$ such that $\theta(P)=t$. \cite{Donoho:1991} showed that for some estimation problems, in the non-private setting, one can design an algorithm that can universally achieve error, $\max_{\theta(P)=t}\ierror{\est}{n}{P}\le \MOCallt{n}{t}{\Qest}{\mathcal{P}}{\theta}$.  

Rather than using of simple hypothesis tests, \cite{Donoho:1991} turn to the problem of distinguishing 
\begin{equation}
    \mathcal{P}_{\le t}=\{f|\theta(f)\le t\}\;\;\;\text{ and }\;\;\; \mathcal{P}_{\ge t+\Delta}=\{f|\theta(f)\ge t+\Delta\}.
    \label{eq:t-Delta-tests}
\end{equation}
Using this test, given the promise that $t\in[t_{\min},t_{\max}]$ and setting $\Delta=|t_{\max}-t_{\min}|/3$, we can rule out the true parameter lying in either $[t_{\min}, t_{\min}+\Delta]$ or $[t_{\max}-\Delta, t_{\max}]$, reducing the search space by a factor of $2/3$ for the next round. We will call this ternary search. 
If we run this algorithm for  $\lceil\frac{\log{\left(|t_1-t_0|/\omega\right)}}{\log(3/2)}\rceil$ steps, then the resulting error on the final estimate is at most $\omega$. The total sample complexity of the estimator is the sum of the sample complexities of the tests performed at each step. This is at most a logarithmic factor times the sample complexity of the most stringent (final) test; however, in many cases it is quite a bit lower than that, since the exponential decrease in $\Delta$ can mean that 
the sample complexity of the final test dominates the overall the sample complexity of the estimator.

\begin{algorithm}[t] \caption{Ternary Search for Functional Estimation, $\hat{\theta}$}\label{algo:BS}
\begin{algorithmic}[1]
\Require{Sample oracle for sample from $P$, $t_0<t_1$, number of rounds $k$}
\State $t_{\min} = t_0$, $t_{\max} = t_1$, and
$\Delta_1=|t_1-t_0|/3$
\For{$i\in[k]$}
\State Let $X_i$ be a sample from $P$ large enough to distinguish between $\mathcal{P}_{\le t_{\min}+\Delta_i}$ and $\mathcal{P}_{\ge t_{\min}+2\Delta_i}$ with probability $1-\frac{1}{3k}$.
\If{$T^*_{t_{\min}+\Delta_i,\Delta_i}(X_i)=\mathcal{P}_{\le t_{\min}+\Delta_i}$}
\State $t_{\max} = t_{\max}-\Delta_i$
\Else
\State $t_{\min} = t_{\min}+\Delta_i$
\EndIf
\State $\Delta_{i+1} = |t_{\max}-t_{\min}|/3$
\EndFor
\State \Return $t_{\min}$
\end{algorithmic}
\end{algorithm}

Algorithm~\ref{algo:BS} describes the algorithm that uses ternary search to estimate $\theta(P)$. In it, $T^*_{t, \Delta}$ denotes the optimal test for distinguishing the hypotheses in Eq.~\eqref{eq:t-Delta-tests}. The algorithm works for essentially any one-dimensional estimation problem. Under some conditions, we can compare its error on a given total sample size to the modulus of continuity (Equation~\ref{modulust} above).  The main condition is that  the difficulty of distinguishing between the two compound hypotheses in Eq.~\ref{eq:t-Delta-tests} should be captured by the difficulty of a \textit{simple} hypothesis test: that is, 
for $t\in[t_0,t_1]$ and $\Delta>0$, there should exist distributions $P_0^*$ and $P_1^*$ such that the number of samples needed to distinguish $\mathcal{P}_{\le t}$ and $\mathcal{P}_{\ge t+\Delta}$ is equal to the number of samples needed to distinguish between $P_0^*$ and $P_1^*$. Suppose we would like a test which competes with the modulus of continuity at some target sample size $\ntarget$. Given this condition, if we run Algorithm~\ref{algo:BS} for $\lceil\frac{\log{\left(\MOCallt{\ntarget}{t}{\Qtest}{\mathcal{P}}{\theta}/|t_1-t_0|\right)}}{\log(2/3)}\rceil$ rounds then it achieves error at most $ \MOCallt{\ntarget}{t}{\Qtest}{\mathcal{P}}{\theta}$. The final, most stringent test requires a sample of size $\ntarget$, by definition.\footnote{This discussion elides the dependency on the tests' error probability, which must be set sufficiently low to ensure that the decisions made at every round are correct; see Algorithm~\ref{algo:BS}.} The overall sample size is thus at most a logarithmic factor larger than $\ntarget$, and in some cases even closer to $\ntarget$.

Algorithm~\ref{algo:BS} can be adapted to the private setting by letting $T^*_{t,\Delta}$ be the optimal $\epsilon$-DP test. In the following section, we adapt an example of this framework from \cite{Donoho:1991} to the private setting. The key difficulty is  showing that the sample complexity of distinguishing $\mathcal{P}_{\le t}$ and $\mathcal{P}_{\ge t+\Delta}$ is characterised by the hardest simple test in the private setting.

\subsection{Tail Rates}\label{tailrates}

In this section, we will consider estimation of the tail decay rate of a certain class of distributions. We will show that we can use the framework of Algorithm~\ref{algo:BS} to design an algorithm for tail decay rate estimation that is instance optimal up to a logarithmic factor.
This section is a private analogue of Section 5 of \cite{Donoho:1991}. The fact that much of Donoho and Liu's argument can be immediately adapted to the private setting is largely due to the fact that the optimal private test is similar to the optimal non-private test, which the optimal local estimator is built from. However, \cite{Donoho:1991} are able to alter Algorithm~\ref{algo:BS} to eliminate the logarithmic factor increase in the sample complexity. This alteration can not be adapted to the private setting.
As a result, unlike in the non-private setting, where the resulting algorithm is instance optimal, our private analogue will be instance optimal up to a logarithmic factor. 

As in Donoho and Liu, rather than estimate the rate at which the tail of the density approaches 0 as $x\to\infty$, we will consider a transformation of the problem to observations $Y_i=1/X_i$. This leads us to estimating the rate at which a density approaches 0 as $x \to 0^+$.  Let $\mathcal{P}~=~\textbf{Tails}(C_-, C_+, \delta, t_0, t_1, \gamma, p)$ be the set of distributions defined on $[0,\infty)$ with densities satisfying:
\[f(x) = Cx^t(1+h(x))\;\; 0\le x\le\delta<1,\]
where
\begin{align*}
0< t_0\le t\le t_1<\infty, \hspace{0.1in}
0< C_-\le C\le C_+<\infty, \hspace{0.1in} \text{ and } \hspace{0.1in}
|h(x)|&\le \gamma x^p.
\end{align*}
For such a density function, the \emph{tail rate} is given by $\theta(f)=t$. The statistic $\theta(f)$ is one dimensional and lies in the bounded interval $[t_0,t_1]$. 

\begin{theorem}\label{maintailtheorem-overall} For every positive integer $n$: 
\begin{enumerate}
    \item For every $t \in [t_0,t_1]$ there exists a density $f$ with $\theta(f)=t$ such that every differentially private estimator has error at least $\MOCallt{n}{t}{\Qtest_{\epsilon}}{\mathcal{P}}{\theta}$ on some distribution in a neighborhood of $f$. That is, 
\[\errorlocall{n}{f}{\Qest_{\epsilon}}{\mathcal{P}}{\theta}\ge \frac{1}{2}\MOCallt{n}{t}{\Qtest_{\epsilon}}{\mathcal{P}}{\theta}\]

 \item There exists an $\varepsilon$-differentially private estimator $\hat{\theta}$ with the following property. For all $t\in [t_0,t_1]$, if $k^*(n)= \lceil\log_{3/2}(\frac{|t_1-t_0|}{ \MOCallt{n}{t}{\Qtest_{\epsilon}}{\mathcal{P}}{\theta}})\rceil$ and $N=n\cdot k^*(n)\cdot \lceil\log k^*(n)\rceil $, then $\hat{\theta}$ has error at most $\MOCallt{n}{t}{\Qtest_{\epsilon}}{\mathcal{P}}{\theta}$ when run on $N$ samples. That is,  \[\ierror{\hat{\theta}}{N}{f}\le \MOCallt{n}{t}{\Qtest_{\epsilon}}{\mathcal{P}}{\theta}\]

\end{enumerate} 
\end{theorem}

Typically, we expect $N$ is $O(n \cdot\log (n|t_1-t_0|)\cdot\log\log (n|t_1-t_0|))$ for $\epsilon \leq 1$ when\\ $\MOCallt{n/\log k^*(n)}{t}{\Qtest_{\epsilon}}{\mathcal{P}}{\theta}~\approx~1/poly(n)$.
As in Donoho et al., the first aspect of the proof is to show that the corresponding testing problem has the property that the difficulty of distinguishing between two intervals in $[t_0,t_1]$ is captured by the difficulty of the hardest two-point testing problem. That is, for $t\in[t_0,t_1]$ and $\Delta>0$, there exists distributions $f_0^*$ and $f_1^*$ such that the number of samples needed to distinguish $\mathcal{P}_{\le t}$ and $\mathcal{P}_{\ge t+\Delta}$ is equal to the number of samples needed to distinguish between $f_0^*$ and $f_1^*$. \cite{Donoho:1991} showed that in the non-private case, the distributions $f_0^*$ and $f_1^*$ satisfy the following conditions:
\begin{align*}
f_0^*(x) &= C_-x^t(1-\gamma x^p), \;\;\; x\le a_1(t,\Delta),\\
f_1^*(x) &= C_+x^{t+\Delta}(1+\gamma x^p), \;\;\; x\le a_1(t,\Delta),
\end{align*}
and 
\[\frac{f_0^*(x)}{f_1^*(x)}=\frac{f_0^*(a_1(t,\Delta))}{f_1^*(a_1(t,\Delta))}, \;\;\; x>a_1(t,\Delta),\] and \[\frac{f_0^*(a_1(t,\Delta))}{f_1^*(a_1(t,\Delta))} = \frac{1-\int_0^{a_1(t,\Delta)}f_0^*(v)dv}{1-\int_0^{a_1(t,\Delta)}f_1^*(v)dv}.\]
In the following lemma we mirror their proof to show that the same distributions also satisfy this condition in the private case.

We say a real random variable $X$ is stochastically less than a random variable $Y$, denoted $X\preceq Y$, if \[\mathbb{P}(X>x)\le \mathbb{P}(Y>x) \;\;\text{ for all } x\in(-\infty, \infty).\]

\begin{lemma}\label{stochord}
If $u$ is non-decreasing and $X\preceq Y$ then $u(X)\preceq u(Y)$. If $X_i\preceq Y_i$ for $i=1,\cdots, n$ then $\sum_{i=1}^n X_i\preceq \sum_{i=1}^n Y_i$.
\end{lemma}

\begin{lemma}\label{worstcasedists}
For the distribution $f_0^*$ and $f_1^*$ described above, we have \[\SC_{\epsilon}(\mathcal{P}_{\le t}, \mathcal{P}_{\ge t+\Delta}) = \SC_{\epsilon}(f_0^*, f_1^*).\] Furthermore, the test statistic for distinguishing $\mathcal{P}_{\le t}$ and $\mathcal{P}_{\ge t+\Delta}$ is the clamped log-likelihood ratio between $f_0^*$ and $f_1^*$. 
\end{lemma}

\begin{proof} To simplify notation, we will let $a=a_1(t,\Delta)$.
Recall from Proposition~\ref{robustthresholds}, the optimal test statistic for distinguishing between $f_0^*$ and $f_1^*$ is given by the noisy clamped log-likelihood ratio. Now, let \[L_{t,\Delta}(x) = \begin{cases}
\ln\left(\frac{C_+}{C_-}x^{\Delta}\frac{1+\gamma x^p}{1-\gamma x^p}\right), & 0<x<a,\\
\ln\left(\frac{C_+}{C_-}a^{\Delta}\frac{1+\gamma a^p}{1-\gamma a^p}\right), & x\ge a,\\
\end{cases}\]
be the log likelihood ratio and let $L^{\epsilon}_{t,\Delta}(x)=[L_{t,\Delta}(x)]_{-\epsilon}^{\epsilon}$ be the clamped log likelihood ratio (where $L_{t,\Delta}(x)$ is projected onto the domain $[-\epsilon,\epsilon]$). \cite{Donoho:1991}, showed that the distribution of $L_{t,\Delta}(x)$ where $x$ is distributed according to $f_0^*$ is stochastically larger than $L_{t,\Delta}(x)$ where $x$ is distributed according to $f$ for any $f\in\mathcal{P}_{\le t}$. Since clamping is a monotone function, this implies that the same property holds for $L^{\epsilon}_{t,\Delta}$. 
Now, for $X=(x_1, \cdots, x_n)$, \[L^{\epsilon}_{n,t,\Delta}(X) =  \sum_{i=1}^n L^{\epsilon}_{t,\Delta}(x_i)\] and $\widetilde{L^{\epsilon}_{n,t,\Delta}}(X) = L^{\epsilon}_{n,t,\Delta}(X)+\Lap(2)$. Again by Lemma \ref{stochord}, we have that among all the distributions in $\mathcal{P}_{\le t}$, $\widetilde{L^{\epsilon}_{n,t,\Delta}}(X)$ is stochastically largest at $f_0^*$. Similarly, we can show that among all the distributions in $\mathcal{P}_{\ge t+\Delta}$, $\widetilde{L^{\epsilon}_{n,t,\Delta}}(X)$ is stochastically smallest at $f_1^*$. Therefore, if we design a test which \texttt{accepts} (corresponding to choosing $\mathcal{P}_{\le t}$) if $\widetilde{L_{n,t,\Delta}}(X)\le 1$ and \texttt{rejects} otherwise. Then we must have 
\[\sup_{f\in \mathcal{P}_{\le t}}\mathbb{P}_f(\texttt{rejects}) = \mathbb{P}_{f_0^*}(\texttt{rejects}),\text{ and }
\sup_{f\in \mathcal{P}_{\ge t+\Delta}}\mathbb{P}_f(\texttt{accepts}) = \mathbb{P}_{f_1^*}(\texttt{accepts}).\]
It follows that if $n\ge \SC_{\epsilon}(f_0^*, f_1^*)$ then $ \mathbb{P}_{f_0^*}(\texttt{rejects})\le 1/3$ and $ \mathbb{P}_{f_1^*}(\texttt{accepts})\le 1/3$ since the clamped log-likelihood test is optimal for distinguishing $f_0^*$ and $f_1^*$. Thus this distinguishes $\mathcal{P}_{\le t}$ and $\mathcal{P}_{\ge t+\Delta}$ with the same number of samples.
\end{proof}

Lemma~\ref{worstcasedists} gives us the tools we need to use Algorithm~\ref{algo:BS} to obtain a near instance optimal (up to logarithmic factors) differentially private algorithm for tail bound estimation. 

\begin{restatable}{theorem}{rmaintailtheorem}\label{maintailtheorem} 
For all $n\in\mathbb{N}$, let $k^*(n)= \lceil\log_{3/2}(\frac{|t_1-t_0|}{ \MOCallt{n}{t}{\Qtest_{\epsilon}}{\mathcal{P}}{\theta}})\rceil$, $N=n\cdot k^*(n)\cdot \lceil\log k^*(n)\rceil $
and $\hat{\theta}$ be the output of Algorithm~\ref{algo:BS} run for $k^*(n)$ rounds. Then \[\ierror{\hat{\theta}}{N}{f}\le \MOCallt{n}{t}{\Qtest_{\epsilon}}{\mathcal{P}}{\theta}.\]
\end{restatable}

\begin{proof}[Proof of Theorem~\ref{maintailtheorem}] We can think of Algorithm~\ref{algo:BS} as at each step dividing the distance between $t_{\min}$ and $t_{\max}$ by $2/3$ and concluding that the true value $t^*$ lies between $t_{\min}$ and $t_{\max}$. Thus, in order to show that $\ierror{\hat{\theta}}{N}{f}\le \MOCallt{n}{t}{\Qtest_{\epsilon}}{\mathcal{P}}{\theta}$, it suffices to show that in order to run for $k^*(n)=\left\lceil\log_{\frac{3}{2}}\left  (\frac{|t_1-t_0|}{\MOCallt{n}{t}{\Qtest_{\epsilon}}{\mathcal{P}}{\theta}}\right)\right\rceil$ iterations, it suffices to have at least $N=n\cdot \lceil \log k^*(n)\rceil \cdot k^*(n)$ samples. In order to make the correct decision with probability $1/3k^*(n)$, it suffices for the last iteration to use $n\cdot \lceil \log k^*(n)\rceil$ samples. Since the hypothesis test at the last step has the largest sample size, $n\cdot\lceil \log k^*(n)\rceil\cdot k^*(n)$ samples is sufficient to run $k^*(n)$ rounds.
\end{proof}

\cite{Donoho:1991} are able to alter Algorithm~\ref{algo:BS} to remove the logarithmic factor.
Given $\Delta>0$, they define the following estimator 
\begin{equation}\label{donohoestimator}
\theta^*_{n, \Delta}(X) = \frac{\Delta}{2}+\sup\{t\;|\; L_{n,t,\Delta}(X)\le0\},
\end{equation} which outputs the largest $t$ for such the hypothesis $\mathcal{P}_{\le t}$ would be accepted. This estimator is well-defined since $L_{n,t,\Delta}$ is deterministic, a crucial distinction when we move to the private setting. Donoho et al. show that if $\Delta$ is sufficiently small then $L_{n,t,\Delta}$ is monotonically decreasing in $t$ for a given $x$, which implies that given input distribution $f$ such that $\theta(f)=t$, the estimation algorithm $T^*_{n, \Delta}$ has error rate $\MOCallt{n}{t}{\Qtest_{\epsilon}}{\mathcal{P}}{\theta}.$

The estimator $\theta^*_{n, \Delta}(X)$ in eqn~\eqref{donohoestimator} can be viewed as performing the test $L_{n,t,\Delta}(X)$ on every $t$ value and outputting the threshold point, where the test flips from \texttt{accept} to \texttt{reject}. We can not replicate this directly in the private setting both because the private test is stochastic (so there is likely to be some false negatives and false positives), and because performing the test on every $t$ value would result in an unreasonably large privacy cost.

\addcontentsline{toc}{section}{References}
\section*{Acknowledgments}

We thank Clément Canonne and John Duchi for helpful conversations and comments. This work was started while the authors were visiting the Simons Institute for the Theory of Computing. Part of this work was done while AM was at Boston University and Northeastern University, where she was supported by BU's Hariri Institute for Computing, NSF award CCF-1763786, and Northeastern's Cybersecurity and Privacy Institute.  
JU's work on this project was supported by NSF awards CCF-1750640 and CNS-2120603. Part of this work was done while JU was visiting Apple. 
AS was supported in part by NSF award CCF-1763786 and a Sloan Foundation Research Award.

\addcontentsline{toc}{section}{References}

\bibliographystyle{abbrvnat}

\bibliography{refs}

\begin{thebibliography}{39}
\providecommand{\natexlab}[1]{#1}
\providecommand{\url}[1]{\texttt{#1}}
\expandafter\ifx\csname urlstyle\endcsname\relax
  \providecommand{\doi}[1]{doi: #1}\else
  \providecommand{\doi}{doi: \begingroup \urlstyle{rm}\Url}\fi

\bibitem[Acharya et~al.(2018)Acharya, Sun, and Zhang]{Acharya:2018}
J.~Acharya, Z.~Sun, and H.~Zhang.
\newblock Differentially private testing of identity and closeness of discrete
  distributions.
\newblock In \emph{Proceedings of the 32nd International Conference on Neural
  Information Processing Systems}, NIPS'18, page 6879–6891, Red Hook, NY,
  USA, 2018. Curran Associates Inc.

\bibitem[Alon et~al.(2019)Alon, Livni, Malliaris, and Moran]{AlonLMM19}
N.~Alon, R.~Livni, M.~Malliaris, and S.~Moran.
\newblock Private pac learning implies finite littlestone dimension.
\newblock In \emph{ACM Symposium on Theory of Computing}, STOC '19, pages
  852--860, 2019.

\bibitem[{Apple Differential Privacy Team}(2017)]{Apple17}
{Apple Differential Privacy Team}.
\newblock Learning with privacy at scale.
\newblock \emph{Apple Machine Learning Journal}, 1\penalty0 (8), 2017.
\newblock
  \url{https://machinelearning.apple.com/docs/learning-with-privacy-at-scale/appledifferentialprivacysystem.pdf}.

\bibitem[Asi and Duchi(2020)]{Asi:2020}
H.~Asi and J.~C. Duchi.
\newblock Near instance-optimality in differential privacy.
\newblock \emph{arXiv preprint arXiv:2005.10630}, 2020.

\bibitem[Balle et~al.(2020)Balle, Barthe, and Gaboardi]{Balle:2020}
B.~Balle, G.~Barthe, and M.~Gaboardi.
\newblock Privacy profiles and amplification by subsampling.
\newblock \emph{Journal of Privacy and Confidentiality}, 10\penalty0 (1), Jan.
  2020.

\bibitem[Barber and Duchi(2014)]{DuchiF14}
R.~F. Barber and J.~C. Duchi.
\newblock Privacy and statistical risk: Formalisms and minimax bounds.
\newblock \emph{arXiv preprint arXiv:1412.4451}, 2014.

\bibitem[Barndorff-Nielsen(1978)]{BN:1978}
O.~Barndorff-Nielsen.
\newblock \emph{Introductory Theory of Exponential Families}, chapter~8.
\newblock John Wiley \& Sons, Ltd, 1978.
\newblock URL
  \url{https://onlinelibrary.wiley.com/doi/abs/10.1002/9781118857281.ch8}.

\bibitem[Beimel et~al.(2008)Beimel, Nissim, and Omri]{BeimelNO08}
A.~Beimel, K.~Nissim, and E.~Omri.
\newblock Distributed private data analysis: Simultaneously solving how and
  what.
\newblock In \emph{International Cryptology Conference}, CRYPTO '08, pages
  451--468, Santa Barbara, CA, USA, 2008.
\newblock \url{https://arxiv.org/abs/1103.2626}.

\bibitem[Blocki et~al.(2013)Blocki, Blum, Datta, and Sheffet]{BlockiBDS13}
J.~Blocki, A.~Blum, A.~Datta, and O.~Sheffet.
\newblock Differentially private data analysis of social networks via
  restricted sensitivity.
\newblock In \emph{4th ACM Conference on Innovations in Theoretical Computer
  Science}, ITCS '13, pages 87--96, Berkeley, CA, USA, 2013. ACM.

\bibitem[Bun and Steinke(2016)]{Bun:CDP}
M.~Bun and T.~Steinke.
\newblock Concentrated differential privacy: Simplifications, extensions, and
  lower bounds.
\newblock volume 9985, pages 635--658, 11 2016.
\newblock ISBN 978-3-662-53640-7.

\bibitem[Bun et~al.(2014)Bun, Ullman, and Vadhan]{BunUV14}
M.~Bun, J.~Ullman, and S.~Vadhan.
\newblock Fingerprinting codes and the price of approximate differential
  privacy.
\newblock In \emph{ACM Symposium on the Theory of Computing}, STOC '14, pages
  1--10, New York, NY, USA, 2014.
\newblock \url{https://arxiv.org/abs/1311.3158}.

\bibitem[Bun et~al.(2015)Bun, Nissim, and Stemmer]{Bun:2016}
M.~Bun, K.~Nissim, and U.~Stemmer.
\newblock Simultaneous private learning of multiple concepts.
\newblock 11 2015.
\newblock \doi{10.1145/2840728.2840747}.

\bibitem[Canonne et~al.(2019)Canonne, Kamath, McMillan, Smith, and
  Ullman]{Canonne:2019}
C.~Canonne, G.~Kamath, A.~McMillan, A.~Smith, and J.~Ullman.
\newblock The structure of optimal private tests for simple hypotheses.
\newblock In \emph{51st Annual ACM Symposium on Theory of Computing}, pages
  310--321, 06 2019.
\newblock \doi{10.1145/3313276.3316336}.

\bibitem[Chan et~al.(2011)Chan, Shi, and Song]{ChanSS11}
T.-H.~H. Chan, E.~Shi, and D.~Song.
\newblock Private and continual release of statistics.
\newblock \emph{ACM Transactions on Information and System Security (TISSEC)},
  14\penalty0 (3):\penalty0 26, 2011.

\bibitem[Chen and Zhou(2013)]{Chen:2013}
S.~Chen and S.~Zhou.
\newblock Recursive mechanism: Towards node differential privacy and
  unrestricted joins.
\newblock In \emph{Proceedings of the 2013 ACM SIGMOD International Conference
  on Management of Data}, SIGMOD '13, page 653–664, New York, NY, USA, 2013.
  Association for Computing Machinery.
\newblock ISBN 9781450320375.

\bibitem[Donoho and Liu(1987)]{Donoho:1987}
D.~L. Donoho and R.~C. Liu.
\newblock Geometrizing rates of convergence, {I}.
\newblock \emph{Technical Report 137}, 1987.

\bibitem[Donoho and Liu(1991)]{Donoho:1991}
D.~L. Donoho and R.~C. Liu.
\newblock Geometrizing rates of convergence, {II}.
\newblock \emph{Ann. Statist.}, 19\penalty0 (2):\penalty0 633--667, 06 1991.
\newblock \doi{10.1214/aos/1176348114}.
\newblock URL \url{https://doi.org/10.1214/aos/1176348114}.

\bibitem[Duchi and Ruan(2018)]{Duchi:2018}
J.~Duchi and F.~Ruan.
\newblock The right complexity measure in locally private estimation: It is not
  the fisher information.
\newblock \emph{arXiv:1806.05756}, 06 2018.

\bibitem[Duchi et~al.(2013)Duchi, Jordan, and Wainwright]{DuchiJW13}
J.~Duchi, M.~Jordan, and M.~Wainwright.
\newblock Local privacy and statistical minimax rates.
\newblock In \emph{IEEE Symposium on Foundations of Computer Science}, FOCS
  '13, pages 429--438, Berkeley, CA, USA, 2013.
\newblock \url{https://arxiv.org/abs/1302.3203}.

\bibitem[{Duchi} et~al.(2013){Duchi}, {Jordan}, and {Wainwright}]{Duchi:2014}
J.~C. {Duchi}, M.~I. {Jordan}, and M.~J. {Wainwright}.
\newblock Local privacy and statistical minimax rates.
\newblock In \emph{2013 IEEE 54th Annual Symposium on Foundations of Computer
  Science}, pages 429--438, 2013.

\bibitem[Dwork and Lei(2009)]{DworkL09}
C.~Dwork and J.~Lei.
\newblock Differential privacy and robust statistics.
\newblock In \emph{Proceedings of the 41st ACM Symposium on Theory of
  Computing}, STOC '09, pages 371--380. ACM, 2009.

\bibitem[Dwork and Rothblum(2016)]{Dwork2016ConcentratedDP}
C.~Dwork and G.~N. Rothblum.
\newblock Concentrated differential privacy.
\newblock \emph{ArXiv}, abs/1603.01887, 2016.

\bibitem[Dwork et~al.(2006)Dwork, McSherry, Nissim, and Smith]{DworkMNS06}
C.~Dwork, F.~McSherry, K.~Nissim, and A.~Smith.
\newblock Calibrating noise to sensitivity in private data analysis.
\newblock In \emph{Conference on Theory of Cryptography}, TCC '06, pages
  265--284, New York, NY, USA, 2006.

\bibitem[Dwork et~al.(2015)Dwork, Smith, Steinke, Ullman, and
  Vadhan]{DworkSSUV15}
C.~Dwork, A.~Smith, T.~Steinke, J.~Ullman, and S.~Vadhan.
\newblock Robust traceability from trace amounts.
\newblock In \emph{IEEE Symposium on Foundations of Computer Science}, FOCS
  '15, 2015.

\bibitem[Edmonds et~al.(2020)Edmonds, Nikolov, and Ullman]{EdmondsNU20}
A.~Edmonds, A.~Nikolov, and J.~Ullman.
\newblock The power of factorization meisms in local and central differential
  privacy.
\newblock In \emph{ACM Symposium on the Theory of Computing}, STOC '20, pages
  425--438, Chicago, IL, USA, 2020.
\newblock \url{https://arxiv.org/abs/1911.08339}.

\bibitem[Kamath and Ullman(2020)]{KamathU20}
G.~Kamath and J.~Ullman.
\newblock A primer on private statistics.
\newblock \emph{arXiv preprint arXiv:2005.00010}, 2020.

\bibitem[Kamath et~al.(2019)Kamath, Li, Singhal, and Ullman]{KamathLSU19}
G.~Kamath, J.~Li, V.~Singhal, and J.~Ullman.
\newblock Privately learning high-dimensional distributions.
\newblock In \emph{Annual Conference on Learning Theory}, COLT '19. JMLR.org,
  2019.

\bibitem[Kamath et~al.(2020)Kamath, Singhal, and Ullman]{KamathSU20}
G.~Kamath, V.~Singhal, and J.~Ullman.
\newblock Private mean estimation of heavy-tailed distributions.
\newblock \emph{https://arxiv.org/abs/2002.09464}, 2020.

\bibitem[Karwa and Vadhan(2018)]{Karwa:2018}
V.~Karwa and S.~P. Vadhan.
\newblock Finite sample differentially private confidence intervals.
\newblock volume abs/1711.03908 of \emph{Innovations in Theoretical Computer
  Science '18}, 2018.

\bibitem[Kasiviswanathan and Smith(2008)]{KasiviswanathanS08}
S.~P. Kasiviswanathan and A.~D. Smith.
\newblock On the `semantics' of differential privacy: A bayesian formulation.
\newblock \emph{CoRR}, abs/0803.3946, 2008.

\bibitem[Kasiviswanathan et~al.(2008)Kasiviswanathan, Lee, Nissim,
  Raskhodnikova, and Smith]{KasiviswanathanLNRS08}
S.~P. Kasiviswanathan, H.~K. Lee, K.~Nissim, S.~Raskhodnikova, and A.~Smith.
\newblock What can we learn privately?
\newblock In \emph{IEEE Symposium on Foundations of Computer Science}, FOCS
  '08, pages 531--540, Philadelphia, PA, USA, 2008.
\newblock \url{https://arxiv.org/abs/0803.0924}.

\bibitem[Kasiviswanathan et~al.(2013)Kasiviswanathan, Nissim, Raskhodnikova,
  and Smith]{KasiviswanathanNRS13}
S.~P. Kasiviswanathan, K.~Nissim, S.~Raskhodnikova, and A.~D. Smith.
\newblock Analyzing graphs with node differential privacy.
\newblock In \emph{10th IACR Theory of Cryptography Conference}, TCC '13, pages
  457--476, Tokyo, Japan, 2013. Springer.

\bibitem[McMahan and Thakurta(2022)]{Google-FL-blog}
B.~McMahan and A.~Thakurta.
\newblock {Blog post: Federated Learning with Formal Differential Privacy
  Guarantees}, Feb. 2022.
\newblock
  \url{https://ai.googleblog.com/2022/02/federated-learning-with-formal.html}.
  Last accessed: June 2022.

\bibitem[Narayanan(2022)]{Narayanan2022PrivateHH}
S.~Narayanan.
\newblock Private high-dimensional hypothesis testing.
\newblock \emph{ArXiv}, abs/2203.01537, 2022.

\bibitem[Nissim et~al.(2007)Nissim, Raskhodnikova, and Smith]{NissimRS07}
K.~Nissim, S.~Raskhodnikova, and A.~Smith.
\newblock Smooth sensitivity and sampling in private data analysis.
\newblock In \emph{Proceedings of the 30th annual ACM Symposium on Theory of
  Computing, STOC}, pages 75--84, 2007.

\bibitem[Rohde and Steinberger(2018)]{Rohde:2019}
A.~Rohde and L.~Steinberger.
\newblock Geometrizing rates of convergence under local differential privacy
  constraints.
\newblock \emph{arXiv:1805.01422}, 2018.

\bibitem[Smith(2011)]{Smith11}
A.~Smith.
\newblock Privacy-preserving statistical estimation with optimal convergence
  rates.
\newblock In \emph{Proceedings of the 43rd Annual ACM Symposium on the Theory
  of Computing}, STOC '11, pages 813--822, New York, NY, USA, 2011. ACM.

\bibitem[Steinke and Ullman(2017)]{SteinkeU17}
T.~Steinke and J.~Ullman.
\newblock Tight lower bounds for differentially private selection.
\newblock In \emph{IEEE Symposium on Foundations of Computer Science}, FOCS
  '17, 2017.

\bibitem[Vadhan(2017)]{Vadhan:2016}
S.~Vadhan.
\newblock \emph{The Complexity of Differential Privacy}, pages 347--450.
\newblock 04 2017.
\newblock ISBN 978-3-319-57047-1.

\end{thebibliography}

\addcontentsline{toc}{section}{Appendix}
\addtocontents{toc}{\setcounter{tocdepth}{-1}}
\appendix

\newpage

\section{Proofs for Section~\ref{DPtestingsec}}\label{DPtestingappendix}

\subsection{Proof of Proposition~\ref{robustthresholds}}\label{arobustthresholds}

\restaterobustthresholds*

As in \cite{Canonne:2019}, define \[\tau = \tau(P,Q) = \max\left\{ \int \max\{P(x)-e^{\epsilon}Q(x), 0\}dx,  \int \max\{Q(x)-e^{\epsilon}P(x), 0\}dx\right\}\] and assume without loss of generality that $\tau = \int \max\{P(x)-e^{\epsilon}Q(x), 0\}dx$. Let $0\le\epsilon'\le\epsilon$ be the smallest value such that $\tau = \int \max\{Q(x)-e^{\epsilon'}P(x), 0\}dx$.
Define $P'=\frac{1}{1-\tau}\min\{P,e^{\epsilon} Q\}$ and $Q'=\frac{1}{1-\tau}\min\{Q,e^{\epsilon'} P\}$.

\begin{lemma} \label{getridofepsprime} For any $\epsilon\in[0, 1]$, and distributions $P$ and $Q$ with the same support,
\[\SC_{\NCLLR_{-\epsilon'}^{\epsilon}}(P,Q)=\Theta(\SC_{\NCLLR_{-\epsilon}^{\epsilon}}(P,Q)).\]
\end{lemma}

\begin{proof} First note that since $\NCLLR_{-\epsilon'}^{\epsilon}$ is an optimal test up to a constant factor \citep{Canonne:2019}, \[\SC_{\NCLLR_{-\epsilon'}^{\epsilon}}(P,Q)=O(\SC_{\NCLLR_{-\epsilon}^{\epsilon}}(P,Q)).\]

\cite{Canonne:2019} show that the following two inequalities are sufficient to prove that $\SC_{\NCLLR_{-\epsilon'}^{\epsilon}}(P,Q) = \Theta\left(\frac{1}{\tau\epsilon+(1-\tau)H^2(P',Q')}\right)$:
\[\mathbb{E}_{P^n}[\CLLR_{-\epsilon'}^{\epsilon}]-\mathbb{E}_{Q^n}[\CLLR_{-\epsilon'}^{\epsilon}]\ge \Omega(n(\tau\epsilon+(1-\tau)H^2(P',Q')))\]
and
\[\max\left\{\mathbb{E}_{P}\left[\left[\ln\frac{P(x)}{Q(x)}\right]_{-\epsilon'}^{\epsilon}\right]^2, \mathbb{E}_{Q}\left[\left[\ln\frac{P(x)}{Q(x)}\right]_{-\epsilon'}^{\epsilon}\right]^2\right\}\le O(\tau\epsilon+(1-\tau)H^2(P',Q'))\]

We first note that the gap between the expectations increases when we move from $\CLLR_{-\epsilon'}^{\epsilon}$ to $\CLLR_{-\epsilon}^{\epsilon}$.
If $P(x)<Q(x)$, then $\CLLR_{-\epsilon}^{\epsilon}(x)<\CLLR_{-\epsilon'}^{\epsilon}(x)$, and otherwise, $\CLLR_{-\epsilon}^{\epsilon}(x)=\CLLR_{-\epsilon'}^{\epsilon}(x)$. Thus, $(P(x)-Q(x))\CLLR_{-\epsilon}^{\epsilon}(x)\ge (P(x)-Q(x))\CLLR_{-\epsilon'}^{\epsilon}(x)$ and
\begin{align*}
\mathbb{E}_{P^n}[\CLLR_{-\epsilon}^{\epsilon}]-\mathbb{E}_{Q^n}[\CLLR_{-\epsilon}^{\epsilon}] &= n\int (P(x)-Q(x))\CLLR_{-\epsilon}^{\epsilon}(x) dx\\
& \ge n\int (P(x)-Q(x))\CLLR_{-\epsilon'}^{\epsilon}(x) dx \\
&= \mathbb{E}_P[\CLLR_{-\epsilon'}^{\epsilon}]-\mathbb{E}_Q[\CLLR_{-\epsilon'}^{\epsilon}]\\
&\ge \Omega(n(\tau\epsilon+(1-\tau)H^2(P',Q')))
\end{align*}

The next step is to bound the second moment of $[\ln\frac{P(x)}{Q(x)}]_{-\epsilon}^{\epsilon}$ under both $P$ and $Q$. Define the following three regions:
\[A = \left\{x\;| \ln\frac{P(x)}{Q(x)}\ge -\epsilon'\right\} \;\;\; B=\left\{x\;|\ln\frac{P(x)}{Q(x)}\in[-\epsilon, -\epsilon']\right\}\;\;\; \text{and} \;\;\; C=\left\{x\;|\ln\frac{P(x)}{Q(x)}\le-\epsilon\right\},\]
so $\CLLR_{-\epsilon}^{\epsilon}(x)=\CLLR_{-\epsilon'}^{\epsilon}(x)+g(x)$ where \[g(x) = \begin{cases} 
0 & \text{if } x\in A\\
\ln\frac{P(x)}{Q(x)}+\epsilon' & \text{if } x\in B\\
-\epsilon+\epsilon' & \text{if } x\in C
\end{cases}\]
Let us first bound the expectation of $-g(x)$ under $P$. Note that since $\ln x$ is concave, $\ln x\le e^{-\epsilon'} x+\epsilon'-1$, since the right hand side is the tangent at $x=e^{\epsilon'}$.
\begin{align*}
\int_{x\in B} P(x)(-g(x))dx &= \int_{x\in B} P(x)\left(\ln\frac{Q(x)}{P(x)}-\epsilon'\right)dx \\
&\le \int_{x\in B} P(x)\left(e^{-\epsilon'}\frac{Q(x)}{P(x)}+\epsilon'-1-\epsilon'\right)dx\\
&= \int_{x\in B} e^{-\epsilon'} Q(x)-P(x) dx\\
&= e^{-\epsilon'} \int_{x\in B} (Q(x)-e^{\epsilon'} P(x)) dx\\
&\le \int_{x\in B} (Q(x)-e^{\epsilon'} P(x)) dx.
\end{align*}
Note also that since $\epsilon>\epsilon'>0$, and $\epsilon-\epsilon'\le e^{\epsilon}-e^{\epsilon'}$ so 
\begin{align*}
\int_{x\in C} P(x)(-g(x))dx &= \int_{x\in C} P(x)(\epsilon-\epsilon')dx\\
&\le \int_{x\in C} P(x)(e^{\epsilon}-e^{\epsilon'})dx\\
&\le \int_{x\in C} Q(x)-e^{\epsilon'}P(x) dx.
\end{align*}
where the last inequality follows from the fact that $e^{\eps}P(x)<Q(x)$ for $x\in C$.
Therefore, $\mathbb{E}_P[-g(x)] \le \int_{x\in B\cup C} (Q(x)-e^{\epsilon'} P(x)) dx=\tau$.
Now, \begin{align*}
\mathbb{E}_{P}\left[\left[\ln\frac{P(x)}{Q(x)}\right]_{-\epsilon}^{\epsilon}\right]^2 &= \int P(x)\left[\left[\ln\frac{P(x)}{Q(x)}\right]_{-\epsilon'}^{\epsilon}\right]^2dx+2\int P(x)g(x)\left[\ln\frac{P(x)}{Q(x)}\right]_{-\epsilon'}^{\epsilon}dx\\
&\hspace{1in}+\int P(x)g(x)^2dx.\\
\end{align*}
Now, whenever $g(x)\neq 0$, $\left[\ln\frac{P(x)}{Q(x)}\right]_{-\epsilon'}^{\epsilon}\in[-\eps, -\eps']$ so \[\int P(x)g(x)\left[\ln\frac{P(x)}{Q(x)}\right]_{-\epsilon'}^{\epsilon}dx\le \eps\int P(x)(-g(x)) dx.\] Thus,
\begin{align*}
\mathbb{E}_{P}\left[\left[\ln\frac{P(x)}{Q(x)}\right]_{-\epsilon}^{\epsilon}\right]^2 
&\le \mathbb{E}_{P}\left[\left[\ln\frac{P(x)}{Q(x)}\right]_{-\epsilon'}^{\epsilon}\right]^2+2\epsilon\int P(x)(-g(x))dx + \int P(x)g(x)^2dx\\
&\le O(\tau\epsilon+(1-\tau)H^2(P',Q'))+2\epsilon \mathbb{E}_P[-g(x)]+(\epsilon-\epsilon')\mathbb{E}_P[-g(x)]\\
&= O(\tau\epsilon+(1-\tau)H^2(P',Q'))
\end{align*}

Also on the support of $g(x)$ ($B\cup C$), $Q(x)=e^{\epsilon'}P(x)+F(x)$ where $\int_{B\cup C}F(x)=\tau$. So, \[\mathbb{E}_Q[-g(x)]=\int (e^{\epsilon'}P(x)+F(x)) (-g(x))dx \le e^{\epsilon'} \mathbb{E}_P[-g(x)]+\epsilon\int_{B\cup C} F(x)dx = O(\tau),\] since $\epsilon=O(1).$ Then, by the same argument as above, $\mathbb{E}_{Q}\left[\left[\log\frac{P(x)}{Q(x)}\right]_{-\epsilon}^{\epsilon}\right]^2=O(\tau\epsilon+(1-\tau)H^2(P',Q'))$, and so we are done.

\end{proof}

\begin{lemma}\label{threshrobust} For all constants $C_1, C_2$, there exists constants $C_3, C_4$ such that for all $\epsilon\in[0,1]$, distributions $P$ and $Q$, if $-a, b\in[C_1\epsilon, C_2\epsilon]$ then
\[\SC_{\NCLLR_a^b}(P,Q)\in[C_3\SC_{\epsilon}(P,Q), C_4\SC_{\epsilon}(P,Q)].\]
\end{lemma}

\begin{proof} Define $f(\epsilon)=\SC_{\epsilon}(P,Q)$.
Let $C$ be a constant, and let us first establish some simple bounds on $f(C\epsilon)$. Suppose $C\le 1$. Then since $f(\epsilon)$ is the optimal sample complexity, we have $f(\epsilon)\le f(C\epsilon)$. Also, by the secrecy-of-the-sample lemma for DP algorithms \citep{Balle:2020}, we have that $f(C\epsilon)\le \frac{1}{C}f(\epsilon)$. If $C\ge1$ then by the same arguments we have $f(C\epsilon)\in[\frac{1}{C}f(\epsilon), f(\epsilon)]$.

Now, note that by Lemma~\ref{getridofepsprime} and the proof of Theorem 2.5 in \cite{Canonne:2019}, there exist constants $D_1, D_2, D_3, D_4$ such that for all $\epsilon$ and $P,Q$, $\SC_{\NCLLR_{-\epsilon}^{\epsilon}}(P,Q)\in[D_1\SC_{\epsilon}(P,Q), D_2\SC_{\epsilon}(P,Q)]$ and  \[|\mathbb{E}_P^n(\CLLR_{-\epsilon}^{\epsilon}(X))-\mathbb{E}_Q^n(\CLLR_{-\epsilon}^{\epsilon}(X))|\ge D_3n\frac{1}{f(\epsilon)}\] and \[\var_P(\CLLR_{\epsilon}^{\epsilon}(X))\le D_4n\frac{1}{f(\epsilon)}.\]

Note that 
\begin{align*}
|\mathbb{E}_P^n(\CLLR_a^b(X))-\mathbb{E}_Q^n(\CLLR_a^b(X))|&\ge |\mathbb{E}_P^n(\CLLR_{-C_1\epsilon}^{C_1\epsilon}(X))-\mathbb{E}_Q^n(\CLLR_{-C_1\epsilon}^{C_1\epsilon}(X))|\\
&\ge D_3n\frac{1}{f(C_1\epsilon)}\\
&\ge D_3\min\{1, C_1\}n\frac{1}{f(\epsilon)}
\end{align*}
where the first inequality follows by definition.  Also, 
\begin{align*}
\var_P(\CLLR_a^b(X))\le \var_P(\CLLR_{-C_2\epsilon}^{C_2\epsilon}(X))\le D_4n \frac{1}{f(C_2\epsilon)}\le D_4\max\{1, C_2\}n\frac{1}{f(\epsilon)}
\end{align*}
where again the first inequality follows by definition. Therefore, there exist constants $C_3, C_4$, depending on $C_1, C_2, D_1, D_2, D_3, D_4$ such that \[\SC_{\NCLLR_a^b}(P,Q)\in[C_3f(\epsilon), C_4f(\epsilon)]\]
\end{proof}

\section{Proofs about Exponential Families}

\subsection{Proof of Corollary~\ref{monotonecontinuous}}\label{amonotonecontinuous}

\rmonotonecontinuous*

\begin{proof}[Proof of Corollary~\ref{monotonecontinuous}]
Recall from Lemma~\ref{monotonelikelihood} that one-parameter exponential families have monotone likelihood ratio so for any $\theta<\theta'$ and any $t$, $\mathbb{P}_{\theta}(x>t)<\mathbb{P}_{\theta'}(x>t)$. Suppose $\theta_1\le\theta_2\le\theta_3$ then 
\begin{align*}
\TV(P_{\theta_1}, P_{\theta_2}) &= \mathbb{P}_{\theta_2}\left(x\ge \frac{A(\theta_2)-A(\theta_1)}{\theta_2-\theta_1}\right)- \mathbb{P}_{\theta_1}\left(x\ge \frac{A(\theta_2)-A(\theta_1)}{\theta_2-\theta_1}\right)\\
&\le \mathbb{P}_{\theta_3}\left(x\ge \frac{A(\theta_2)-A(\theta_1)}{\theta_2-\theta_1}\right)- \mathbb{P}_{\theta_1}\left(x\ge \frac{A(\theta_2)-A(\theta_1)}{\theta_2-\theta_1}\right)\\
&\le \TV(P_{\theta_1}, P_{\theta_3}).
\end{align*}  
Thus, monotonicity holds. To prove continuity, note that 
\begin{align*}
\KL(P_{\theta}\|P_{\theta+h}) &= \int p_{\theta}(x)\ln\frac{p_{\theta}(x)}{p_{\theta+h}(x)} d\mu\\
&= \int p_{\theta}(x)\left(\theta x-A(\theta)-(\theta+h) x+A(\theta+h)\right) d\mu\\
&= -h\mathbb{E}_{\theta}[x]+A(\theta+h)-A(\theta)\\
&= A(\theta+h)-A(\theta)-hA'(\theta)\\
&\le \max_{\theta'\in[\theta, \theta+h]}A''(\theta')h^2.
\end{align*} 
Therefore, by Pinsker's inequality, \begin{equation}\label{upperboundonTV}
\TV(P_{\theta}, P_{\theta+h})\le |h|\sqrt{\max_{\theta'\in[\theta, \theta+h]}A''(\theta')}.
\end{equation} Therefore,
\[\TV(P_{\theta}, P_{\theta+h_1})-\TV(P_{\theta}, P_{\theta+h_2})\le \TV(P_{\theta+h_1}, P_{\theta+h_2}) \le |h_1-h_2|\sqrt{\max_{\theta'\in[\theta+h_1, \theta+h_2]}A''(\theta')}.\] Since $A''$ is continuous, there exists $\gamma$ such that if $|h_1-h_2|\le\gamma$ then $\sqrt{\max_{\theta'\in[\theta+h_1, \theta+h_2]}A''(\theta')}\le 2\sqrt{A''(\theta+h_1)}$. Thus for any $\rho\ge0$, if $|h_1-h_2|\le \min\left\{\gamma,\frac{\rho}{2\sqrt{A''(\theta+h_1)}}\right\}$ then $\TV(P_{\theta}, P_{\theta+h_1})-\TV(P_{\theta}, P_{\theta+h_2})\le\rho$. Therefore, $h\to \TV(P_{\theta}, P_{\theta+h})$ is continuous and monotone.
\end{proof}

\subsection{Proof of Lemma~\ref{concentrationexp}}\label{aconcentrationexp}

\rconcentrationexp*

\begin{proof}[Proof of Lemma~\ref{concentrationexp}] 
Recall that $\int e^{\theta x}d\mu = e^{A(\theta)}.$
Let $\lambda\le\kappa(\theta)$ then 
\begin{align*}
\mathbb{E}_{P_{\theta}}\left[e^{\lambda x}\right] &= \int e^{\lambda x}e^{\theta x-A(\theta)} d\mu\\
&= e^{-A(\theta)} \int e^{(\lambda +\theta) x} d\mu\\
&= e^{A(\theta+\lambda)-A(\theta)}.
\end{align*}
Now, \begin{align*}
\mathbb{E}_{P_{\theta}}\left[e^{|\lambda (x-A'(\theta))|}\right] &= \mathbb{E}_{P_{\theta}}\left[e^{\lambda (x-A'(\theta))}\textbf{1}_{\lambda (x-A'(\theta))\ge0}\right]+ \mathbb{E}_{P_{\theta}}\left[e^{-\lambda (x-A'(\theta))}\textbf{1}_{\lambda (x-A'(\theta))\le0}\right]\\
&\le \mathbb{E}_{P_{\theta}}\left[e^{\lambda x-\lambda A'(\theta)}\right]+ \mathbb{E}_{P_{\theta}}\left[e^{-\lambda x+\lambda A'(\theta)}\right]\\
&= e^{A(\theta+\lambda)-A(\theta)-\lambda A'(\theta)}+e^{A(\theta-\lambda)-A(\theta)+\lambda A'(\theta)}\\
&\le e^{\frac{\lambda^2}{2}\max_{\theta'\in[\theta, \theta+\lambda]}A''(\theta')}+e^{\frac{\lambda^2}{2}\max_{\theta'\in[\theta-\lambda, \theta]}A''(\theta')}\\
&\le 2e^{\lambda^2 A''(\theta)}
\end{align*}
where the last inequality follows since $\lambda\le\kappa(\theta)$. Therefore for any $u>0$,
\begin{align*}
\mathbb{P}_{P_{\theta}}(|x-A'(\theta)|\ge u)
&= \mathbb{P}_{P_{\theta}}(e^{\lambda|x-A'(\theta)|}\ge e^{\lambda u})\\
&\le \frac{\mathbb{E}_{P_{\theta}}\left[e^{|\lambda (x-A'(\theta))|}\right]}{e^{\lambda u}}\\
&\le \frac{2e^{\lambda^2 A''(\theta)}}{e^{\lambda u}}
\end{align*}
where the first inequality follows from Markov's inequality.
Let $u=2\sqrt{A''(\theta)}\sqrt{\ln(2/\beta)}+\frac{\ln(2/\beta)}{\kappa(\theta)}$ and $\lambda=\min\left\{\kappa(\theta), \frac{\sqrt{\ln(2/\beta)}}{\sqrt{A''(\theta)}}\right\}$, so
$u \ge \lambda A''(\theta)+\frac{\ln(2/\beta)}{\lambda}$ and
\begin{align*}
\mathbb{P}_{P_{\theta}}\left(|x-A'(\theta)|\ge 2\sqrt{A''(\theta)}\sqrt{\log(2/\beta)}+\frac{\ln(2/\beta)}{\kappa(\theta)}\right)
&\le \mathbb{P}_{P_{\theta}}\left(|x-A'(\theta)|\ge \lambda A''(\theta)+\frac{\ln(2/\beta)}{\lambda}\right)\\
&\le 2e^{\lambda^2A''(\theta)-\lambda^2A''(\theta)-\ln(2/\beta)}\\
&= 2e^{-\ln(2/\beta)}\\
&=\beta.
\end{align*}
The second statement follows immediately.

\end{proof}

\subsection{Initial Estimator - Proof of Theorem~\ref{initialestthm}}\label{appendix:meanest}

\rinitialestthm*

In this section we slightly generalise the algorithm and analysis from \cite{Karwa:2018} beyond Gaussian distributions. We will show that their algorithm provides accurate estimates of the mean of sufficiently nice exponential families. This algorithm first estimates the variance of the distribution, then estimates the mean. Both steps of the estimation are performed using differentially private histogram queries.

Let $\rho = \mathbb{E}_{P}[|X-\mathbb{E}_P(x)|^3]$ be the absolute third moment of $P$ and $\sigma$ be the standard deviation. Since the algorithm of \cite{Karwa:2018} is designed for Gaussian distributions we will use the following lemma that describes the rate of convergence of the central limit theorem.

\begin{lemma}[Berry-Esseen theorem]\label{BE}
Let $n\in\mathbb{N}$ and $X_1, \cdots, X_n$ be iid samples from a distribution $P$, and $\rho = \mathbb{E}_{P}[|X-\mathbb{E}_P(x)|^3]$. Set $S_n =\frac{1}{n} \sum_{j=1}^n X_j$, $\mu=\mathbb{E}_P[x]$ and $\sigma^2=\var(P)$, and let $Y\sim \mathcal{N}(\mu, \frac{\sigma^2}{n})$ then for some absolute constant $\berryesseen>0$,
\begin{itemize}
\item (Uniform) For all $a>0$, \[|\mathbb{P}[S_n\le a]-\mathbb{P}[Y\le a]|\le \frac{\berryesseen\rho}{\sigma^3\sqrt{n}}\]
\item (Non-uniform) For all $a>0$, \[|\mathbb{P}[S_n\le a]-\mathbb{P}[Y\le a]|\le \frac{\berryesseen\rho}{(1+|a|)^3\sigma^3\sqrt{n}}.\]
\end{itemize}
\end{lemma}

\begin{lemma}[Histogram Learner \cite{DworkMNS06, Bun:2016, Vadhan:2016}]\label{histlearn} For all $K\in\mathbb{N}$ and domain $\Omega$, for any collection of disjoint bins $B_1, \cdots, B_K$ defined on $\Omega, n\in\mathbb{N}, \epsilon\ge0, \delta\in(0,1/n), \lambda>0$ and $\beta\in(0,1)$ there exists an $(\epsilon,\delta)$-DP algorithm $M:\Omega^n\to\mathbb{R}^K$ such that for every distribution $D$ on $\Omega$, if 
\begin{enumerate}
\item $X_1, \cdots, X_N\sim D$ and $p_k=\mathbb{P}(X_i\in B_k)$
\item $(\tilde{p_1}, \cdots, \tilde{p_K})=M(X_1, \cdots, X_n)$ and 
\item \[n\ge \max\left\{\min\left\{\frac{8}{\epsilon\lambda}\ln\left(\frac{2K}{\beta}\right), \frac{8}{\epsilon\lambda}\ln\left(\frac{4}{\beta\delta}\right)\right\}, \frac{1}{2\lambda^2}\ln\left(\frac{4}{\beta}\right)\right\}\]
\end{enumerate}
then, \[\mathbb{P}_{X\sim D, M}(\max_k|\tilde{p_k}-p_k|\le\lambda)\ge 1-\beta\;\;\;\text{ and },\]
\[\mathbb{P}_{X\sim D, M}(\arg\max_k\tilde{p_k}=j)\le \begin{cases}
np_j+2e^{-(\epsilon n/8)\cdot(\max_kp_k)} & \text{ if } K< 2/\delta\\
np_j & \text{ if } K\ge 2/\delta
\end{cases}\]
where the probability is taken over the randomness of $M$ and the data $X_1, \cdots, X_n$.
\end{lemma}

\begin{algorithm}[ht] \caption{Variance estimator}\label{algo:varianceestimate}
\begin{algorithmic}[1]%\small
\Require{Sample $X = (x_{1},\dots,x_{n})\sim P, \epsilon, \delta, \sigma_{\min}, \sigma_{\max}, \beta, \thirdmoment$.}
\State Let $\phi= \lceil(600\berryesseen\thirdmoment)^2\rceil$, where $\berryesseen$ is the absolute constant from Lemma~\ref{BE}.
\State If \[n<c\phi\min\left\{\frac{1}{\epsilon}\ln\left(\frac{\ln\left(\frac{\sigma_{\max}}{\sigma_{\min}}\right)}{\beta}\right), \frac{1}{\epsilon}\ln\left(\frac{1}{\delta\beta}\right)\right\},\] output $\bot$, where $c$ is an absolute constant whose existence is ensured by Lemma~\ref{histlearn}.
\State Divide $[\sigma_{\min}, \sigma_{\max}]$ into bins of exponentially increasing length. The bins are of the form $B_j = (2^j, 2^{j+1}]$ for $j=j_{\min}, \cdots, j_{\max},$ where $j_{\max} = \lceil \ln_2\frac{\sigma_{\max}}{\sqrt{\phi}}\rceil+1$ and $j_{\min}=\lfloor\ln_2\frac{\sigma_{\min}}{\sqrt{\phi}}\rfloor-2.$
\State Let $Z_i = \frac{1}{\phi}\sum_{j=1}^{\phi} x_{(i-1)\phi+j}$ for $i=1, \cdots, \lfloor n/\phi\rfloor$.
\State Let $Y_i = Z_{2i}-Z_{2i-1}$ for $i=1, \cdots, \lfloor n/2\phi \rfloor$
\State Run the histogram learner of Lemma~\ref{histlearn} with privacy parameters $(\epsilon, \delta)$ and bins $B_{j_{\min}}, \cdots, B_{j_{\max}}$ on input $|Y_1|, \cdots, |Y_n|$ to obtain noisy estimates $\tilde{p_{j_{\min}}}, \cdots, \tilde{p_{j_{\max}}}$. Let \[\hat{l}=\arg\max\tilde{p_j}\]
\State Output $\hat{\sigma} = 2^{\hat{l}+2}\sqrt{\phi}$.
\end{algorithmic}
\end{algorithm}

Note in particular that the use of approximate $(\epsilon, \delta)$-DP allows us to set the $K=\infty$, while the sample complexity remains finite. 
The following lemma states that provided $\rho/\sigma^3$ is bounded, Algorithm~\ref{algo:varianceestimate} can estimate the standard deviation up to a multiplicative constant.

\begin{lemma}\label{highprobstd} For all $n\in\mathbb{N}$, $\sigma_{\min}<\sigma_{\max}\in[0, \infty], \epsilon>0, \delta\in(0,\frac{1}{n}], \beta\in(0,1/2), \thirdmoment>0,$ Algorithm~\ref{algo:varianceestimate} is $(\epsilon, \delta)$-DP and satisfies that if $X_1, \cdots, X_n$ are iid draws from $P$, where $P$ has standard deviation $\sigma\in[\sigma_{\min}, \sigma_{\max}]$ and  \textcolor{black}{$\frac{\rho}{\sigma^3}\le \thirdmoment$} then if \[n\ge c \thirdmoment^2\min\left\{\frac{1}{\epsilon}\ln\left(\frac{\ln\left(\frac{\sigma_{\max}}{\sigma_{\min}}\right)}{\beta}\right), \frac{1}{\epsilon}\ln\left(\frac{1}{\delta\beta}\right)\right\},\] (where $c$ is a universal constant), we have \[\mathbb{P}_{X\sim P, M} (\sigma\le\hat{\sigma}\le 8\sigma)\ge 1-\beta.\]
\end{lemma}

\begin{proof}[Proof of Lemma~\ref{highprobstd}]
This proof follows almost directly from Theorem 3.2 of \cite{Karwa:2018}. Note that each $Y_i$ is sampled from a distribution with mean 0 and variance $\frac{2\sigma^2}{\phi}$, and in addition is the sum of $\phi$ independent random variables. As in \cite{Karwa:2018}, there exists a bin $B_l$ with label $l\in(\lfloor \ln_2\frac{\sigma_{\min}}{\sqrt{\phi}}\rfloor-1, \lceil \ln_2\frac{\sigma_{\max}}{\sqrt{\phi}}\rceil)$ such that $\frac{\sigma}{\sqrt{\phi}}\in(2^l, 2^{l+1}]=B_l$. Define, \[p_j = \mathbb{P}(|Y_i|\in B_j).\] Sort the $p_j$'s as $p_{(1)}\ge p_{(2)}\ge\cdots$ and let $j_{(1)}, j_{(2)}, \cdots$ be the corresponding bins. Then, the following two facts imply the result (as in \cite{Karwa:2018}).

\noindent \textbf{Fact 1:} The bins corresponding to the largest and second largest mass $p_{(1)}, p_{(2)}$ are $(j_{(1)}, j_{(2)})\in\{(l,l-1), (l, l+1), (l+1,l)\}$.

\noindent \textbf{Fact 2:} $p_{(1)}-p_{(3)}>1/300$.

Now, let $W_i\sim N(0, 2\frac{\sigma^2}{\phi})$ and let $q_i, q_{(i)}$ be the corresponding probabilities for $W_i$. Then \cite{Karwa:2018} showed that: 
\begin{itemize}
\item The bins corresponding to the largest and second largest mass $q_{(1)}, q_{(2)}$ are $(j_{(1)}, j_{(2)})\in\{(l,l-1), (l, l+1), (l+1,l)\}$.
\item $q_{(1)}-q_{(3)}>1/100$.
\end{itemize}

By Lemma~\ref{BE}, since $\phi= \lceil(600\berryesseen\thirdmoment)^2\rceil$, for all $j$, $|p_j-q_j|\le 1/300$. Therefore, $\{p_{(1)}, p_{(2)}\}=\{q_{(1)}, q_{(2)}\}$, which implies both Fact 1 and Fact 2. 
\end{proof}

\begin{algorithm}[ht] \caption{Range estimator}\label{rangeestalg}
\begin{algorithmic}[1]%\small
\Require{Sample $X = (x_{1},\dots,x_{n})\sim P, \epsilon> 0, \delta\in[0,1], \beta\in (0,1/2), R\in(0,\infty), \sigma>0, B>0$.}
\State Let $r=\left\lceil \frac{R}{2\sigma}\right\rceil$. Divide $[-R-\sigma/2, R+\sigma/2]$ into $2r+1$ bins of length at most $2\sigma$ each in the following manner - bin $B_j$ equals $(2(j-0.5)\sigma, 2(j+0.5)\sigma]$, for $j\in\{-r,\cdots, r\}$. 
\State Run the histogram learner of Lemma~\ref{histlearn} with privacy parameters $(\epsilon, \delta)$ and bins $B_{-r}, \cdots, B_{r}$ on input $x_1, \cdots, x_n$ to obtain noisy estimates $\tilde{p_{-r}}, \cdots, \tilde{p_r}$. Let \[\hat{l} = \arg\max_{j=-r,\cdots, r} \tilde{p_j}.\]
\State Output $(x_{\min}, x_{\max})$, where \[x_{\min} = 2\sigma\hat{l}-\sigma(6+C)\sqrt{\ln(4n/\beta)}, \;\;\; x_{\max}=2\sigma\hat{l}+\sigma(6+C)\sqrt{\ln(4n/\beta)}.\]
\end{algorithmic}
\end{algorithm}

\begin{theorem}\label{rangeest} 
For all $n\in\mathbb{N}$, $\sigma>0, \epsilon>0, \delta\in[0,1], \beta\in(0,1/2), R\in(0,\infty]$, $C\ge 0$, Algorithm~\ref{rangeestalg} is $(\epsilon, \delta)$-DP. For all measures $\mu$ and $\theta\in\thetarange$, if $x_1, \cdots, x_n$ are sampled from $P_{\theta}$ where $A'(\theta)\in(-R,R)$, $A''(\theta)\le \sigma^2$,
$\kappa(\theta)\ge \frac{1}{C}\frac{\sqrt{\log(2/\beta)}}{\sqrt{A''(\theta)}}$,
and \[n\ge c\min\left\{\frac{1}{\epsilon}\ln\left(\frac{R}{\sigma\beta}\right), \frac{1}{\epsilon}\ln\left(\frac{1}{\delta\beta}\right)\right\},\] (where $c$ is a universal constant), then we have \[\mathbb{P}_{x\sim P_{\theta}, M}(\forall i,\;\; x_{min}\le x_i\le x_{\max})\ge 1-\beta\] and , \[|x_{\max}-x_{\min}|=\textcolor{black}{\sigma(6+C)\sqrt{\ln(4n/\beta)}}.\]
\end{theorem}

\begin{proof}
By Lemma~\ref{concentrationexp} and a union bound,  with probability $1-\beta/2$, we have \[\forall i\;:\; |x_i-\mu|\le \sigma(2+C)\sqrt{\ln(4n/\beta)}.\] Next, as in the proof of Theorem 3.1 from \cite{Karwa:2018}, we want to show that with probability $1-\beta/2$, we have \[|\mu-\hat{l}\sigma|\le 2\sigma.\] Note that by Chebyshev's inequality $\mathbb{P}[|x-\mu|\le 2\sigma]\ge 3/4$ so there exists a pair of neighbouring bins $B_j, B_{j+1}$ such that $\mathbb{P}[x\in B_j\cup B_{j+1}]\ge 3/4$ and $\mu\in B_j\cup B_{j+1}$. Also, for all $i\notin\{j,j+1\}$, $\mathbb{P}[x\in B_i]\le 1/4$. Let $j^*=\arg\max_{k=j,j+1} \mathbb{P}[x\in B_k]$. Then $\mathbb{P}[x\in B_{j^*}]\ge 3/8$, and $\mathbb{P}[x\in B_{j^*}]-\mathbb{P}[x\in B_{i}]\ge 1/8$ for all $i\notin\{j,j+1\}$. Then by Lemma~\ref{histlearn}, setting $\lambda=1/8$, $n$ is large enough that with probability $1-\beta/2$, $\hat{l}\in\{j,j+1\}$. Therefore, $|\mu-2\hat{l}\sigma|\le 4\sigma$. Therefore, with probability $1-\beta$, for all $i$, \[|x_i-2\hat{l}\sigma|\le |x_i-\mu|+|\mu-2\hat{l}\sigma|\le \sigma(2+C)\sqrt{\ln(4n/\beta)}+4\sigma \le \sigma(6+C)\sqrt{\ln(4n/\beta)}.\]
\end{proof}

As in \cite{Karwa:2018} combining these two algorithms gives us an estimator of the range with unknown variance. Since this range contains all the data points with high probability, we can clamp the data to this range, and add noise proportional to the width of the range. Note that we can remove the dependence on the range $[0,R]$ and $[\sigma_{\min}, \sigma_{\max}]$ in the sample complexity since $n\ge \frac{c\thirdmoment}{\epsilon}\ln(\frac{1}{\delta\beta})$ is sufficient to ensure that the bounds required in both Theorem~\ref{rangeest} and Lemma~\ref{highprobstd} hold.
This completes the proof of Theorem~\ref{initialestthm}.

\begin{algorithm}[ht] \caption{Initial mean estimator, $\initial$}\label{meanest}
\begin{algorithmic}[1]%\small
\Require{$X_1, \cdots, X_n, \beta, \epsilon, \delta, \thirdmoment, B$.}
\State If \[n<\frac{c\max\{\thirdmoment^2, 1\} }{\epsilon}\ln\left(\frac{1}{\delta\beta}\right),\] output 0.
\State Run Algorithm~\ref{algo:varianceestimate} to obtain an estimate $\hat{\sigma}$ of the variance with privacy parameters $(\epsilon, \delta)$, $\sigma_{\min}=0$ and $\sigma_{\max}=\infty$.
\State Run Algorithm~\ref{rangeestalg} with privacy parameters $(\epsilon, \delta)$, $R=\infty$, and standard deviation $\hat{\sigma}$ to obtain a range $[X_{\min}, X_{\max}]$.
\State Let \[Y_i = \begin{cases}
X_i & \text{if } X_i\in[X_{\min}, X_{\max}]\\
X_{\max} & \text{if } X_i>X_{\max}\\
X_{\min} & \text{if } X_i>X_{\min}
\end{cases}\]
\State Let $Z$ be a Laplace random variable with mean 0 and scale parameter $\frac{X_{\max}-X_{\min}}{\epsilon n}$.
\State Output \[\frac{\sum_{i=1}^n Y_i}{n}+Z.\]
\end{algorithmic}
\end{algorithm}

\subsection{Proof of Corollary~\ref{allforalog}}

\rallforalog*

\begin{proof}
By Theorem~\ref{nonprivupper} and Theorem~\ref{initialestthm}, there exists a constant $C>0$ such that for all $\theta_0\in\thetarange$ satisfying the two conditions and $n\ge \frac{c\thirdmoment^2\ln(1/\delta)}{\epsilon}$ with probability 0.8 we have, 
\[|\initial(X)-A'(\theta_0)| = \frac{C}{2}\left( \frac{\sqrt{A''(\theta_0)}}{\sqrt{n}}+ \frac{\sqrt{A''(\theta_0)}}{n\epsilon}\sqrt{\ln\left(n\right)}\right).\] 
Now, since $\kappa(\theta)\ge \frac{1}{C}\frac{\sqrt{\log(2/\beta)}}{\sqrt{A''(\theta)}}$
and $\epsilon=\Omega(\frac{\ln n}{n})$, there exists $N\in\mathbb{N}$ such that for all $n>N$, \[\frac{C}{2}\left(\frac{\sqrt{A''(\theta_0)}}{\sqrt{n}}+\frac{\sqrt{A''(\theta_0)}\sqrt{\ln n}}{\epsilon n}\right)\le \frac{1}{2}A''(\theta_0)\kappa(\theta_0)\] combined with Lemma~\ref{belongsinkappa} implies that $A'^{-1}(\initial(X))\in\thetarange(\theta_0)$.
Therefore, 
\begin{align*}
|(A')^{-1}(\initial(X))-\theta_0| &\le \max_{t\in[\initial(X), A'(\theta_0)]}(A')^{-1})'(t)|\initial(X)-\theta_0|\\
&= \max_{t\in[\initial(X), A'(\theta_0)]}\frac{1}{A''(A'^{-1}(t))}|\initial(X)-\theta_0|\\
&\le 2\frac{1}{A''(\theta_0)}\frac{C}{2}\left( \frac{\sqrt{A''(\theta_0)}}{\sqrt{n}}+ \frac{\sqrt{A''(\theta_0)}}{n\epsilon}\sqrt{\ln\left(n\right)}\right)\\
&= C\left( \frac{1}{\sqrt{nA''(\theta_0)}}+ \frac{1}{n\epsilon\sqrt{A''(\theta_0)}}\sqrt{\ln\left(n\right)}\right)
\end{align*}
\end{proof}

\subsection{Proof of Lemma~\ref{boundonmodulus}}\label{aboundonmodulus}

\rboundonmodulus*

\begin{proof}
\textbf{Lower Bound:}
Recall from the proof of Corollary~\ref{monotonecontinuous} that $\TV(P_{\theta}, P_{\theta+h})\le |h|\sqrt{\max_{\theta'\in[\theta, \theta+h]}A''(\theta')}$. Now, let $h$ be such that $\TV(P_{\theta}, P_{\theta+h})=\beta$, so \[|h|\ge\frac{\beta}{\sqrt{\max_{\theta'\in[\theta, \theta+h]}A''(\theta')}}.\] By assumption, 
$\kappa(\theta)\ge J^{-1}_{\TV,\theta}(\beta)\ge |h|$
 and $\max_{\theta'\in[\theta, \theta+h]}A''(\theta')\le2A''(\theta_0)$. Therefore $|h|\ge  \frac{\beta}{\sqrt{2A''(\theta)}}$ 
so $J_{\TV, \theta}^{-1}(\beta)~\ge~\frac{\beta}{\sqrt{2A''(\theta)}}$. 

\noindent\textbf{Upper Bound:}
Since $\beta<\frac{9}{128\fourthmoment}$, there exists a constant $C$ such that $\frac{16\fourthmoment}{9}\le C\le \frac{1}{8\beta}$. Let $h=\frac{C\beta}{\sqrt{A''(\theta)}}$. If $h\ge\kappa(\theta)$ then we are done since $J_{\TV, \theta}^{-1}(\beta)\le\kappa(\theta)\le h$, so assume that $h\le\kappa(\theta)$.
It suffices to prove that $\TV(P_{\theta}, P_{\theta+h})\ge\beta$ since then again by monotonicity and continuity of $h\to \TV(P_{\theta}, P_{\theta+h})$, we are done. By the Paley-Zygmund inequality, \[\mathbb{P}_{P_{\theta}}\left(X\ge A'(\theta)+\frac{1}{2}\sqrt{A''(\theta)}\right)\ge \left(1-\frac{1}{4}\right)^2\frac{(\mathbb{E}(X-\mathbb{E}(X))^2)^2}{\mathbb{E}[(X-\mathbb{E}(X))^4]}\ge \frac{9}{16\fourthmoment}.\] Then,
\begin{align*}
TV(P_{\theta}, P_{\theta+h}) &= \int_{\frac{A(\theta+h)-A(\theta)}{h}}^{\infty} \left(e^{(\theta+h)x-A(\theta+h)}-e^{\theta x-A(\theta)}\right)d\mu\\
&\ge \int_{\frac{A(\theta+h)-A(\theta)}{h}+\frac{1}{4}\sqrt{A''(\theta)}}^{\infty} \left(e^{(\theta+h)x-A(\theta+h)}-e^{\theta x-A(\theta)}\right)d\mu\\
&= \int_{\frac{A(\theta+h)-A(\theta)}{h}+\frac{1}{4}\sqrt{A''(\theta)}}^{\infty} \left(e^{hx-(A(\theta+h)-A(\theta))}-1\right)P_{\theta}(x)d\mu\\
&\ge \left(e^{\frac{1}{4}h\sqrt{A''(\theta)}}-1\right) \mathbb{P}_{P_{\theta}}\left(X\ge \frac{A(\theta+h)-A(\theta)}{h}+\frac{1}{4}\sqrt{A''(\theta)}\right).
\end{align*}
Now, 
\begin{align*}
\frac{A(\theta+h)-A(\theta)}{h}+\frac{1}{4}\sqrt{A''(\theta)}&\le A'(\theta)+h\max_{\theta'\in[\theta, \theta+h]}A''(\theta)+\frac{1}{4}\sqrt{A''(\theta)}\\
&\le A'(\theta)+\frac{1}{8\sqrt{A''(\theta)}}2A''(\theta)+\frac{1}{4}\sqrt{A''(\theta)}\\
&\le A'(\theta)+\frac{1}{2}\sqrt{A''(\theta)},
\end{align*}
where the second inequality holds since $h\ge \frac{1}{8}\frac{1}{\sqrt{A''(\theta)}}$.
Thus,
\begin{align*}
TV(P_{\theta}, P_{\theta+h})
&\ge \left(e^{\frac{1}{4}h\sqrt{A''(\theta)}}-1\right) \mathbb{P}_{P_{\theta}}\left(X\ge A'(\theta)+\frac{1}{2}\sqrt{A''(\theta)}\right)\\
&\ge \frac{1}{4}\frac{9}{4\fourthmoment} h\sqrt{A''(\theta)}\\
&\ge \beta,
\end{align*}
where the second inequality follows from the fact that $e^x-1\ge x$ for all $x>0$, and the final inequality follows from the definition of $h$, and the assumptions on $C$.
\end{proof}

\subsection{Proof of Lemma~\ref{highprivexptest}}\label{ahighprivexptest}

\rhighprivexptest*

\begin{proof}[Proof of Lemma~\ref{highprivexptest}] 
By assumption there exists constant $A_1$ and $A_2$ such that $A_1\frac{\ln n}{n}\le\epsilon\le A_2\frac{1}{\sqrt{n}}$.
Set $N=\frac{8\sqrt{2}}{\epsilon_n}\max\{\frac{1}{J_{TV,\theta_0}(\kappa(\theta_0))},  \frac{128\fourthmoment}{9}\}$ then $n\ge N$ implies
$\frac{8\sqrt{2}}{\epsilon n}\le \frac{9}{128\fourthmoment}$ and $\kappa(\theta_0)\ge J_{\TV,\theta}^{-1}(\frac{8\sqrt{2}}{\epsilon n})$.
Combined with the first assumption and  Lemma~\ref{boundonmodulus}, this implies that that there exists constants $C_1$ and $C_2$ (depending only on $\fourthmoment$) such that for all $\beta\le\frac{8\sqrt{2}}{\epsilon n}$, \[J_{\TV, \theta_0}^{-1}(\beta)\in\left[\frac{C_1\beta}{ \sqrt{A''(\theta_0)}}, \frac{C_2\beta}{\sqrt{A''(\theta_0)}}\right].\] Let $C=\frac{C_2 8\sqrt{2}}{C_1}$. Then \[|\theta_0-\theta_1|\ge CJ_{\TV, \theta_0}^{-1}\left(\frac{1}{\epsilon n}\right) \ge \frac{C_28\sqrt{2}}{C_1}\frac{C_1}{\epsilon n\sqrt{A''(\theta_0)}}\ge J_{\TV, \theta_0}^{-1}\left(\frac{8\sqrt{2}}{\epsilon n}\right).\]
Thus, $\TV(P_{\theta_0}, P_{\theta_1})\ge \frac{8\sqrt{2}}{\epsilon n}$.

Next, assume that $|\theta_0-\theta_1|=CJ_{\TV, \theta_0}^{-1}\left(\frac{1}{\epsilon_nn}\right)$. Assume without loss of generality that $\theta_0\le\theta_1$. Note that by Markov's inequality, it is sufficient to show that 
\begin{align*}
&\mathbb{E}_{\theta_1}\left[f_{\hat{t}}(X)+\Lap\left(\frac{1}{\epsilon n}\right)\right]-\mathbb{E}_{\theta_0}\left[f_{\hat{t}}(X)+\Lap\left(\frac{1}{\epsilon n}\right)\right] \\
&\hspace{1in}\ge \frac{1}{4}\min\left\{\sqrt{\var_{\theta_0}\left(f_{\hat{t}}(X)+\Lap\left(\frac{1}{\epsilon n}\right)\right)}, \sqrt{\var_{\theta_1}\left(f_{\hat{t}}(X)+\Lap\left(\frac{1}{\epsilon n}\right)\right)}\right\}
\end{align*}
Let us first analyze the gap in expectations in the test statistic $f_{\hat{t}}(\cdot)+\Lap(\frac{1}{\epsilon n})$.
Note, \[\mathbb{P}_{\theta_1}\left[T(X)>\frac{A(\theta_0)-A(\theta_1)}{\theta_0-\theta_1}\right]-\mathbb{P}_{\theta_0}\left[T(X)>\frac{A(\theta_0)-A(\theta_1)}{\theta_0-\theta_1}\right] = TV(P_{\theta_0}, P_{\theta_1}).\] Therefore, 
\begin{align*}
&\mathbb{E}_{\theta_1}\left[f_{\hat{t}}(X)+\Lap\left(\frac{1}{\epsilon n}\right)\right]-\mathbb{E}_{\theta_0}\left[f_{\hat{t}}(X)+\Lap\left(\frac{1}{\epsilon n}\right)\right]\\
&\hspace{1in}= TV(P_{\theta_0}, P_{\theta_1})+\mathbb{P}_{\theta_1}[T(x)\in I]-\mathbb{P}_{\theta_0}[T(x)\in I],
\end{align*} where $I$ has endpoints $\hat{t}$ and $\frac{A(\theta_0)-A(\theta_1)}{\theta_0-\theta_1}$. Assume, for ease of notation, that $\hat{t}\ge \frac{A(\theta_0)-A(\theta_1)}{\theta_0-\theta_1}$ so $\hat{t}=\frac{A(\theta_0)-A(\theta_1)}{\theta_0-\theta_1}+\Gamma$ where 
\begin{align*}
\Gamma &= \hat{t}-\frac{A(\theta_0)-A(\theta_1)}{\theta_0-\theta_1}\\
&= \hat{t}-A'(\theta_0)+A'(\theta_0)-\frac{A(\theta_0)-A(\theta_1)}{\theta_0-\theta_1}\\
&\le \textcolor{black}{b\left(\sqrt{\frac{A''(\theta_0)}{n}}+\frac{\sqrt{A''(\theta_0)\ln\left(n\right)}}{\epsilon n}\right)}+\max_{\theta'\in[\theta_0, \theta_1]}A''(\theta')|\theta_0-\theta_1|.
\end{align*}
Now,  
\begin{align*}
\mathbb{P}_{\theta_1}[T(x)\in I]-\mathbb{P}_{\theta_0}[T(x)\in I] &= \int_{a\in I} e^{a\theta_1-A(\theta_1)} d\mu - \int_{a\in I} e^{a\theta_0-A(\theta_0)} d\mu\\
&= \int_{a\in I} e^{a(\theta_1-\theta_0)+A(\theta_0)-A(\theta_1)}e^{a\theta_0-A(\theta_0)} d\mu - \int_{a\in I} e^{a\theta_0-A(\theta_0)} d\mu\\
&\le \max_{a\in I}(e^{a(\theta_1-\theta_0)+A(\theta_0)-A(\theta_1)}-1)\mathbb{P}_{\theta_0}[T(X)\in I]\\
&\le \max_{a\in I}(e^{a(\theta_1-\theta_0)+A(\theta_0)-A(\theta_1)}-1)\\
&= e^{\hat{t}(\theta_1-\theta_0)+A(\theta_0)-A(\theta_1)}-1\\
&= e^{\Gamma(\theta_1-\theta_0)}-1. 
\end{align*}
Now $|\theta_0-\theta_1|=CJ_{\TV, \theta_0}^{-1}\left(\frac{1}{\epsilon_nn}\right)\le\frac{C_2C}{\epsilon n \sqrt{A''(\theta_0)}}$ and thus there exists a constant $C_3$ (depending on $\fourthmoment, C_1, C_2, A_1$ and $A_2$) such that
\begin{align*}
\mathbb{P}_{\theta_0}[T(x)\in I]-\mathbb{P}_{\theta_1}[T(x)\in I] 
&\le e^{\left[\textcolor{black}{b\left(\sqrt{\frac{A''(\theta_0)}{n}}+\frac{\sqrt{A''(\theta_0)\ln\left(n\right)}}{\epsilon n}\right)}+\max_{\theta'\in[\theta_0, \theta_1]}A''(\theta')|\theta_0-\theta_1|\right](\theta_1-\theta_0)}-1\\
&\le e^{\frac{bC_2C}{\epsilon n^{1.5}}+\frac{bC_2C\sqrt{\ln n}}{\epsilon^2 n^2}+\frac{2C_2^2C^2}{\epsilon^2 n^2}}-1\\
&= e^{\frac{C_3\sqrt{\ln n}}{\epsilon^2 n^{2}}}-1
\end{align*}
where the second inequality follows from $\epsilon=O(1/\sqrt{n})$, which implies that $\epsilon n^{1.5}=\Omega(\epsilon^2n^2)$. Now $\epsilon=\Omega(\ln n/n)$ implies $e^{\frac{C_3\sqrt{\ln n}}{\epsilon^2 n^{2}}}-1=o(\frac{1}{\epsilon n})$, thus since $\TV(P_{\theta_0}, P_{\theta_1})\ge \frac{8\sqrt{2}}{\epsilon n}$, there exists $N$ (depending on $C, B, A_1$ and $A_2$) such that if $n\ge \max\{N, \frac{8\sqrt{2}}{\epsilon_nJ_{TV,\theta_0}(\kappa(\theta_0))},  \frac{8\sqrt{2}}{\epsilon_n}\frac{128}{9\fourthmoment}\}$ then \[\mathbb{E}_{\theta_1}\left[f_{\hat{t}}(X)+\Lap\left(\frac{1}{\epsilon n}\right)\right]-\mathbb{E}_{\theta_0}\left[f_{\hat{t}}(X)+\Lap\left(\frac{1}{\epsilon n}\right)\right] \ge \frac{1}{2}\left(\TV(P_{\theta_0}, P_{\theta_1}\right))\]
Also, \[\min\left\{\var_{\theta_0}\left(f_{\hat{t}}(X)+\Lap\left(\frac{1}{\epsilon n}\right)\right), \var_{\theta_1}\left(f_{\hat{t}}(X)+\Lap\left(\frac{1}{\epsilon n}\right)\right)\right\}\le \frac{1}{n}+\frac{1}{\epsilon^2n^2}\le 2\left(\frac{1}{\epsilon^2n^2}\right),\] where the second inequality holds since $\epsilon\le1/\sqrt{n}$. 
Thus, since $\TV(P_{\theta_0}, P_{\theta_1})]\ge\frac{8\sqrt{2}}{\epsilon n}$, we have \[\sqrt{\var_{\theta_0}\left(f_{\hat{t}}(X)+\Lap\left(\frac{1}{\epsilon n}\right)\right)}\le \frac{1}{4}\left( \mathbb{E}_{\theta_0}\left[f_{\hat{t}}(X)+\Lap\left(\frac{1}{\epsilon n}\right)\right]-\mathbb{E}_{\theta}\left[f_{\hat{t}}(X)+\Lap\left(\frac{1}{\epsilon n}\right)\right]\right)\] Thus, the test distinguishes between $P_{\theta_0}$ and $P_{\theta_1}$.

Next, assume that $|\theta_0-\theta_1| \ge CJ_{\TV, \theta_0}^{-1}\left(\frac{1}{\epsilon_nn}\right)$. Let $\theta_1'$ be such that $|\theta_0-\theta_1'|= CJ_{\TV, \theta_0}^{-1}\left(\frac{1}{\epsilon_nn}\right)$ and let $n>\max\{N, \frac{8\sqrt{2}}{\epsilon_nJ_{TV,\theta_0}(\kappa(\theta_0))},  \frac{8\sqrt{2}}{\epsilon_n}\frac{128}{9\fourthmoment}\}$. Then by Lemma~\ref{monotonelikelihood},
\begin{align*}
&\mathbb{E}_{\theta_1}\Big[f_{\hat{t}}(X)+\Lap\left(\frac{1}{\epsilon n}\right)\Big]-\mathbb{E}_{\theta_0}\Big[f_{\hat{t}}(X)+\Lap\left(\frac{1}{\epsilon n}\right)\Big]\\
&\hspace{1in}\ge \mathbb{E}_{\theta_1'}\Big[f_{\hat{t}}(X)+\Lap\left(\frac{1}{\epsilon n}\right)\Big]-\mathbb{E}_{\theta_0}\Big[f_{\hat{t}}(X)+\Lap\left(\frac{1}{\epsilon n}\right)\Big]\\
&\hspace{1in}\ge \frac{1}{2}\TV\left(P_{\theta_0}, P_{\theta_1'}\right)\\
\end{align*}
and as above \[\min\left\{\var_{\theta_0}\left(f_{\hat{t}}(X)+\Lap\left(\frac{1}{\epsilon n}\right)\right), \var_{\theta_1}\left(f_{\hat{t}}(X)+\Lap\left(\frac{1}{\epsilon n}\right)\right)\right\}\le \frac{1}{n}+\frac{1}{\epsilon^2n^2}\] so we are done.
\end{proof}

\subsection{Proof of Proposition~\ref{lowprivmodcontinuity}}\label{alowprivmodcontinuity}

\rlowprivmodcontinuity*

\begin{lemma}\label{lownopriv}
For all constants $\fourthmoment>0$,  there exists constants $k, C, N>0 $ such that for all $\theta_0\in\thetarange$ if
\begin{itemize}
\item  $\epsn\ge\frac{k}{\sqrt{n}}$ 
\item  $\frac{\mathbb{E}_{\theta}((X-A'(\theta))^4)}{A''(\theta)^2}\le \fourthmoment$ 
\end{itemize}
then for all $h\ge \frac{C}{\sqrt{nA''(\theta)}}$ and $n$ such that $\kappa(\theta)\ge \frac{C}{\sqrt{n A''(\theta)}}$ and $n\ge N$, we have \[\SC_{\epsn}(P_{\theta}, P_{\theta+h}) \le n.\]
\end{lemma}

\begin{proof} Let $C$ and $N$ be constants to be specified later with the relationship that $\frac{C}{\sqrt{N}}\le \frac{1}{8}$. Let $n\ge \max\{N, \frac{C^2}{A''(\theta)(\kappa(\theta))^2}\}$.
First, assume that $h=\frac{C}{\sqrt{nA''(\theta)}}$. 
Let \[\CLLR(X) = \sum_{i=1}^{n} \left[\ln\frac{P_{\theta+h}(x_i)}{P_{\theta}(x_i)}\right]_{-\epsilon}^{\epsilon}\] and \[\NCLLR_{\epsilon}(X) = \CLLR(X) + \Lap(2).\]
Our goal is to show that under the conditions outlined in the lemma statement,  
\[\sqrt{\var_{\theta}(\NCLLR)}\le (1/4)|
\mathbb{E}_{\theta+h}(\NCLLR)-\mathbb{E}_{\theta}(\NCLLR)|\] and \[\sqrt{\var_{\theta+h}(\NCLLR)}\le (1/4)|
\mathbb{E}_{\theta+h}(\NCLLR)-\mathbb{E}_{\theta}(\NCLLR)|.\]

Firstly, note that since for any $x$ and $y$, \[\left|\left[\ln\frac{P_{\theta+h}(x)}{P_{\theta}(x)}\right]_{-\epsilon}^{\epsilon}-\left[\ln\frac{P_{\theta+h}(y)}{P_{\theta}(y)}\right]_{-\epsilon}^{\epsilon}\right|\le \left|\ln\frac{P_{\theta+h}(x)}{P_{\theta}(x)}-\ln\frac{P_{\theta+h}(y)}{P_{\theta}(y)}\right|.\] 
Further, $\var_{\theta}(LLR) \le 18 n H^2(P_{\theta}, P_{\theta+h})\le 18n\TV(P_{\theta}, P_{\theta+h})\le 36nh^2A''(\theta),$ where the first inequality follows from a simple manipulation (see the proof of Theorem 2.5 in \cite{Canonne:2019} for details), the second is a standard inequality, and the third follows from eqn~\eqref{upperboundonTV}.
Thus, we have that $\var_{\theta}(\NCLLR)\le \var_{\theta}(\LLR)+2=36nh^2A''(\theta)+2$ and $\var_{\theta+h}(\NCLLR)\le \var_{\theta+h}(\LLR)+2=nh^2A''(\theta+h)+2\le 72nh^2A''(\theta)+2$.
Now, also as in the proof of Lemma~\ref{boundonmodulus}, since \textcolor{black}{$|h|\le\frac{1}{8\sqrt{A''(\theta)}}$} and $|h|\le\kappa(\theta_0)$, we have 
\begin{align*}
\mathbb{P}_{P_{\theta}}\left(x\ge \frac{A(\theta+h)-A(\theta)}{h}+\frac{1}{4}\sqrt{A''(\theta)}\right)&\ge \mathbb{P}_{P_{\theta}}\left(x\ge A'(\theta)+h\max_{\theta'\in[\theta, \theta+h]}A''(\theta')+\frac{1}{4}\sqrt{A''(\theta)}\right)\\
&\ge \mathbb{P}_{P_{\theta}}\left(x\ge A'(\theta)+\frac{1}{2}\sqrt{A''(\theta)}\right)\\
&\ge \frac{9}{16\fourthmoment}
\end{align*}
Let $C_2=\sqrt{\frac{32\fourthmoment}{9}}+\frac{1}{16}$, by Chebyshev's's inequality we have 
\begin{align*}
\mathbb{P}_{P_{\theta}}\left(x\ge \frac{A(\theta+h)-A(\theta)}{h}+C_2\sqrt{A''(\theta)}\right)&\le \mathbb{P}_{P_{\theta}}\left(x\ge A'(\theta)-h\max_{\theta'\in[\theta, \theta+h]}A''(\theta)+C_2\sqrt{A''(\theta)}\right)\\
&\le  \mathbb{P}_{P_{\theta}}\left(x\ge A'(\theta)+\left(C_2-\frac{1}{16}\right)\sqrt{A''(\theta)}\right)\\
&\le \frac{1}{(C_2-\frac{1}{16})^2}\\
&\le \frac{9}{32\fourthmoment}
\end{align*}
where the second inequality follows since $h\le \frac{1}{8\sqrt{A''(\theta)}}$. Therefore, \[\mathbb{P}_{P_{\theta}}\left( \frac{A(\theta+h)-A(\theta)}{h}+\frac{1}{4}\sqrt{A''(\theta)}\le x\le \frac{A(\theta+h)-A(\theta)}{h}+C_2\sqrt{A''(\theta)}\right)\ge \frac{9}{32\fourthmoment}.\]
Now, 
\begin{align*}
&\frac{1}{n}(\mathbb{E}_{\theta+h}(\NCLLR)-\mathbb{E}_{\theta}(\NCLLR))= \int_{-\infty}^{\infty}\left[\ln\frac{P_{\theta+h}(x_i)}{P_{\theta}(x_i)}\right]_{-\epsilon}^{\epsilon} (P_{\theta+h}(x)-P_{\theta}(x))d\mu(x)\\
&\hspace{1in}= \int_{-\infty}^{\infty}\left[hx-A(\theta+h)+A(\theta)\right]_{-\epsilon}^{\epsilon}(P_{\theta+h}(x)-P_{\theta}(x)) d\mu(x)\\
&\hspace{1in}\ge \int_{\frac{A(\theta+h)-A(\theta)}{h}+\frac{1}{4}\sqrt{A''(\theta)}}^{\frac{A(\theta+h)-A(\theta)}{h}+C_2\sqrt{A''(\theta)}} \left[hx-A(\theta+h)+A(\theta)\right]_{-\epsilon}^{\epsilon}(P_{\theta+h}(x)-P_{\theta}(x)) d\mu(x)\\
\end{align*}
Now, if $x\in\left[{\frac{A(\theta+h)-A(\theta)}{h}+\frac{1}{4}\sqrt{A''(\theta)}}, {\frac{A(\theta+h)-A(\theta)}{h}+C_2\sqrt{A''(\theta)}}\right]$ then \[hx-A(\theta+h)+A(\theta)\in \left[\frac{1}{4}h\sqrt{A''(\theta)}, hC_2\sqrt{A''(\theta)}\right]\subset\left[0, \frac{C_2C}{\sqrt{n}}\right],\]
where the subset relationship holds since $h=\frac{C}{\sqrt{nA''(\theta)}}$.
Thus, letting $k=CC_2$, if \textcolor{black}{$\epsilon\ge \frac{k}{\sqrt{n}}$} then the truncation has no impact inside this region and
\begin{align*}
&\frac{1}{n}(\mathbb{E}_{\theta+h}(\NCLLR)-\mathbb{E}_{\theta}(\NCLLR)) \\
&\hspace{0.1in}\ge \int_{\frac{A(\theta+h)-A(\theta)}{h}+\frac{1}{4}\sqrt{A''(\theta)}}^{\frac{A(\theta+h)-A(\theta)}{h}+C_2\sqrt{A''(\theta)}} \left(hx-A(\theta+h)+A(\theta)\right)(P_{\theta+h}(x)-P_{\theta}(x)) d\mu(x)\\
&\hspace{0.1in}= \int_{\frac{A(\theta+h)-A(\theta)}{h}+\frac{1}{4}\sqrt{A''(\theta)}}^{\frac{A(\theta+h)-A(\theta)}{h}+C_2\sqrt{A''(\theta)}} \left(hx-A(\theta+h)+A(\theta)\right)P_{\theta}(x)\left(e^{hx-A(\theta+h)+A(\theta)}-1\right) d\mu(x)\\
&\hspace{0.1in}\ge  \frac{1}{4}h\sqrt{A''(\theta)}\left(e^{ \frac{1}{4}h\sqrt{A''(\theta)}}-1\right) \mathbb{P}_{P_{\theta}}\left( \frac{1}{4}\sqrt{A''(\theta)}\le x- \frac{A(\theta+h)-A(\theta)}{h}\le C_2\sqrt{A''(\theta)}\right)\\
&\hspace{0.1in}\ge \frac{1}{16}\frac{9}{32\fourthmoment} h^2A''(\theta).
\end{align*}
Therefore, recalling that $h=\frac{C}{\sqrt{n A''(\theta)}}$ and $\max\{\var_{\theta}(\NCLLR), \var_{\theta+h}(\NCLLR)\}\le 72nh^2A''(\theta)$, there exists a large enough $C$ (where large enuogh depends only on $\zeta$) such that 
 $\max\left\{\sqrt{\var_{\theta}(\NCLLR)}, \sqrt{\var_{\theta+h}(\NCLLR)}\right\}\le(1/4)(\mathbb{E}_{\theta+h}(\NCLLR)-\mathbb{E}_{\theta}(\NCLLR))$, as required.

Next, assume that $h'\ge \frac{C}{\sqrt{nA''(\theta)}}$. Note that $ \left[\ln\frac{P_{\theta+h}(x)}{P_{\theta}(x)}\right]_{-\epsilon}^{\epsilon}$ is monotone increasing in $x$, so for any threshold $\tau$, $\mathbb{P}_{X\sim P_{\theta+h'}^n}(\NCLLR(X)\ge\tau)\ge \mathbb{P}_{X\sim P_{\theta+h}^n}(\NCLLR(X)\ge\tau)$. Therefore, if $\NCLLR$ distinguishes $\theta$ and $\theta+h$ with $n$ samples then it also distinguishes $\theta$ and $\theta+h'$ with $n$ samples. 
\end{proof}

\begin{proof}[Proof of Proposition~\ref{lowprivmodcontinuity}]
First note that by definition $\omega_n(P_{\theta}, \Qtesteps)\ge\omega_n(P_{\theta}, \Qtest)$ and $\omega_n(P_{\theta}, \Qtest)=\Theta\left(\frac{1}{\sqrt{nA''(\theta)}}\right)$. By Lemma~\ref{lownopriv}, there exists $k,C_2\ge0$ such that if $n$ is sufficiently large, $\epsilon\ge\frac{k}{\sqrt{n}}$ and $h\ge \frac{C_2}{\sqrt{nA''(\theta)}}$ then $\SC_{\epsilon}(P_{\theta+h}, P_{\theta})\le n$. Therefore, $\omega_n(P_{\theta}, \Qtesteps)\le \frac{C_2}{\sqrt{nA''(\theta)}}.$
\end{proof}
%%%%%%%%%%%%%%%%%

\subsection{Proof of Proposition~\ref{mainexp}}

\rmainexp*

\begin{proof} 

The proposition follows from a combination of Corollary~\ref{maincorlowpriv} and Corollary~\ref{allforalog}. Let $N, k , C_1, D_1$ and $D_2$ be as in Corollary~\ref{maincorlowpriv}. Note that since $\frac{D}{A''(\theta_0)(\kappa(\theta_0))^2}\le DC^2=O(1)$ we can set $N$ large enough that $N\ge \frac{D_2}{A''(\theta_0)(\kappa(\theta_0))^2}$. By Corollary~\ref{allforalog}, there exists a constant $C_2$ such that
\[|A'^{-1}(\initial(X))-\theta_0|\le C_2\left(\frac{1}{\sqrt{nA''(\theta_0)}}+\frac{1}{n\epsilon\sqrt{A''(\theta_0)}}\sqrt{\ln(n)}\right).\]
Again since $\kappa(\theta_0)\sqrt{A''(\theta_0)}\ge \frac{1}{C}$, there exists constants $N_1$ and $D$ such that for all $n>N_1$ if $\frac{\sqrt{\ln n}}{\sqrt{n}}\le D A''(\theta_0)$ then, \[C_2\left(\frac{1}{\sqrt{nA''(\theta_0)}}+\frac{1}{n\epsilon\sqrt{A''(\theta_0)}}\sqrt{\ln(n)}\right)\le \min\{\kappa(\theta_0), D_1\epsilon\sqrt{nA''(\theta_0)}.\] The condition $\frac{\sqrt{\ln n}}{\sqrt{n}}\le D A''(\theta_0)$ is implied by the assumption that $n\ge \frac{2\ln(1/(DA''(\theta_0))^2)}{(DA''(\theta_0))^2}$.
Thus, the estimator from Corollary~\ref{allforalog} satisfies the requirements of Corollary~\ref{maincorlowpriv} and so we are done.
\end{proof}

\subsection{Proof of Lemma~\ref{mainlemmaexp}}\label{amainlemmaexp}

\rmainlemmaexp*

\begin{proof}[Proof of Lemma~\ref{mainlemmaexp}] Let $\tilde{\theta} = A'^{-1}(\hat{t})$ and recall that for ease of notation we let $\alpha_n(\theta) = \frac{1}{\sqrt{nA''(\theta)}}$. 
Since $|\theta_0-A'^{-1}(\hat{t})|\le\kappa(\theta_0)$, we have $\alpha_n(\theta_0)\in[\frac{1}{2}\tilde{\alpha}, 2\tilde{\alpha}]$. Thus by Proposition~\ref{lowprivmodcontinuity}, there exists constants $k$, $C_1$, and $C_2$ (depending only on $\fourthmoment$) such that $\omega_n(P_{\theta}, \Qtesteps)\in[C_1\alpha_n(\theta_0), C_2\alpha_n(\theta_0)]$. Further, by Lemma~\ref{threshrobust}, there exist a constant $C_3$ (depending on $C_1$ and $C_2$) such that if $a, b\in[C_1/4\epsilon, 4C_2\epsilon]$ then \[\SC_{\NCLLR_a^b}(P,Q)\le C_3\cdot\SC_{\epsilon}(P,Q).\] Now, set $C=\sqrt{C_3}\frac{C_2}{C_1}$ and assume first that $|\theta_0-\theta_1|=C\cdot\omega_n(P_{\theta}, \Qtesteps)$. Then 
\begin{align*}
|\theta_0-\theta_1|&= C\omega_n(P_{\theta}, \Qtesteps)\\
&\ge\frac{CC_1}{\sqrt{nA''(\theta_0)}}\\
&=\frac{C_2}{\sqrt{(C_2/C_1C)^2nA''(\theta_0)}}\\
&\ge \omega_{(C_2/C_1C)^2n, \SC_{\epsilon}}(P_{\theta_0}).
\end{align*}
Therefore, $\SC_{\epsilon}(P_{\theta_0}, P_{\theta_1})\le (\frac{C_2}{C_1C})^2n$. Thus, if $a,b\in[C_1/4, 4C_2]$ then \begin{equation}\label{equalSC}
\SC_{\NCLLR_a^b}(P_{\theta_0}, P_{\theta_1})\le C_3\left(\frac{C_2}{C_1C}\right)^2n\le n.
\end{equation}

Now, recall that $P_{\theta}(x)=e^{\theta x-A(\theta)}$ and $|\theta_0-\theta_1|=C\cdot \omega_n(P_{\theta}, \Qtesteps)$ so
\begin{align*}
(\theta_1-\theta_0)[x-\hat{t}]_{-\epsilon/C\tilde{\alpha}}^{\epsilon/C\tilde{\alpha}} &= \left[(\theta_1-\theta_0)(x-\hat{t})\right]_{-\epsilon\frac{ \omega_{n,\SC_{\epsilon}}(\theta_0)}{\tilde{\alpha}}}^{\epsilon\frac{ \omega_{n,\SC_{\epsilon}}(\theta_0)}{\tilde{\alpha}}}\\
&=\left[\ln\frac{P_{\theta_1}(x)}{P_{\theta_0}(x)}+(A(\theta_1)-A(\theta_0))- (\theta_1-\theta_0)\hat{t}\right]_{-\epsilon\frac{ \omega_{n,\SC_{\epsilon}}(\theta_0)}{\tilde{\alpha}}}^{\epsilon\frac{ \omega_{n,\SC_{\epsilon}}(\theta_0)}{\tilde{\alpha}}}\\
&= \left[\ln\frac{P_{\theta_1}(x)}{P_{\theta_0}(x)}\right]_{\epsilon\frac{ \omega_{n,\SC_{\epsilon}}(\theta_0)}{\tilde{\alpha}}-\Gamma}^{\epsilon\frac{ \omega_{n,\SC_{\epsilon}}(\theta_0)}{\tilde{\alpha}}-\Gamma}+\Gamma
\end{align*}
where 
\begin{align*}
\Gamma &= A(\theta_1)-A(\theta_0)-(\theta_1-\theta_0)\hat{t} \\
&\le A(\theta_1)-A(\theta_0)-(\theta_1-\theta_0)A'(\theta_0)+(\theta_1-\theta_0)(A'(\theta_0)-\hat{t})\\
&\le \left(\max_{\theta'\in[\theta_1, \theta_0]}A''(\theta')\right)(\theta_1-\theta_0)^2+|\theta_1-\theta_0||A'(\theta_0)-\hat{t}|\\
\end{align*}
Now, if we let $D_2=C_2^2C^2$ then $n\ge \frac{D_2}{A''(\theta_0)(\kappa(\theta_0))^2}$ implies that \[|\theta_0-\theta_1|=C\omega_{n, \SC_{\epsilon}}(\theta_0)\le C\frac{C_2}{\sqrt{n A''(\theta_0)}}= \frac{\sqrt{D_2}}{\sqrt{nA''(\theta_0)}}\le\kappa(\theta_0).\] 
Therefore, \[\Gamma\le 2A''(\theta_0)(\theta_1-\theta_0)^2+|\theta_1-\theta_0||A'(\theta_0)-\hat{t}|\le 2C^2C_2^2\frac{1}{n}+C_2C\alpha_n(\theta_0)|A'(\theta_0)-\hat{t}|.\] 
Thus, noting that $1/n\ll\epsilon$, and setting $b=\frac{C_1}{8CC_2}$ then if $|\theta_0-A'^{-1}(\hat{t})|\le \frac{b\epsilon}{\alpha_n(\theta_0)}=b\epsilon\sqrt{nA''(\theta_0)}$ then there exists $N$ such that for all $n\ge \max\{N, \frac{\nB}{A''(\theta_0)(\kappa(\theta_0))^2}\}$, $\Gamma\le \frac{C_1\epsilon}{4}$.
Therefore, for all $n\ge N$,  the truncation parameters $\epsilon\frac{ \omega_{n,\SC_{\epsilon}}(\theta_0)}{\tilde{\alpha}}-\beta, \epsilon\frac{ \omega_{n,\SC_{\epsilon}}(\theta_0)}{\tilde{\alpha}}+\beta\in[C_1/4, 4C_2]$. 
So, by eqn~\eqref{equalSC}, $\SC_{\NCLLR_{-a}^b}(P_{\theta_0}, P_{\theta_1})\le n$ which implies $n$ samples are sufficient for the test statistic $[x-\hat{t}]_{-\epsilon/C\tilde{\alpha}}^{\epsilon/C\tilde{\alpha}}$ to distinguish between $P_{\theta_0}$ and $P_{\theta_1}$. The threshold $\tau$ can be chosen as the midpoint between $\mathbb{E}_{X\sim P_{\theta_0}^n}[\hat{f}_{\tilde{\alpha}}(X)]$ and $\mathbb{E}_{X\sim P_{\theta_1}^n}[\hat{f}_{\tilde{\alpha}}(X)]$.

Now, assume that $|\theta_0-\theta_1|\ge C\cdot\omega_n(P_{\theta}, \Qtesteps)$. Let $\theta_1'$ be such that $\theta_0<\theta_1'<\theta_1$ and $|\theta_0-\theta_1|=C\cdot\omega_n(P_{\theta}, \Qtesteps)$. Then, by the previous argument, there exists a threshold $\tau$ such that $n$ samples are sufficient for the test statistic $[x-\tilde{t}]_{-\epsilon/C\tilde{\alpha}}^{\epsilon/C\tilde{\alpha}}$ to distinguish between $P_{\theta_0}$ and $P_{\theta_1'}$. Noting that this test statistic is monotone in $x$, we have by Lemma~\ref{monotonelikelihood} (the fact that $P_{\theta_1}$ stochastically dominates $P_{\theta_1'}$) that this test statistic also distinguishes between $P_{\theta_0}$ and $P_{\theta_1}$ with $n$ samples. Additionally, since $\mathbb{E}_{\theta_1'}[\NCLLR(X)]\ge\mathbb{E}_{\theta_1}[\NCLLR(X)]\ge \mathbb{E}_{\theta_0}[\NCLLR(X)] $, we maintain that $|\mathbb{E}_{X\sim P_{\theta_0}^n}[\hat{f}_{\tilde{\alpha}}(X)]-\tau|\le |\mathbb{E}_{X\sim P_{\theta_1}^n}[\hat{f}_{\tilde{\alpha}}(X)]-\tau|$.
\end{proof}

\section{Proofs for Section~\ref{tailrates}}

\subsection{Proof of Theorem~\ref{maintailtheorem}}

\rmaintailtheorem*

\begin{proof}[Proof of Theorem~\ref{maintailtheorem}] We can think of Algorithm~\ref{algo:BS} as at each step dividing the distance between $t_{\min}$ and $t_{\max}$ by $2/3$ and concluding that the true value $t^*$ lies between $t_{\min}$ and $t_{\max}$. Thus, in order to show that $\ierror{\hat{\theta}}{N}{f}\le \MOCallt{n}{t}{\Qtest_{\epsilon}}{\mathcal{P}}{\theta}$, it suffices to show that it is possible to run for $k^*(n)=\left\lceil\log_{\frac{3}{2}}\left  (\frac{|t_1-t_0|}{\MOCallt{n}{t}{\Qtest_{\epsilon}}{\mathcal{P}}{\theta}}\right)\right\rceil$ iterations with $N=n\cdot \lceil \log k^*(n)\rceil \cdot k^*(n)$ samples. In order to make the correct decision with probability $1/3k^*(n)$, it suffices for the last iteration to use $n\cdot \lceil \log k^*(n)\rceil$ samples. Since the hypothesis test at the last step has the largest sample size, $n\cdot\lceil \log k^*(n)\rceil\cdot k^*(n)$ samples is sufficient to run $k^*(n)$ rounds.
\end{proof}

%%%%%%%%%%%%%%%%%

\end{document}